\documentclass[12pt]{article}
\usepackage{amsmath}
\usepackage{amsfonts}
\usepackage{amssymb}
\usepackage{amscd}
\usepackage{amsthm}
\usepackage{bbm}
\usepackage{color}      
\usepackage[linktocpage=true,plainpages=false,pdfpagelabels=false]{hyperref}

\input xypic

\usepackage[matrix,arrow,curve]{xy}

\newcommand{\F}{{\mathbb{F}}}

\newcommand{\Z}{{\mathbb Z}}

\newcommand{\N}{{\mathbb N}}
\newcommand{\Q}{{\mathbb Q}}

\newcommand{\C}{{\mathbb C}}

\newcommand{\Hc}{{\mathcal H}}

\newcommand{\Ac}{{\mathcal A}}

\newcommand{\res}{\mathop{\rm res}}       

\newcommand{\Pc}{{\mathcal P}}
\newcommand{\Cc}{{\mathcal C}}

\newcommand{\Vc}{{\mathcal V}}

\newcommand{\LL}{\mathcal L}

\newcommand{\Ab}{{\mathbb A}}

\newcommand{\OO}{{\mathcal O}}

\newcommand{\Nil}{{\rm Nil}}

\newcommand{\Hom}{{\rm Hom}}

\newcommand{\Gal}{{\rm Gal}}

\newcommand{\Sym}{{\rm Sym}}

\newcommand{\Ker}{{\rm Ker}}

\newcommand{\rk}{{\rm rk}}

\newcommand{\Spec}{{\rm Spec} }

\newcommand{\Tor}{{\rm Tor}}

\newcommand{\sgn}{{\rm sgn}}
\newcommand{\Li}{{\rm Li}}

\newcommand{\gm}{\mathbb G_{m}}
\newcommand{\ga}{\mathbb G_{a}}

\newcommand{\z}{{\mathbb{Z}}}

\newcommand{\g}{{\mathbb{G}}}
\newcommand{\lrto}{\longrightarrow}

\newcommand{\lo}{L}
\newcommand{\uz}{\underline{\mathbb{Z}}}

\newcommand{\vv}{{\mathbb{V}}}

\theoremstyle{plain}
\newtheorem{theor}{Theorem}[section]
\newtheorem{prop}[theor]{Proposition}

\newtheorem{corol}[theor]{Corollary}
\newtheorem{lemma}[theor]{Lemma}

\theoremstyle{remark}
\newtheorem{rmk}[theor]{Remark}
\newtheorem{examp}[theor]{Example}

\theoremstyle{definition}
\newtheorem{defin}[theor]{Definition}
\newtheorem{defin-prop}[theor]{Definition-Proposition}

\newcommand{\quash}[1]{}  

\title{A higher-dimensional Contou-Carr\`ere symbol: local theory\footnotetext{This work is supported by the RSF under a grant 14-50-00005.}}
\author{Sergey Gorchinskiy and Denis Osipov\\ \\
\small{Steklov Mathematical Institute, Moscow, Russia}\\
\small{e-mails: {\tt gorchins@mi.ras.ru}, {\tt d\_osipov@mi.ras.ru}}}
\date{}

\begin{document}
\maketitle

\begin{abstract}
We construct a higher-dimensional Contou-Carr\`ere symbol and we study its various fundamental properties. The higher-dimensional Contou-Carr\`ere symbol is defined by means of the boundary map for $K$-groups. We prove its universal property. We provide an explicit formula for the
 higher-dimensional
 Contou-Carr\`ere symbol over $\Q$ and we prove integrality of this formula. A relation with the higher-dimensional Witt pairing is also studied.
\end{abstract}

\tableofcontents

\section{Introduction}

We start by recalling the one-dimensional situation, that is, the case of Laurent series in one variable. For any commutative associative unital ring $A$, Contou-Carr\`ere~\cite{CC1},~\cite{CC2} introduced a remarkable bilinear antisymmetric pairing between the groups of invertible elements in the ring~$A((t))$ of Laurent series over $A$
\begin{equation}\label{pairing}
A((t))^*  \times A((t))^*  \lrto A^* \,,
\end{equation}
which is functorial with respect to $A$. Pairing~\eqref{pairing} is called now a Contou-Carr\`ere symbol. We will also call it a one-dimensional Contou-Carr\`ere symbol and we will denote it by~$CC_1$. Using an observation of Deligne, one can define $CC_1$ with the help of the following analogy with the analytic setting.

Let $X$ be a Riemann surface, $p$ be a point on $X$, and $f,g$ be two meromorphic functions on $X$. By $(f,g)_p\in\C^*$ denote the tame symbol of $f$ and $g$ at $p$, which is given by an algebraic formula
$$
(f,g)_p=(-1)^{\nu_p(f)\nu_p(g)}\left[\frac{f^{\nu_p(g)}}{g^{\nu_p(f)}}\right](p)\,,
$$
where $\nu_p$ denotes the valuation at $p$. Beilinson~\cite{Bei} and Deligne~\cite{Del} have independently discovered an analytic expression for $(f,g)_p$ (see also a detailed exposition in~\cite[\S\,1]{EV}). Namely, let~$U$ be a complement to finitely many points in $X$ such that~$f$ and $g$ are holomorphic and invertible on $U$. Then, by definition, the functions~$f$ and $g$ are elements of the first Deligne cohomology group ${H^1_D\big(U, \Z(1)\big) \simeq H^0 (U, \OO_{U}^*)}$ and their cup-product~$f \cup g$ is an element of the second Deligne cohomology group ${H^2_D\big(U, \z(2)\big)\simeq H^1\big(U, \OO_{U}^*  \to \Omega_U^1\big)}$. The last group admits a geometric description as the group of equivalence classes of holomorphic line bundles with holomorphic connection. The point is that the (inverse of the) tame symbol $(f,g)_p$ is equal to the monodromy of~$f \cup g$ at $p$, that is, the monodromy along a small circle $\sigma_p$ in counterclockwise direction around $p$. In particular, when~${\nu_p(f)=0}$, then the above interpretation gives immediately the following obvious formula:
\begin{equation}  \label{analytic}
(f,g)_p  = \exp  \left(   \frac{1}{2\pi i}  \oint_{\sigma_p} \log(f)\, \frac{dg}{g}   \right)=\exp\,\res\nolimits_p \left( \log(f) \, \frac{dg}{g}\right)   \,,
\end{equation}
where $\log(f)$ is a branch of the logarithm of $f$ on a small neighborhood of $p$ (see more general formulas when $\nu_p(f)$ is an arbitrary integer in~\cite{Bei} and~\cite[(2.7.2)]{Del}).

Now we come back to the Contou-Carr\`ere symbol. If $A$ is a field, then the Contou-Carr\`ere symbol coincides with the tame symbol. For example, one has the equality $CC_1(t,t)=-1$. Let us describe $CC_1$ for arbitrary Laurent series. For simplicity, suppose in what follows that $A$ is not a product of two rings. Then a Laurent series $f=\sum\limits_{l\geqslant m}a_lt^l$, $a_l\in A$, is invertible in~$A((t))$ if and only if there is an integer $l_0 \geqslant m$ such that $a_{l_0} \in A^*$ and for any $l < l_0$, the coefficient~$a_l$ is nilpotent. Denote the integer~$l_0$ by~$\nu(f)$. For all $a\in A^*$ and $g\in A((t))^*$, there is an equality $CC_1(a,g)=a^{\nu(g)}$. By bilinearity and the antisymmetric property, it remains to define $CC_1(f,g)$ when $\nu(f)=0$ and the constant term $a_0$ of $f$ equals $1$, which is the most essential case for the Contou-Carr\`ere symbol. Using an analogy with analytic formula~\eqref{analytic}, Deligne~\cite[\S\,2.9]{Del} suggested the following explicit formula for the Contou-Carr\`ere symbol $CC_1$ over a $\Q$-algebra $A$:
\begin{equation}  \label{1-dim}
CC_1(f,g)=\exp\,\res\left(\log(f)\,\frac{dg}{g}\right)\,,
\end{equation}
where $f,g\in A((t))^*$, $\nu(f)=0$, and $a_0=1$. Here ``$\log$'' and ''$\exp$'' are defined by the standard formal series. For the convergence of $\log$ one uses an $A$-linear topology on $A((t))$ with the base of open neighborhoods of zero given by $A$-submodules $t^mA[[t]]$, $m\in \Z$ (the ring $A$ is discrete). One shows that the expression under $\exp$ is a nilpotent element in $A$, whence the right hand side of formula~\eqref{1-dim} is a well-defined element in $A^*$. This defines the Contou-Carr\`ere symbol over $\Q$-algebras.

It turns out that formula~\eqref{1-dim} can be extended by elementary methods to a pairing over all rings, not only over $\Q$-algebras (see, e.g.,~\cite[\S\,2]{OZ1}). Shortly, the reason is as follows. Any element in $A((t))^*$ decomposes uniquely into a product of a power of $t$, an element in $A^*$, and an infinite product of elements of type $1-u_i t^i$, where $i\ne 0$, $u_i \in A$, and $u_i$ is nilpotent for $i<0$ and equals zero for sufficiently negative $i$. When $A$ is a $\Q$-algebra, one easily deduces from formula~\eqref{1-dim} that if $i>0$ and $j<0$, then there is an equality $CC_1(1-u_i t^i,1-v_j t^j)=(1-u_i^{-j/(i,j)}v_j^{i/(i,j)})^{(i,j)}$, where $(i,j)>0$ denotes the greatest common divisor of $i$ and~$j$. This expression does not have non-trivial denominators and makes sense over any ring. In addition, $CC_1(1-u_i t^i,1-v_j t^j)=1$ when $i$ and~$j$ have the same sign. Thus by bilinearity and the antisymmetric property, we obtain the Contou-Carr\`ere symbol over all rings. The corresponding explicit formula was given in~\cite[Intr.]{AP} in the case of Artinian rings. Note that the above discussion also implies that formally opening brackets in the right hand side of formula~\eqref{1-dim}, one actually obtains a power series with integral coefficients whose variables are coefficients of $f$ and $g$.

Moreover, the extension of $CC_1$ from $\Q$-algebras to all rings is unique. This is a consequence of an easy argument with flatness. Namely, the reason is that the functor $L\g_m(A):=A((t))^*$ on the category of all rings is represented by an ind-affine scheme, which is ind-flat. This means that $L\gm$ is a (formal) direct limit of flat affine schemes over $\z$ (see more details in~\cite[\S\,2]{OZ1}).

\medskip

Now we pass to double iterated Laurent series, which are elements of the ring~$A((t_1))((t_2))$. It is well-known to experts that, usually, complexity grows drastically with such generalization, because many basic one-dimensional facts fail in the two-dimensional case. Nevertheless, using an analog of formula~\eqref{1-dim}, the second named author and Zhu~\cite[Def.\,3.4]{OZ1} defined a multilinear antisymmetric functorial map
\begin{equation}  \label{2-dim}
CC_2\;:\;A((t_1))((t_2))^*  \times A((t_1))((t_2))^*   \times A((t_1))((t_2))^*   \lrto A^*  \,,
\end{equation}
where $A$ is a $\Q$-algebra. The key ingredient is the formula
 \begin{equation}  \label{CC_2}
 CC_2(f,g,h)= \exp\,\res \left( \log(f)\,\frac{dg}{g}   \wedge \frac{dh}{h}   \right)    \, ,
 \end{equation}
where $f,g,h \in A((t_1))((t_2))^*$ and there are some conditions on $f$ so that the series $\log(f)$ converges in the topology of $A((t_1))((t_2))$. Note that this topology is quite different from the topology on $A((t)))$. The map $\res$ is the two-dimensional residue and is equal to the coefficient by $t_1^{-1} t_2^{-1} dt_1 \wedge dt_2$.

How to extend formula~\eqref{CC_2} to a map over all rings, not just $\Q$-algebras?  Similarly as in the one-dimensional case, one shows that the value of $CC_2$ on elements of type $1 -u_{i,j} t_1^i t_2^j$ defined by formula~\eqref{CC_2} does not have non-trivial denominators and thus makes sense over all rings. However, as it is explained in~\cite[\S\,3.4]{OZ1}, this method works, for example, for Noetherian rings, but it is unclear how this method could work for all rings. The problem is with the decomposition of an arbitrary element from $A((t_1))((t_2))^*$
into an infinite product of elements of type $1-u_{i,j} t_1^i t_2^j$.  Thus less elementary methods were applied in~\cite{OZ1} in order to extend the map $CC_2$ to all rings. More exactly, there were used two alternative approaches: one based on categorical central extensions and higher commutators, and the other one based on the boundary map for algebraic $K$-groups
$$
\partial_{m+1} \;  :  \;   K_{m+1} \big(A((t))\big)  \lrto K_m(A)   \,  , \qquad m \geqslant 0  \,  ,
$$
where $A$ is any commutative ring. Note that it is important here that the ring $A$ can be non-regular, in particular, can have nilpotent elements. The construction of the maps $\partial_{m+1}$ is based on ``the Localization Theorem for projective modules'' proved by Gersten~\cite{Ger} and in a more general case by Grayson~\cite{Gr}. Kato~\cite[\S\,2.1]{K1} applied this theorem to a particular ring $A((t))$ to obtain the maps~$\partial_{m+1}$. In~\cite[\S\,7]{OZ1} the construction of the maps~$\partial_{m+1}$ was extended to arbitrary (including non-Noetherian) commutative rings. The maps~$\partial_{m+1}$ are functorial with respect to a ring $A$ (see Proposition~\ref{nat-tr} of the present paper).

Now the map $CC_2$ for an arbitrary ring $A$ is defined as the composition of the product between algebraic $K$-groups, the map $\partial_{m+1}$, and the natural map  $\det\colon K_1(A)\to A^*$:
\begin{multline}  \label{2-dim-K-theory}
A((t_1))((t_2))^*  \times A((t_1))((t_2))^*   \times A((t_1))((t_2))^*   \lrto \\
\lrto K_3\big(A((t_1))((t_2))\big)   \stackrel{\partial_3}{\lrto}  K_2\big(A((t_1))\big)  \stackrel{\partial_2}{\lrto}
K_1(A) \stackrel{\det} {\lrto} A^*  \, .
\end{multline}
It was proved in~\cite[Theor.\,7.2]{OZ1} that the map defined by an explicit formula~\eqref{CC_2} coincides with the one defined by the formula~\eqref{2-dim-K-theory} over $\Q$-algebras. For this it was used functoriality of the map given by formula~\eqref{2-dim-K-theory} together with a careful analysis of geometric properties of the ind-affine scheme that represents the functor $L^2\gm(A):=A((t_1))((t_2))^*$.

Finally, it was proved in~\cite[Lem.~3.10]{OZ1} that the extension of $CC_2$ from $\Q$-algebras to all rings is unique. For this it was used that the ind-scheme $L^2 \g_m$ is from the class~$\mathcal{E}\mathcal{F}$ (essentially flat) \cite[Def.\,3.1, Lem.\,3.5]{OZ1}, which is a weaker property than being ind-flat. Note that contrary to the ind-affine scheme $L\g_m$, the ind-affine scheme $L^2 \g_m$ is not known to be ind-flat (in particular, schemes in the standard representation of $L^2\gm$ are not flat, see a discussion after~\cite[Lem.\,3.4]{OZ1}).

\medskip

We came to the higher-dimensional case, that is, to iterated Laurent series in~$n$ variables for arbitrary $n\geqslant 1$. Define the functor $L^n \g_m (A):= A((t_1))  \ldots ((t_n))^*$ and, as usual, put $\g_m(A):= A^*$. From the above point of view it is now natural to define the $n$-dimensional Contou-Carr\`ere symbol $CC_n$ as a multilinear antisymmetric morphism of functors $(L^n \g_m)^{\times (n+1)}\to \g_m$ given by the formula
\begin{equation}  \label{CC_n}
\begin{CD}
CC_n \quad :  \quad  L^n \g_m (A)^{\times (n+1)}
@>>>
K_{n+1}\big(A((t_1))  \ldots ((t_n))\big)
@>{\partial_2  \cdot \ldots \cdot \partial_{n+1}}>>
 K_1(A)
 @>{\det}>>
 A^* \, ,
\end{CD}
\end{equation}
where the first arrow is the product between algebraic $K$-groups. It is also natural to expect that when $A$ is a $\Q$-algebra then $CC_n$ is given by the following explicit formula, which generalizes formulas~\eqref{1-dim} and~\eqref{CC_2}:
\begin{equation}  \label{Q-n-dim}
CC_n(f_1,f_2,\ldots,f_{n+1})=\exp\,\res\left(\log(f_1)\,\frac{df_2}{f_2}\wedge\ldots\wedge
\frac{df_{n+1}}{f_{n+1}}\right)   \, ,
\end{equation}
where $f_1, \ldots, f_{n+1}  \in L^n \g_m (A)$, $f_1$ satisfies additional conditions so that $\log(f_1)$ is well-defined, and the map $\res$ is the $n$-dimensional residue. Note that an analytic analog of expression~\eqref{Q-n-dim} was studied by Brylinski and McLaughlin in~\cite{BM}, where it was related with the product in Deligne cohomology of an $n$-dimensional complex manifold. Obtaining an equality between formulas~\eqref{CC_n} and~\eqref{Q-n-dim} for $\Q$-algebras was one of the main motivations for the present paper.

\quash{
Finally, by analogy with the case $n=1$, one can also ask whether formula~\eqref{Q-n-dim} is actually given by a power series with integral coefficients whose variables are coefficients of $f_1,\ldots,f_{n+1}$. }

\medskip

Let us explain main results of the paper. For this  goal, first note that since the product between algebraic $K$-groups satisfies the Steinberg relations, the map $CC_n$ defined by formula~\eqref{CC_n} factors through the Milnor $K$-group $K^M_{n+1}\big(A((t_1))\ldots((t_n))\big)$. Here, given a commutative ring $B$, its Milnor \mbox{$K$-group} $K^M_{m}(B)$, $m\geqslant 0$, is defined as the quotient of the group~$(B^*)^{\otimes m}$ by the subgroup generated by Steinberg relations, that is, by elements ${b_1 \otimes \ldots \otimes b_{m}}$ such that $b_i+b_{i+1}=1$ for some $i$, $1\leqslant i\leqslant m-1$. Define the functor $L^nK^M_{n+1}(A):=K^M_{n+1}\big(A((t_1))\ldots((t_n))\big)$.

Our key result is the isomorphism of groups
$$
\Hom^{gr}\big(L^nK^M_{n+1},\gm\big)\simeq\z\,,
$$
where $\Hom^{gr}$ denotes the group of all morphisms between the corresponding functors that respect the group structures, see Theorem~\ref{theor-key}. Besides, the morphism $CC_n$ is the generator of this infinite cyclic group, see Theorem~\ref{theor:intCC}.

Explicitly, this means the following. Let $\Phi\colon (\lo^n\gm)^{\times(n+1)}\to\gm$ be any multilinear morphism of functors that satisfies the Steinberg relations. Then $\Phi=(CC_n)^i$ for an integer $i\in\z$. As far as we know, this universal property of the $n$-dimensional Contou-Carr\`ere symbol is new also in the classical one-dimensional case.

Note that we prove Theorems~\ref{theor-key} and~\ref{theor:intCC} in a more general setting, namely, for restriction of the functors involved to the category of $R$-algebras, where $R$ is an arbitrary ring such that the natural homomorphism $R\to R\otimes_{\z}\Q$ is injective.

Moreover, in the course of the proof of Theorem~\ref{theor-key} we obtain an explicit formula for the map $CC_n$ over $\Q$-algebras thus proving the equality between formulas~\eqref{CC_n} and~\eqref{Q-n-dim}, see Theorem~\ref{prop-CCrational}. This gives a new sense to formula~\eqref{Q-n-dim} as the result of direct calculations of certain canonical maps. (Note that we announced the equality between formulas~\eqref{CC_n} and~\eqref{Q-n-dim} in the short note~\cite{GO1}, which did not contain proofs.)

Also, we show that for any natural number $N$, the extension of any morphism of functors $(L^n\gm)_{\Q}^{\times N}\to(\gm)_{\Q}$ from $\Q$-algebras to all rings is unique, see Theorem~\ref{cor:uniq}.

We prove that formally opening brackets in the right hand side of formula~\eqref{Q-n-dim}, one obtains a power series with integral coefficients whose variables are coefficients of $f_1,\ldots,f_{n+1}$, see Theorem~\ref{theor:integalexpl}. This gives an explicit expression over any ring for the map defined by formula~\eqref{CC_n}. Note that this integrality result generalizes the case $n=1$ and is new already in the case $n=2$.

As an illustration of our results, we show how the Contou-Carr\`ere symbol $CC_n$ and its explicit formula~\eqref{Q-n-dim} lead to a new definition of the $n$-dimensional generalization of the Witt pairing, see Proposition~\ref{prop-Witt}. Note that the $n$-dimensional Witt pairing is crucial for the Parshin explicit construction of the higher local class field theory, see~\cite{P1}.

\medskip

Our method is based on the study of geometric properties of ind-affine schemes that represent the functor $\lo^n\gm$ and its certain special subgroups. Toward this end we develop a theory of so-called thick ind-cones. A thick ind-cone $X$ over a ring $R$ is an ind-closed subscheme in a (possibly, infinite-dimensional) affine space over $R$ with some additional properties, see Definitions~\ref{defin:indcone},~\ref{defin:dset}. Roughly speaking, $X$ is invariant under homothethies and contains sufficiently many points with nilpotent coordinates. The main feature of a thick ind-cone $X$ is that regular functions on $X$ are uniquely expanded as power series in coordinates of the ambient affine space. Thick ind-cones possess many additional nice properties, in particular, the class of thick ind-cones in closed under products and extensions of scalars, which perfectly fits our needs. One can say that the class of thick ind-cones is a suitable replacement of the class of ind-flat ind-affine schemes. We expect a vast range of further applications of this technique to various questions on iterated Laurent series over rings. The notion of a thick ind-cone arose as a modification of the notion of an ind-scheme from the class ${\mathcal E}{\mathcal F}$ used in~\cite{OZ1} for two-dimensional iterated Laurent series.

\medskip

The paper is organized as follows. Section~\ref{sect:notation} contains notation and terminology, mainly concerning functors, which is used throughout the paper. In Section~\ref{iterated}, we collect and prove some general facts about the ring of iterated Laurent series $A((t_1))\ldots((t_n))$, the topology on it, and its differential forms.

In Section~\ref{multiplicative}, we define iterated loop groups, see Definition~\ref{defin:loopgroup}, and introduce our main object of study, the iterated loop group $L^n\gm$. We construct a decomposition of~$\lo^n\gm$ that generalizes a result of Contou-Carr\`ere in the case $n=1$, see Proposition~\ref{prop-decomp}. We also introduce special group subfunctors $(\lo^n\gm)^{0}$ and~$(\lo^n\gm)^{\sharp}$ of $\lo^n\gm$, see Definition~\ref{defin:specialmult}. Namely, $(\lo^n\gm)^{0}$ is the group of iterated Laurent series with trivial valuation, and~$(\lo^n\gm)^{\sharp}$ is a subgroup of $(\lo^n\gm)^0$ such that the logarithm is well-defined on it over $\Q$, see Proposition~\ref{log-map}.

We develop the theory of thick ind-cones in Section~\ref{representab}. First we give a short account of ind-affine schemes in Subsection~\ref{subsect:indschemes}. Then we introduce a ring of power series in infinitely many variables, see Definition~\ref{defin:powerseries}, define algebraic convergence of power series, see Definition~\ref{defin-conv}, and consider certain ind-closed subschemes in ind-affine spaces that consist of points with nilpotent coordinates, see Definition~\ref{defin:hat}. These ind-closed subschemes are used to define thick ind-cones, see Definitions~\ref{defin:indcone},~\ref{defin:dset}. The main properties of thick ind-cones are contained in Propositions~\ref{prop-injectiveseries},~\ref{prop:geninj} and Lemmas~\ref{lemma:decomposethickcone},~\ref{lemma:extescalthickcone}. Finally, we discuss connectedness of ind-schemes over a base ring in Subsection~\ref{subsect:connect} and density of an ind-closed subscheme of an ind-affine scheme in Subsection~\ref{subsection:density}.

In Section~\ref{section:reprfunc}, we apply the theory of thick ind-cones to iterated loop groups. We start by introducing a non-commutative product of (strict) ind-sets, see Definition~\ref{defin:starindind}, which is very useful to work with representability of iterated loop functors, see Proposition~\ref{lemma-repraffine}. Then we prove that the ind-affine schemes that represent the functor $\lo^n\gm$ and its special subgroups are products of thick-ind cones and ind-flat ind-affine schemes, see Proposition~\ref{lemma:sharprepr}. This leads to many nice properties of these ind-affine schemes, see Theorems~\ref{prop-key} and~\ref{cor:uniq}. We also show that $(\lo^n\gm)^0$ is connected and that $(\lo^n\gm)^{\sharp}$ is dense in $(\lo^n\gm)^0$, see Proposition~\ref{prop:densegm}. We study characters of $\lo^n\gm$ over $\Q$ and prove that they commute with their differentials through the exponential map, see Propositions~\ref{lemma:difflngmga} and~\ref{prop:charactgm}. Finally, we show that any functorial linear map $\Omega^n_{A((t_1))\ldots((t_n))}\to A$ factors through the group of ``continuous'' differential forms $A((t_1))\ldots((t_n))dt_1\wedge\ldots\wedge dt_n$, see Proposition~\ref{lemma:algiterforms}.

Section~\ref{sect:Ktheory} collects auxiliary facts on Milnor $K$-groups and algebraic $K$-groups. Namely, we recall the main result from~\cite{GOMilnor} that claims an isomorphism between the tangent space to the Milnor $K$-group $TK^M_{m+1}(A)$ and the group of absolute K\"ahler differential forms $\Omega^m_A$ when $A$ is a ring that contains $\frac{1}{2}$ and has sufficiently many invertible elements, see Theorem~\ref{thm:tangentMilnor} (this is a generalization of a well-known result of Bloch~\cite{Blo}). In Subsection~\ref{subsect:bound}, we recall the construction of a boundary map for algebraic $K$-groups $\partial\colon K_{m+1}\big(A((t))\big)\to K_m(A)$ and show its functoriality, see Proposition~\ref{nat-tr}.

Section~\ref{sect:main} contains the main results of the paper described above and their proofs, see Theorems~\ref{theor-key},~\ref{theor:intCC},~\ref{prop-CCrational}, and~\ref{theor:integalexpl}. Before stating the main results, we study in Subsection~\ref{subsect:add} a simpler case of an additive symbol, see Definition~\ref{def:addsymb}, which allows to see the main patterns related to the Contou-Carr\`ere symbol. The proof of the key result, Theorem~\ref{theor-key}, is contained in Subsection~\ref{subsect:proofrigid}. Briefly, first we pass from a base ring~$R$ to the ring $S:=R\otimes_{\z}\Q$ using the theory of thick ind-cones. Since $S$ is a $\Q$-algebra, we reduce a character of $(L^nK^M_{n+1})_S$ to its differential. Finally, we apply the description of the tangent space to Milnor $K$-groups in order to obtain both the key result and the explicit formula~\eqref{Q-n-dim}. In order to show the  result on integrality of the explicit formula~\eqref{Q-n-dim}, we introduce a completed version of the Contou-Carr\`ere symbol, see Definition~\ref{defin:compCC}, which is an interesting object of further study in its own right.

Finally, after a short account of the explicit higher local class field theory in Subsection~\ref{subsect:localcft}, we relate
the Contou-Carr\`ere symbol $CC_n$ with the $n$-dimensional Witt pairing in Subsection~\ref{ASWp}.

\medskip

The authors are grateful to S.\,Galkin for pointing out that a particular case of the result on integrality of the explicit formula~\eqref{Q-n-dim} has been obtained by Kontsevich in~\cite{Kon}.

\section{Notation and terminology}\label{sect:notation}

By a {\it ring} we mean a commutative associative unital ring and the same for {\it algebras} over a ring. Throughout the paper, $A$ and $R$ denote arbitrary rings unless something more is specified on the rings $A$ and $R$.

\medskip

Let us fix some terminology concerning functors. We work with covariant functors from the category of commutative associative unital rings to the category of sets and to the category of Abelian groups. We call them, for short, just {\it functors} and {\it group functors}, respectively. We usually denote test rings on which we evaluate functors by $A,B,\ldots$

By a {\it subfunctor} $F\subset G$, we mean a morphism of functors $F\to G$ such that for any ring~$A$, the corresponding map $F(A)\to G(A)$ is injective. By a {\it group subfunctor} $F\subset G$, we mean a subfunctor such that the corresponding morphism $F\to G$ is, in addition, a morphism of group functors.

Given a ring $R$, by a {\it functor over $R$} (respectively, a {\it group functor over $R$}), we mean a covariant functor from the category of commutative associative unital $R$-algebras to the category of sets (respectively, to the category of Abelian groups). Given a functor~$F$, by~$F_{R}$ denote the corresponding functor over $R$ which is the restriction of~$F$ to the category of $R$-algebras. We usually denote base rings over which we consider functors by $R,S,\ldots$

All functors that we consider in this paper satisfy the following property: morphisms from a functor~$F$ over a ring $R$ to a functor $G$ over  $R$ form a set $\Hom_R(F,G)$. For example, this holds if $F$ and $G$ are representable by (ind-)affine schemes over $R$. If $F$ and $G$ are group functors, then by~$\Hom^{gr}_R(F,G)$ we denote the set of all morphisms of group functors from~$F$ to~$G$. By~$\underline{\Hom}^{gr}_R(F,G)$ we denote the internal Hom, that is, $\underline{\Hom}^{gr}_R(F,G)$ is a group functor over $R$ that sends and $R$-algebra $A$ to the group $\Hom^{gr}_A(F_A,G_A)$.

\medskip

Let us also introduce notation for various products. Let $n$ be a natural number. If $X$ is an object in a category with finite products (e.g., $X$ is a set, a scheme, or a functor), then by~$X^{\times n}$ we denote the $n$-th Cartesian power of~$X$ (which is a set, a scheme, or a functor, respectively). We emphasize that even if $X$ is additionally a group object (e.g.,~$X$ is a group, a group scheme, or a group functor), then $X^{\times n}$ still denotes just an object in the initial category, without a group structure (which is a set, a scheme, or a functor, respectively).

In contrast, if $M$ is an object in an additive category (e.g., $M$ is an Abelian group, a module over a ring, or a sheaf of Abelian groups), then $M^{\oplus n}$ denotes the $n$-fold direct sum of $M$ with itself (which is an Abelian group, a module over the ring, or a sheaf of Abelian groups, respectively). We sometimes abbreviate $M^{\oplus n}$ to $M^n$, in particular, when $M=\Z$.

For Abelian groups $M$ and $N$, by a {\it multilinear map} $\varphi\colon M^{\times n}\to N$ we mean a map of sets (according to our notation, $M^{\times n}$ is just a set) such that
$$
\varphi(m_1,\ldots,m_i,\ldots,m_n)+\varphi(m_1,\ldots,m'_i,\ldots,m_n)=
\varphi(m_1,\ldots,m_i+m'_i,\ldots,m_n)
$$
for all $m_1,\ldots,m_i,m'_i\ldots,m_n\in M$, $1\leqslant i\leqslant n$. Note that we still use the term ``multilinear'' even if the group law in $M$ or $N$ is denoted multiplicatively, in particular, for the group of invertible elements of a ring with the group law given by the product of elements.

\section{Ring of iterated Laurent series}\label{iterated}

\subsection{Definition}\label{subsect:Laurent}

 Given a ring $A$, we have the ring $A[[t]]$ of {\it power series} and the ring $A((t))=A[[t]][t^{-1}]$ of {\it Laurent series} over $A$ in a formal variable $t$. Explicitly, $A[[t]]$ is the ring of all series of the form $\sum\limits_{0 \leqslant l \in \Z} a_lt^l$, where $a_l \in A$, and $A((t))$ is the ring of all series of the form $\sum\limits_{ m \leqslant l \in \Z} a_l t^{l}$, where $m$ can be any integer and $a_l\in A$. For short, we put
$$
\LL(A):=A((t))\,.
$$
Repeating this construction, we obtain the ring of {\it iterated Laurent series} over $A$ in formal variables $t_1,\ldots,t_n$:
$$
\LL^n(A):=A((t_1))\ldots((t_n))\,.
$$

\medskip

Let us give an explicit description of iterated Laurent series. For this, introduce a set
$$
\Lambda_n:=\big\{(\lambda_1,\ldots,\lambda_n)\;\mid\;\lambda_p \, \colon  \, \z^{n-p}\lrto \z \quad \mbox{for} \quad 1\leqslant p\leqslant n-1,\quad \lambda_n\in\Z \big\}\,.
$$
We stress that $\lambda_p$ is an arbitrary function for each $p$, where $1 \leqslant p \leqslant n-1$. Given \mbox{$\lambda=(\lambda_1,\ldots,\lambda_n)\in\Lambda_n$}, define a set
$$
\z^n_{\lambda}:=\big\{(l_1,\ldots,l_n)\in\Z^n\;\mid\;l_n\geqslant \lambda_n,\,l_{n-1}\geqslant \lambda_{n-1}(l_n),\,\ldots,\,l_{1}\geqslant \lambda_{1}(l_2,\ldots,l_n)\big\}\,.
$$
For a multi-index $l=(l_1,\ldots,l_n)\in\z^n$, put $t^{l}:=t_1^{l_1}\ldots t_n^{l_n}$. Then~$\LL^n(A)$ is the ring of all series of the form $\sum\limits_{l\in \z^n_{\lambda}} a_l t^{l}$, where $\lambda\in\Lambda_n$ and $a_l\in A$.

\medskip

By $\Nil(R)$ denote the nilradical of a ring $R$, that is, the set of all nilpotent elements of~$R$. We will use the following fact about nilpotent iterated Laurent series.

\begin{rmk}\label{rmk:nilpseries}
Suppose that
$$
\mbox{$f=\sum\limits_{l\in\z^n}a_lt^l=\sum\limits_{i\in\Z}g_i t_n^i$}
$$
is a nilpotent element in $\LL^n(A)$, where $g_i\in\LL^{n-1}(A)=A((t_1))\ldots((t_{n-1}))$. One shows by induction on $i$ that $g_i\in\Nil\big(\LL^{n-1}(A)\big)$ for all $i\in\Z$. Further, by induction on $n$, we obtain that $a_l\in\Nil(A)$ for all $l\in\z^n$.
\end{rmk}

If $A$ is Noetherian, then the converse is true as well, because $\Nil(A)^N=0$ for some~\mbox{$N\in\N$}. However, in general, the converse is false. For example, take
$$
A=\Z[\varepsilon_1,\varepsilon_2,\ldots]/(\varepsilon_1^2,\varepsilon_2^3,\ldots)
$$
and consider the power series $f=\sum\limits_{l\geqslant 1}\varepsilon_lt^l$ in $A[[t]]$. Then all coefficients of $f$ are nilpotent, but $f$ is not nilpotent (one can show this by taking reductions modulo primes).

\subsection{Topology}\label{subsect:topology}

Let us introduce a topology on the ring of iterated Laurent series $\LL^n(A)$ over a ring~$A$. This topology is given by iterated direct and inverse limits and is defined inductively as follows. A base of open neighborhoods of zero in $\LL(A)=A((t))$ consists of \mbox{$A$-submodules} \mbox{$U_m:=t^m A[[t]]$}, $m\in \Z$. A base of open neighborhoods of zero in $\LL^n(A)=A((t_1))\ldots ((t_n))$ consists of \mbox{$A$-submodules}
\begin{equation}  \label{top}
\mbox{$U_{m, \{V_j \}}:=\Big(\bigoplus\limits_{j<m} t_n^j\cdot V_j \Big)$} \,\oplus \, \mbox{$ t_n^m\cdot
\LL^{n-1}(A)[[t_n]]\,,$}
\end{equation}
where $m\in\Z$ and for every integer $j$ with $j<m$, we have that the $A$-module $V_j$ is from the base of open neighborhoods of zero in $\LL^{n-1}(A)=A((t_1))\ldots((t_{n-1}))$. Now a topology on $\LL^n(A)$ is defined uniquely by the condition that the additive group of $\LL^n(A)$ is a topological group.

\begin{rmk}
When $A={\mathbb F}_q$ is a finite field, the above topology on the higher local field \mbox{$\LL^n(A)={\mathbb F}_q((t_1))\ldots((t_n))$} was introduced by Parshin~\cite{P1} for constructions in \mbox{$n$-dimensional} local class field theory.
\end{rmk}

Recall that a {\it Cauchy sequence} in a topological (Abelian) group is a sequence of elements $\{f_i\}$, $i\in\N$, such that for any open neighborhood of zero $U$, there is $N\in\N$ that satisfies $f_i-f_j\in U$ for all $i,j\geqslant N$. Clearly, this property does not depend on the order of elements in the sequence and thus is well-defined for an arbitrary (non-ordered) countable set of elements.

\begin{lemma}\label{lemma:top}
\hspace{0cm}
\begin{itemize}
\item[(o)]
The topological space $\LL^n(A)$ is Haussdorf.
\item[(i)]
For any Cauchy sequence $\{f_i\}$, $i\in \N$, in $\LL^n(A)$, there is $m\in \Z$ such that
$$
f_i\in t_n^m\cdot\LL^{n-1}(A)[[t_n]]=t_n^m\cdot A((t_1))\ldots((t_{n-1}))[[t_n]]
$$
for all $i\in \N$.
\item[(ii)]
Every Cauchy sequence has a limit in $\LL^n(A)$.
\item[(iii)]
The natural isomorphism of rings
$$
\LL^{n-1}(A)[[t_n]]\simeq \varprojlim\limits_{m\in\N}\LL^{n-1}\big(A[[t_n]]/(t_n^m)\big)
$$
is also an isomorphism of topological groups, where the topology on the left hand side is restricted from $\LL^n(A)$ and the topology on the right hand side is the inverse limit of topologies on $\LL^{n-1}(B_m)$, where the ring $B_m$ is $A[[t_n]]/(t_n^m)$.
\end{itemize}
\end{lemma}
\begin{proof}

$(o)$ This is just obvious.

The proofs of items~$(i)$ and $(ii)$ are similar to that of~\cite[\S\,1, Prop.\,2.1]{P1} and~\cite[\S\,1, Prop.\,2.2]{P1}, respectively. We give them for convenience of the reader.

$(i)$ Suppose the converse. Then we obtain a strictly decreasing sequence of negative integers $\{j_k\}$, $k\in\N$, such that for each $k\in \N$, there is $i_k\in\N$ that satisfies
$$
f_{i_k}\in t_n^{j_k}\cdot\LL^{n-1}(A)[[t_n]]\,,\qquad f_{i_k}\notin t_n^{j_k+1}\cdot\LL^{n-1}(A)[[t_n]]\,.
$$
Since $\LL^{n-1}(A)$ is Haussdorf, there is an open neighborhood of zero $W_k\subset \LL^{n-1}(A)$ such that
$$
f_{i_k}\notin t_n^{j_k}\cdot W_k \, \oplus  \, t_n^{j_k+1}\cdot\LL^{n-1}(A)[[t_n]]\,.
$$
Now for each $j<0$, let $V_j:=W_k$ if $j=j_k$ for some $k\in\N$ and let $V_j$ be any open neighborhood of zero in~$\LL^{n-1}(A)$ otherwise. Then the subsequence $\{f_{i_k}\}$ (and hence the sequence $\{f_i\}$) does not satisfy the Cauchy condition with respect to the open neighborhood of zero \mbox{$U_{0, \{V_j \}}$} in $\LL^n(A)$, which gives a contradiction.

$(ii)$ The proof is based on item~$(i)$ and a trivial induction on~$n$.

$(iii)$ Note that for any $A$-module $M$, we can consider an \mbox{$\LL^n(A)$-module} ${\LL^n(M):=M((t_1))\ldots((t_n))}$ with an analogous topology as on~$\LL^n(A)$ given by a formula similar to~\eqref{top}. (We note that $M$ has the discrete topology.) One shows by induction on $n$ that for all $A$-modules $M$ and $M'$, the natural isomorphism of~\mbox{$\LL^n(A)$-modules}
$$
\LL^n(M)\oplus\LL^n(M')\simeq \LL^n(M\oplus M')
$$
is also an isomorphism of topological groups. In particular, for any $m\in\N$, we have an isomorphism of topological groups
$$
\LL^{n-1}(A)^{\oplus m}\simeq \LL^{n-1}(A^{\oplus m})\,.
$$
Note that for any $A$-algebra $B$, we readily have by induction on $n$ that the base of open neighborhoods of zero in $\LL^{n-1}(B)$ given by $B$-modules by equation~\eqref{top} coincides with the base of open neighborhoods of zero in $\LL^{n-1}(M)$ given by $A$-modules. Therefore applying an isomorphism of $A$-modules \mbox{$A^{\oplus m}\simeq A[[t_n]]/(t_n^m)$} and using the isomorphism of topological groups
$$
\LL^{n-1}(A)[[t_n]]\simeq\varprojlim\limits_{m\in \N}\LL^{n-1}(A)^{\oplus m}\,,
$$
we finish the proof.
\end{proof}

\begin{rmk}
Actually, a stronger statement than Lemma~\ref{lemma:top}$(ii)$ holds: the topology on~$\LL^n(A)$ is complete, that is, every Cauchy net has a limit in~$\LL^n(A)$, not only a Cauchy sequence. Equivalently, the natural homomorphism
$$
\LL^n(A)\lrto\varprojlim\limits_{(m,\{V_j\})}\LL^n(A)/U_{m,\{V_j\}}
$$
is an isomorphism. The proof is by induction on $n$ and uses that inverse limits commute with direct sums in the following sense: suppose that for each element $\lambda$ in a set $\Lambda$, it is given a directed set $I_{\lambda}$ and an inverse system of Abelian groups $\{P^{\lambda}_i\}$, $i\in I_{\lambda}$. Define a directed set $J:=\prod\limits_{\lambda\in\Lambda}I_{\lambda}$. Then there is an isomorphism of Abelian groups
$$
\mbox{$\bigoplus\limits_{\lambda\in\Lambda}\Big(\varprojlim\limits_{i\in I_{\lambda}}P_i^{\lambda}\Big)\simeq
\varprojlim\limits_{(i_{\lambda})\in J}\Big(\bigoplus\limits_{\lambda\in\Lambda}P^{\lambda}_{i_\lambda}\Big)$}\,.
$$
The proof of this fact follows the same idea as the proof of Lemma~\ref{lemma:top}$(i)$.
\end{rmk}

Note that $\LL^n(A)$ is not a topological ring  when $n \geqslant 2$, since  for all open $A$-submodules $U$ and $U'$ in~$\LL^n(A)$ we have $U \cdot U' =\LL^n(A)$,~\cite[\S\,1, Rem.\,1]{P1}. Indeed, for any iterated Laurent series $f\in\LL^n(A)$, there is a monomial $t_n^it_{n-1}^j$, $i,j\in\Z$, such that $t_n^it_{n-1}^j f\in U$ and $(t_n^it_{n-1}^j)^{-1}\in U'$. However the next lemma still asserts a certain compatibility between the topology and the product in $\LL^n(A)$.

\begin{lemma}\label{lemma:seriesconvdom}
\hspace{0cm}
\begin{itemize}
\item[(i)]
For any element $f\in \LL^n(A)$, multiplication by $f$ is a continuous map from~$\LL^n(A)$ to itself.
\item[(ii)]
Given two Cauchy sequences $\{f_i\}$, $i\in \N$, and $\{g_{j}\}$, $j\in \N$, in $\LL^n(A)$, their pair-wise product $\{f_i g_{j}\}$, $(i,j)\in \N^{\times 2}$, is a Cauchy sequence as well (with respect to any bijection $\N^{\times 2}\simeq\N$) and the limit of the sequence $\{f_i g_j\}$ equals the product of the limits of $\{f_i\}$ and $\{g_j\}$.
\end{itemize}
\end{lemma}
\begin{proof}
We follow an argument from the proof of~\cite[\S\,1, Prop.\,2.3]{P1}. In order to show both items, we use induction on $n$. The case $n=1$ is clear for each of two items. To make the induction step, observe that item~$(i)$ is evident for $f=t_n^m$, $m\in\Z$, that is, multiplication by $t_n$ is a homeomorphism of the topological group $\LL^n(A)$. Now using this fact, we consider each item separately.

$(i)$ By item~$(i)$ for $t_n^m$, $m\in\Z$, we can suppose that $f \in \LL^{n-1}(A)[[t_n]]$. It follows from formula~\eqref{top} that a subset of $\LL^n(A)$ is open if and only if its intersection with each $A$-submodule $t_n^l \cdot \LL^{n-1}(A)[[t_n]]$, $l\in\Z$, is open. Therefore it is enough to prove that multiplication by $f$ is a continuous map from $t_n^l \cdot \LL^{n-1}(A)[[t_n]]$ to itself for all $l\in\Z$. We can also fix $l=0$ and finish the proof using Lemma~\ref{lemma:top}$(iii)$ and the induction hypothesis.

$(ii)$ By item~$(i)$ for $t_n^m$, $m\in\Z$, and by Lemma~\ref{lemma:top}$(i)$, it is enough to prove a version of item~$(ii)$ for the subring
$$
A((t_1))\ldots((t_{n-1}))[[t_n]]=\LL^{n-1}(A)[[t_n]]\subset \LL^n(A)
$$
with the topology being restricted from $\LL^n(A)$. By the induction hypothesis and the isomorphism from Lemma~\ref{lemma:top}$(iii)$, this finishes the proof.
\end{proof}

\subsection{Convergence of power series}

By definition, a series $\sum\limits_{i\geqslant 0}f_i$ of elements of a topological Abelian group {\it converges} if there is a limit of the sequence of partial sums~$\sum\limits_{i=1}^N f_i$. By Lemma~\ref{lemma:top}$(ii)$, a series converges in~$\LL^n(A)$ if and only if the sequence of its terms $\{f_i\}$, $i\in\N$, tends to zero, because the base of open neighborhoods of zero is given by subgroups. If this holds, the result of the summation does not depend on the order of summation and also on the way to represent the series as a double series or a higher iterated series, cf.~\cite[Ch.\,4, Sec.\,5, Theor.\,1]{BS}. In particular, convergence is well-defined for an infinite sum of countably many elements in~$\LL^n(A)$.

\begin{rmk}\label{rmk:multtop}
Given two convergent series $\sum\limits_{i\geqslant 0}f_i$ and $\sum\limits_{j\geqslant 0}g_j$ in $\LL^n(A)$, by Lemma~\ref{lemma:seriesconvdom}$(ii)$, the series $\sum\limits_{(i,j)\in\N^{\times 2}}f_ig_j$ is also convergent (with respect to any bijection $\N^{\times 2}\simeq\N$) and we have the equalities
$$
\mbox{$\Big(\sum\limits_{i\geqslant 0}f_i\Big)\cdot\Big(\sum\limits_{j\geqslant 0}g_j\Big)=\sum\limits_{i\geqslant 0}\Big(\sum\limits_{j\geqslant 0}f_ig_j\Big)=\sum\limits_{j\geqslant 0}\Big(\sum\limits_{i\geqslant 0}f_ig_j\Big)\,.$}
$$
\end{rmk}

\begin{defin}\label{defin:addsharp}
Define $\LL^n(A)^{\sharp}$ to be the set of all elements $f\in\LL^n(A)$ such that the sequence~$\{f^i\}$,~$i\in\N$, tends to zero in $\LL^n(A)$.
\end{defin}

It follows from Lemma~\ref{lemma:seriesconvdom}$(ii)$ that if the sequences $\{f^i\}$ and $\{g^i\}$ tend to zero, then $\{(f+g)^i\}$ tends to zero as well. In other words, $\LL^n(A)^{\sharp}$ is a subgroup of $\LL^n(A)$. Also, for any power series $\varphi\in A[[x_1,\ldots,x_r]]$ and elements $f_1,\ldots,f_r\in\LL^n(A)^{\sharp}$, the series $\varphi(f_1,\ldots,f_r)$ converges in $\LL^n(A)$. Here is a more explicit description of the group~$\LL^n(A)^{\sharp}$, cf.~\cite[Lem.\,1.1]{FVZ}.

\begin{prop}\label{prop:analsharp}
For any ring $A$, there is an equality of groups
\begin{equation}\label{eq:expladd}
\LL^n(A)^{\sharp}=\Big\{\mbox{$\sum\limits_{l\in\z^n} a_{l}t^{l}\;\mid\;
\sum\limits_{l >0} a_{l}t^{l} \in \LL^n(A)
,\,\sum\limits_{l\leqslant 0} a_lt^{l} \in \Nil\big(\LL^n(A)\big)$}\Big\}\,.
\end{equation}
\end{prop}
\begin{proof}
The proof is by induction on $n$. The base $n=0$, that is, the case of $A$ with a discrete topology, is clear: $\{f^i\}$ tends to zero in $A$ if and only if $f$ is a nilpotent element in $A$. Let us show an induction step. It is readily seen that both sides of formula~\eqref{eq:expladd} contain the group $t_n\cdot\LL^{n-1}(A)[[t_n]]$.

Now take an element \mbox{$f=\sum\limits_{i\leqslant 0}g_i t_n^i\in \LL^n(A)^{\sharp}$}, where $g_i\in\LL^{n-1}(A)$. Since $\{f^i\}$ is a Cauchy sequence, using Lemma~\ref{lemma:top}$(i)$, we obtain that $g_i\in\Nil\big(\LL^{n-1}(A)\big)$ for all $i<0$. Thus the series $\sum\limits_{i<0}g_it_n^i$ is nilpotent.

Further, by the induction hypothesis, we have the equality between the intersections of the subring $\LL^{n-1}(A)\subset\LL^n(A)$ with two sides of formula~\eqref{eq:expladd}. To prove this, we use Remark~\ref{rmk:nilpseries} and the fact that the restriction of topology from~$\LL^n(A)$ to $\LL^{n-1}(A)$ coincides with the topology on $\LL^{n-1}(A)$. This proves that the group $\LL^n(A)^{\sharp}$ is a subgroup of the group defined by the right hand side of formula~\eqref{eq:expladd}. To see the inclusion in the opposite direction we note  that $\Nil\big(\LL^n(A)\big) \subset \LL^n(A)^{\sharp}$.
\end{proof}

The proof of the following statement uses Remark~\ref{rmk:multtop} and copies the proof of~\cite[Ch.\,4, Sec.\,5, Theor.\,2]{BS} (see also~\cite[Lem.\,1.2]{FVZ}).

\begin{lemma}\label{lemma:convdom}
Let $\varphi\in A[[x_1,\ldots,x_r]]$ and $\psi_i\in A[[y_i]]$, $1\leqslant i\leqslant r$, be power series such that the constant terms of all $\psi_i$ equal zero. Let $\phi\in A[[y_1,\ldots,y_r]]$ be the formal composition $\varphi\big(\psi_1,\ldots,\psi_r\big)$. Then for all elements $f_1,\ldots,f_r\in \LL^n(A)^{\sharp}$, there is an equality in~$\LL^n(A)$
$$
\varphi\big(\psi_1(f_1),\ldots,\psi_r(f_r)\big)=\phi(f_1,\ldots,f_r)\,.
$$
\end{lemma}

\subsection{Differential forms}
\label{dif-forms}

Consider absolute K\"ahler differentials $\Omega^1_{\LL^n(A)}$ for the ring~$\LL^n(A)$. By $K\subset \Omega^1_{\LL^n(A)}$ denote the~$\LL^n(A)$-submodule generated by all elements $df-\sum\limits_{i=1}^n\frac{\partial{f}}{\partial{t_i}}dt_i$, where \mbox{$f\in\LL^n(A)$}. In particular, $K$ contains the elements $da$, where $a\in A$.

\begin{defin}\label{defin:tildeforms}
Define the quotient
$$
\widetilde{\Omega}^1_{\LL^n(A)}:=\Omega^1_{\LL^n(A)}/K\,.
$$
\end{defin}

It is easily shown that $\widetilde{\Omega}^1_{\LL^n(A)}$ is a free module over the ring $\LL^n(A)$ with the basis $dt_1,\ldots,dt_n$. Further, put
$\widetilde{\Omega}^i_{\LL^n(A)}:=\bigwedge_{\LL^n(A)}^{i}\widetilde{\Omega}^1_{\LL^n(A)}$.

\medskip

One checks directly that the de Rham differential $d\colon\Omega^1_{\LL^n(A)}\to \Omega^2_{\LL^n(A)}$ sends $K$ to~\mbox{$K\wedge\Omega^1_{\LL^n(A)}$}. Therefore we have a well-defined de Rham differential \mbox{$d\colon \widetilde{\Omega}^i_{\LL^n(A)}\to\widetilde{\Omega}^{i+1}_{\LL^n(A)}$} for any $i\geqslant 0$ (which we denote by the same letter ``$d$''),
and, obviously, the following diagram is commutative:
$$
\begin{CD}
{\Omega}^i_{\LL^n(A)}  @>{d}>>  {\Omega}^{i+1}_{\LL^n(A)}  \\
 @VV V @VVV \\
 \widetilde{\Omega}^i_{\LL^n(A)} @>{d}>> \widetilde{\Omega}^{i+1}_{\LL^n(A)}
\end{CD}
$$
where the vertical arrows are the natural maps.

\medskip

The topology on $\LL^n(A)$ defined in Subsection~\ref{subsect:topology} naturally induces a topology on each free $\LL^n(A)$-module~$\widetilde{\Omega}^i_{\LL^n(A)}$, $i\geqslant 0$. One easily checks that the de Rham differential is continuous with respect to this topology. Therefore for any element $f\in\LL^n(A)^{\sharp}$ and a power series $\varphi\in A[[x]]$, there is an equality in $\widetilde{\Omega}^1_{\LL^n(A)}$
\begin{equation}\label{eq:varphi}
\frac{\partial\varphi}{\partial x}(f)df=d(\varphi(f))\,.
\end{equation}

Let $\log(1+x):=\sum\limits_{i\geqslant 1}(-1)^{i+1}\frac{x^i}{i}$ and $\exp(x):=\sum\limits_{i\geqslant 0}\frac{x^i}{i!}$ be the usual power series from the ring $\Q[[x]]$. Let $\Li_2(x):=\sum\limits_{i\geqslant 1}\frac{x^i}{i^2}   \in \Q[[x]]$ be the dilogarithm.

The next lemma is directly implied by formula~\eqref{eq:varphi}.

\begin{lemma} \label{dif-form}
Let $A$ be a $\Q$-algebra. There are the following equalities in $\widetilde{\Omega}^1_{\LL^n(A)}$.
\begin{itemize}
\item[(i)]
For any $f\in \LL^n(A)^{\sharp}$, there is an equality \mbox{$\frac{df}{1+f}= \frac{d(1+f)}{1+f}=d\log(1+f)$}.
\item[(ii)]
For any $f\in\LL^n(A)^{\sharp}\cap \LL^n(A)^*$, there is an equality
$$
\log(1-f)\frac{df}{f}=d\,(-\Li_2(f))\,.
$$
\item[(iii)]
For any $f\in\LL^n(A)^{\sharp}$ and $\psi\in A[[x]]^*$, there exists $\varphi\in A[[x]]$ such that
$$
\log(1-f)\frac{df}{\psi(f)}=d(\varphi(f))\,.
$$
\end{itemize}
\end{lemma}

\medskip

There is a homomorphism
\begin{equation}  \label{res}
\res\;:\; \widetilde{\Omega}^n_{\LL^n(A)}\lrto A\,,\qquad \mbox{$\sum\limits_{l\in\z^n}a_lt^{l}\cdot dt_1\wedge\ldots \wedge dt_n\longmapsto a_{-1\ldots-1}$}\,,
\end{equation}
called a {\it residue map}. For simplicity, we also use notation  ``$\res$'' for the map from $\Omega^n_{\LL^n(A)}$ to $A$ which is the composition of the natural map $\Omega^n_{\LL^n(A)}  \to \widetilde{\Omega}^n_{\LL^n(A)}$  with the map $\res$ defined by formula~\eqref{res}.

For any element $\eta\in\widetilde{\Omega}^{n-1}_{\LL^n(A)}$, we have $\res(d\eta)=0$. Moreover we have the following fact.

\begin{lemma}\label{lemma:cohom}
For any $\Q$-algebra $A$, there is an isomorphism
$$
\res\;:\;\widetilde{\Omega}^{n}_{\LL^n(A)}/d\widetilde{\Omega}^{n-1}_{\LL^n(A)}\stackrel{\sim}\longrightarrow A\,.
$$
\end{lemma}
\begin{proof}
The proof is based on the following observation: $d$ is continuous and if $l_i\ne -1$ for some $i$, $1\leqslant i\leqslant n$, then there is an equality
$$
t_1^{l_1}\ldots t_n^{l_n}dt_1\wedge\ldots\wedge dt_n=d\left(\frac{(-1)^{i-1}}{l_i+1} t_1^{l_1}\ldots t_i^{l_i+1}\ldots t_n^{l_n}dt_1\wedge\ldots dt_{i-1}\wedge dt_{i+1}\wedge\ldots\wedge dt_n\right)\,.
$$
\end{proof}

\section{Iterated loop group of $\gm$}\label{multiplicative}

\subsection{Iterated loop functors}\label{subsection:loopgroup}


\begin{defin}\label{defin:loopgroup}
For a functor $F$ (see Section~\ref{sect:notation}), a {\it loop functor} of $F$ is the functor~$\lo F$ given by the formula $LF(A):=F\big(A((t))\big)$ for a ring $A$. We call $\lo^n F:=\underbrace{\lo \ldots \lo}_n F$ an {\it $n$-iterated loop functor} of $F$  for a natural number $n$.  If $F$ is a group functor, then we also call $\lo F$ and~$\lo^n F$ a {\it loop group} and an {\it $n$-iterated loop group} of $F$, respectively.
\end{defin}

Explicitly, we have
$$
\lo^nF(A)=F\big(A((t_1))\ldots((t_n))\big)
$$
for a ring $A$. We have natural morphisms of functors
\begin{equation}   \label{loop-embed}
F \lrto \lo F\, , \qquad F\lrto L^n F
\end{equation}
that correspond to the maps
$$
F(A)  \lrto F(A((t))) \, , \qquad
F(A)\lrto F\big(A((t_1))\ldots((t_n))\big)
$$
given by constant series. For the additive group scheme $\ga = \Spec(\Z[x])$, we have that~$\lo^n\ga(A)$ is the additive group of the ring of iterated Laurent series $\LL^n(A)$. We are interested in the $n$-iterated loop group~$\lo^n\gm$ of the multiplicative group scheme~${\gm= \Spec(\Z[x,x^{-1}])}$. Explicitly, we have  that~$\lo^n\gm(A)$ is the multiplicative group $\LL^n(A)^*$ of invertible elements in the ring~$\LL^n(A)$.

\subsection{Decomposition}\label{subsection:decomp}

In order to describe the $n$-iterated loop group $\lo^n\gm$, we consider several group subfunctors in it. First we have an embedding of group functors
\begin{equation}\label{eq:constsubgr}
\gm\hookrightarrow\lo^n\gm
\end{equation}
given by constant series.

\begin{defin}\label{defin:uz}
Let $\uz$ be the group functor defined as follows: given a ring $A$, we have that $\uz(A)$ is the group of locally constant functions on $\Spec(A)$ with values in~$\z$.
\end{defin}

In other words, $\uz$ is the constant sheaf with respect to the Zariski topology associated with the group $\z$. Since affine schemes are quasi-compact, given a ring $A$, an element $\underline{l}\in\uz(A)$ determines a decomposition into a finite product of rings
$$
\mbox{$A\simeq\prod\limits_{i=1}^N A_i$}
$$
and a collection of integers $l_i$, $1\leqslant i\leqslant N$, such that for every $i$, the restriction of $\underline{l}$ to~$\Spec(A_i)$ is the constant function whose value equals~$l_i$.

Analogous facts hold for the group functor $\uz^n$. We have an embedding of group functors
\begin{equation}\label{eq:discrsubgr}
\uz^n\hookrightarrow\lo^n\gm
\end{equation}
that sends $\underline{l}\in\uz^n(A)$ to $t^{\underline l}$ for a ring $A$, where we put
$$
\mbox{$t^{\underline l}:=(t_1^{l_{11}}\cdot\ldots\cdot t_n^{l_{n1}},\,\ldots,\,t_1^{l_{1M}}\cdot\ldots\cdot t_n^{l_{nM}})\in \LL^n(A)\simeq\prod\limits_{i=1}^M \LL^n(A_i)$}
$$
and a decomposition $A\simeq\prod\limits_{i=1}^M A_i$ is such that for every $i$, $1\leqslant i\leqslant M$, the restriction of~$\underline{l}$ to $\Spec(A_i)$ is the constant function whose value equals $(l_{1i},\ldots,l_{ni})\in\z^n$.

\medskip

Further, consider the following {\it lexicographical order} on $\z^n$: put $(l_1,\ldots,l_n)\leqslant (l'_1,\ldots,l'_n)$ if and only if either $l_n < l'_n$ or $l_n=l'_n$ and $(l_1, \ldots, l_{n-1}) \leqslant (l'_1, \ldots, l'_{n-1})$. Note that the order is invariant under translations on the group $\z^n$. We abbreviate $(0,\ldots,0)$ to $0$. Define the group functors
\begin{equation}\label{eq:v+}
\vv_{n,+}(A):=\Big\{\mbox{$1+\sum\limits_{0<l\in\z^n} a_{l}t^{l}\;\mid\;  \sum\limits_{l>0} a_{l}t^{l} \in \LL^n(A) $} \Big\}\,,
\end{equation}
\begin{equation}\label{eq:v-}
\vv_{n,-}(A):=\Big\{\mbox{$1+\sum\limits_{0>l\in\z^n} a_{l}t^{l}\;\mid\; \sum\limits_{l<0} a_lt^{l} \in \Nil\big(\LL^n(A)\big)$}\Big\}\,.
\end{equation}
The group structure on $\vv_{n,+}$ and $\vv_{n,-}$ is given by the product of iterated Laurent series and we have embeddings of group functors
\begin{equation}\label{eq:contsubgr}
\vv_{n,+}\hookrightarrow\lo^n\gm\,,\qquad \vv_{n,-}\hookrightarrow\lo^n\gm\,.
\end{equation}

\begin{prop}\label{prop-decomp}
Embeddings~\eqref{eq:constsubgr},~\eqref{eq:discrsubgr}, and~\eqref{eq:contsubgr} induce an isomorphism of group functors
\begin{equation}\label{eq:decomCC}
\uz^n\times\gm\times\vv_{n,+}\times\vv_{n,-}\simeq\lo^n\gm\,.
\end{equation}
\end{prop}
\begin{proof}
We use induction on $n$. The base $n=1$ is proved by Contou-Carr\`ere~\cite[Lem.\,1.3]{CC1},~\cite[Lem.\,0.8]{CC2}. For the induction step, we consider loop functors of all functors in formula~\eqref{eq:decomCC} with $n$ being replaced by~$n-1$. Then we use the isomorphism  $ \uz \simeq \lo\uz$ induced by morphisms~\eqref{loop-embed} and  proved by the second named author and Zhu~\cite[Lem.\,3.2]{OZ1}, and we use also the isomorphisms
\begin{equation}\label{eq:isomiter}
\vv_{1,+}\times\lo\vv_{{n-1},+}\simeq\vv_{n,+}\,,\qquad \vv_{1,-}\times\lo\vv_{n-1,-}\simeq\vv_{n,-}\,.
\end{equation}
 Here, given a ring $A$, elements in $\vv_{1,+}(A)$ and $\vv_{1,-}(A)$ are Laurent series in $t_1$ with coefficients in $A$ that satisfy conditions from~\eqref{eq:v+} and~\eqref{eq:v-}, respectively, with $n=1$. Similarly, elements in $\lo\vv_{n-1,+}(A)$ and~$\lo\vv_{n-1,-}(A)$ are iterated Laurent series in $t_2,\ldots,t_n$ with coefficients in $A((t_1))$ that satisfy conditions from~\eqref{eq:v+} and~\eqref{eq:v-} with $n$ being replaced by $n-1$ and $A$ being replaced by $A((t_1))$. Isomorphisms~\eqref{eq:isomiter} can be checked directly with the help of the following obvious fact: given a ring~$B$ and subgroups $P,Q\subset B^*$ such that $P\cap Q=\{1\}$ and $P\cdot (Q-1)=Q-1$, we have the isomorphism
$$
P\times Q\stackrel{\sim}\longrightarrow P+Q-1\,,\qquad (p,q)\longmapsto p\cdot q=p+p\cdot(q-1)\,.
$$
We apply this fact to $B=\LL^n(A)$, $P=\vv_{1,+}(A)$, \mbox{$Q=\lo\vv_{n-1,+}(A)$} and \mbox{$P=\vv_{1,-}(A)$}, \mbox{$Q=\lo\vv_{n-1,-}(A)$}.
\end{proof}

\begin{rmk}
For $n=2$, Proposition~\ref{prop-decomp} was proved in~\cite[\S\,3.1]{OZ1}.
\end{rmk}

Decomposition~\eqref{eq:decomCC} defines projections
\begin{equation}\label{eq:projections}
\nu\;:\;\lo^n\gm\lrto\uz^n\,,\qquad \pi\;:\;\lo^n\gm\lrto\gm\,.
\end{equation}

\begin{examp}\label{examp:projections}
\hspace{0cm}
\begin{itemize}
\item[(i)]
Suppose that for an invertible iterated Laurent series \mbox{$f=\sum\limits_{l\in\z^n} a_lt^{l}\in\lo^n\gm(A)$}, one has that $\nu(f)\in \z^n\subset \uz^n(A)$, that is, $\nu(f)$ is a constant function on $\Spec(A)$. For instance, this is true when $\Spec(A)$ is connected. Then $f$ has an invertible coefficient $a_l\in A^*$ for some~$l\in\z^n$ and $\nu(f)\in\z^n$ is the smallest index of an invertible coefficient with respect to the lexicographical order on $\z^n$. Moreover, all coefficients of~$f$ with smaller indices than $\nu(f)$ are nilpotent elements in $A$.
\item[(ii)]
If $A=k$ is a field and $n=1$, then $\nu$ is the natural discrete valuation on the field~$\LL(A)=k((t))$ and $\pi$ sends a non-zero series to the first non-zero coefficient.
\item[(iii)]
Let $\varepsilon$ be a formal variable that satisfies $\varepsilon^2=0$, take $A=\Z[\varepsilon]$, and consider the invertible Laurent series in $\LL(A)$
$$
f=(\varepsilon t^{-1}+1)\cdot (1+t)=\varepsilon t^{-1}+(1+\varepsilon)+t\,.
$$
Then $\nu(f)=0$ and $\pi(f)=1$, but the constant term of $f$ is not $1$.
\end{itemize}
\end{examp}

\subsection{Special subgroups} \label{sp-sub}

We have group functors $\Nil$ and $1+\Nil$, where the group structures are given by the sum and the product of elements in a ring, respectively. In particular, there is an embedding of group functors $(1+\Nil)\hookrightarrow\gm$.

\begin{defin}\label{defin:specialmult}
Define the following group subfunctors of $\lo^n\gm$ (see formulas~\eqref{eq:v+},~\eqref{eq:v-}, and~\eqref{eq:projections}):
$$
(\lo^n\gm)^{0}:=\Ker(\nu)=\gm\times\vv_{n,+}\times\vv_{n,-}\,,
$$
$$
(\lo^n\gm)^{\sharp}:=(1+\Nil)\times\vv_{n,+}\times\vv_{n,-}\,.
$$
\end{defin}

Evidently, $(\lo^n\gm)^{\sharp}$ is embedded into $(\lo^n\gm)^0$, its intersection with $\gm$ is the group functor $1+\Nil$,  and we have that $(\lo^n\gm)^0=\gm\cdot(\lo^n\gm)^{\sharp}$. By Proposition~\ref{prop-decomp}, we have decompositions
\begin{equation}\label{eq:decommult}
\uz^n\times \big(\gm\cdot(\lo^n\gm)^{\sharp}\big)=\uz^n\times(\lo^n\gm)^0\simeq\lo^n\gm\,.
\end{equation}

Here is a more explicit description of the group functors $(\lo^n\gm)^{0}$ and $(\lo^n\gm)^{\sharp}$.

\begin{lemma}  \label{expl-form}
Let $A$ be a ring.
\begin{itemize}
\item[(i)]
For any element $f\in(\lo^n\gm)^{0}(A)$,  the constant term of~$f$ is invertible and its class in $A^*/\big(1+\Nil(A)\big)$ is equal to the class of $\pi(f)\in A^*$ (cf. Example~\ref{examp:projections}(iii)).
\item[(ii)]
There is an equality of group functors
$$
(\lo^n\gm)^{0}(A)=\Big\{\mbox{$1+\sum\limits_{l\in\z^n} a_{l}t^{l}\;\mid\;
{ 1+a_0\in A^*,\,\sum\limits_{l>0} a_{l}t^{l} \in \LL^n(A),}\,\sum\limits_{l< 0} a_lt^{l} \in \Nil\big(\LL^n (A)\big)$}\Big\}\,.
$$
\item[(iii)]
There is an equality of group functors
$$
(\lo^n\gm)^{\sharp}(A)=\Big\{\mbox{$1+\sum\limits_{l\in\z^n} a_{l}t^{l}\;\mid\;
{ a_0\in \Nil(A),\,\sum\limits_{l > 0} a_{l}t^{l} \in \LL^n(A),}\,\sum\limits_{l<0} a_lt^{l} \in \Nil\big(\LL^n (A)\big)$}\Big\}\,.
$$
\end{itemize}
\end{lemma}
\begin{proof}
$(i)$
Open brackets in decomposition~\eqref{eq:decomCC}.

$(ii)$
Note that an iterated Laurent series ${\sum\limits_{l\in\z^n}b_lt^l\in\LL^n(A)}$ is nilpotent if and only if the constant term~\mbox{$b_0\in \Nil(A)$ and the iterated Laurent series $\sum\limits_{l>0}b_lt^l$ and $\sum\limits_{l<0}b_lt^l$ are  nilpotent} (this is easily proved by induction on $n$ with the help of Remark~\ref{rmk:nilpseries}). Using this, one shows that the left hand side is contained in the right hand side. Then  one checks that the right hand side is contained in a coset of~$(\gm\times\vv_{n,+}\times\vv_{n,-})(A)$ in~$\lo^n\gm(A)$, since, by Proposition~\ref{prop-decomp}, the subgroup $\uz^n(A)$ is a transversal for the cosets of {${ (\gm\times\vv_{n,+}\times\vv_{n,-})  (A)}$}.

$(iii)$ This follows directly from items~$(i)$ and~$(ii)$.
\end{proof}

\quash{We will also use the following subfunctor in $L^n\gm$:
\begin{equation}\label{eq:bekar}
(\lo^n\gm)^{\natural}:=\vv_{n,+}\times\vv_{n,-}\,.
\end{equation}
Explicitly, for any ring $A$, we have that $(\lo^n\gm)^{\natural}(A)$ is the set of all iterated Laurent series in $(L^n\gm)^0(A)$ such that there constant term is equal to one. Note that $(\lo^n\gm)^{\natural}$ is not a group functor. We have decompositions of functors (which do not respect a group structure)
\begin{equation}\label{eq:decombekar}
\uz^n\times\gm\times(\lo^n\gm)^{\natural}\simeq\lo^n\gm\,,\\
\gm\times(\lo^n\gm)^{\natural}\simeq(\lo^n\gm)^0\,,\\
(1+\Nil)\times(\lo^n\gm)^{\natural}\simeq(\lo^n\gm)^{\sharp}\,.
\end{equation}}

\begin{defin}\label{defin:sharpaddfunctor}
Define the following group subfunctor of $\lo^n\ga$ (see~Definition~\ref{defin:addsharp} and Proposition~\ref{prop:analsharp}):
$$
(\lo^n\ga)^{\sharp}(A):=\LL^n(A)^{\sharp}=\Big\{\mbox{$\sum\limits_{l\in\z^n} a_{l}t^{l}\;\mid\;
\sum\limits_{l >0} a_{l}t^{l} \in \LL^n(A)
,\,\sum\limits_{l\leqslant 0} a_lt^{l} \in \Nil\big(\LL^n(A)\big)$}\Big\}\,.
$$
\end{defin}

Obviously, there is an isomorphism of functors $(\lo^n\gm)^{\sharp}\simeq(\lo^n\ga)^{\sharp}$, $f\mapsto f-1$, which does not respect the group structure. However over $\Q$ there is a group isomorphism, given in the following proposition.
\quash{Namely, let $\log(1+x):=\sum\limits_{i\geqslant 1}(-1)^{i+1}\frac{x^i}{i}$ and $\exp(x):=\sum\limits_{i\geqslant 0}\frac{x^i}{i!}$ be the usual power series from the ring $\Q[[x]]$.}

\begin{prop}\label{log-map}
There is a well-defined isomorphism of group functors
$$
\log\;:\; (\lo^n\gm)^{\sharp}_{\Q}\stackrel{\sim}\longrightarrow (\lo^n\ga)^{\sharp}_{\Q}
$$
whose inverse is given by $\exp$.
\end{prop}
\begin{proof}
This follows directly from Lemma~\ref{lemma:convdom}, because there are equalities of power series
$$
\exp\big(\log(1+x)\big)=1+x\,,\qquad \log\big(\exp(x)\big)=x\,,
$$
$$
\log(1+x+y+xy)=\log(1+x)+\log(1+y)\,,\qquad\exp(x+y)=\exp(x)\cdot\exp(y)\,.
$$
\end{proof}

\section{Auxiliary results on ind-affine schemes}
\label{representab}

Let $R$ be a ring.

\subsection{Ind-schemes}\label{subsect:indschemes}

Recall some general notions related to ind-schemes. Given a category $\Cc$, one has the category of ind-objects in $\Cc$,~\cite{GV},~\cite{AM}. Explicitly, an ind-object in $\Cc$ is given by a directed partially ordered set $I$, a collection of objects~$C_i$ in~$\Cc$ for all $i\in I$, and a collection of compatible morphisms $C_i\to C_j$ for all $i\leqslant j$. Such an ind-object is denoted by $\mbox{``$\varinjlim\limits_{i\in I}$''}C_i$.

\medskip

Ind-objects in the category of schemes are called {\it ind-schemes}. An {\it ind-affine scheme} is an ind-scheme $X$ which is isomorphic to $\mbox{``$\varinjlim\limits_{i\in I}$''}X_i$ with all $X_i$, $i\in I$, being affine schemes. Similarly, we have the notions of ind-schemes and ind-affine schemes over $R$.

Given an ind-scheme $X$, we denote the functor that sends a ring $A$ to the set ${X(A):=\Hom\big(\Spec(A),X\big)}$ also by $X$. Note that $X(A)$ equals $\varinjlim\limits_{i\in I}X_i(A)$ for an ind-scheme $X=\mbox{``$\varinjlim\limits_{i\in I}$''}X_i$. A morphism between ind-schemes is the same as a morphism between the corresponding functors on the category of rings.

\medskip

For example, a discussion after Definition~\ref{defin:uz} implies that the group functor $\uz$ is represented by the ind-affine scheme $\mbox{``$\varinjlim\limits_{N\geqslant 0}$''}\,\coprod\limits_{-N\leqslant i\leqslant N}\Spec(\Z)$\,.

\medskip

An {\it ind-closed subscheme of a scheme} $V$ is an ind-object in the category of closed subschemes of $V$ with morphisms being embeddings between closed subschemes. Explicitly, an ind-closed subscheme of $V$ is given by a directed partially ordered set $I$ and a collection of closed subschemes $X_i\subset V$ for all $i\in I$ such that $X_i\subset X_j$ if~$i\leqslant j$. In particular, an ind-closed subscheme in an affine scheme is an ind-affine scheme.

An {\it ind-closed subscheme of an ind-scheme} $Y\simeq \mbox{``$\varinjlim\limits_{j\in J}$''}Y_j$ is an ind-scheme ${X\simeq \mbox{``$\varinjlim\limits_{i\in I}$''}X_i}$ together with a morphism of ind-schemes $X\to Y$ such that for any $i\in I$, there is $j\in J$ and there is a commutative diagram
$$
\begin{CD}
X_i  @>>>  Y_j  \\
 @VV V @VVV \\
X @>>> Y
\end{CD}
$$
with the top horizontal arrow being a closed embedding of schemes and the vertical arrows being the natural morphisms.

In particular, if $Y$ is an ind-closed subscheme of a scheme $V$, then an ind-closed subscheme of $Y$ is the same as an ind-closed subscheme $X\simeq \mbox{``$\varinjlim\limits_{i\in I}$''}X_i$ of $V$ such that for any $i\in I$, there is $j\in J$ that satisfies $X_i\subset Y_j$. Equivalently, we have a morphism $X\to Y$ in the category of ind-closed subschemes of $V$.

\medskip

Given two ind-closed subschemes ${X\simeq \mbox{``$\varinjlim\limits_{i\in I}$''}X_i}$ and ${Y\simeq\mbox{``$\varinjlim\limits_{j\in J}$''}Y_j}$ of a scheme $V$, {\it the intersection} $X\cap Y$ is the ind-closed subscheme $\mbox{``$\varinjlim\limits_{(i,j)\in I\times J}$''}X_i\cap Y_j$ of $V$. Here we take a schematic intersection of two closed subschemes of $V$, that is, the intersection is given by the ideal sheaf generated by the ideal sheaves of two closed subschemes. More generally, one can also consider the intersection of two ind-closed subschemes in an ind-scheme, though we will not use it.

\medskip

Given an ind-scheme $X$, the ring of {\it regular functions} on $X$ is defined by the formula $\OO(X):=\Hom(X,\Ab^1)$. In particular, $V=\Spec\big(\OO(V)\big)$ for an affine scheme $V$. Explicitly, for an ind-scheme $X\simeq \mbox{``$\varinjlim\limits_{i\in I}$''}X_i$, there is an isomorphism \mbox{$\OO(X)\simeq\varprojlim\limits_{i\in I}\OO(X_i)$}. This isomorphism is well-behaved with respect to morphisms of ind-schemes. A morphism of ind-schemes $\alpha\colon X\to Y$ induces a homomorphism of rings $\alpha^*\colon\OO(Y)\to\OO(X)$. For an ind-scheme $X$ over $R$, the ring $\OO(X)$ is naturally an $R$-algebra.

\medskip

One represents regular functions on a closed subscheme of an affine space by polynomials in coordinates. This is no more true for an ind-closed subscheme of an affine space. In this case, it is natural to represent regular functions by power series in coordinates. For example, for the ind-closed subscheme $X=\mbox{``$\varinjlim\limits_{d\in \N}$''}\Spec\big(\Z[x]/(x^d)\big)$ of $\Ab^1=\Spec\big(\Z[x]\big)$, we have an isomorphism $\OO(X)\simeq \Z[[x]]$. In Subsection~\ref{subsect:thickindcone}, we generalize this for a certain class of ind-closed subschemes of ind-affine spaces.

\subsection{Algebraic convergence of power series}\label{subsect:convalg}

Let us introduce some notation concerning polynomials and power series in infinitely many variables. Let $M$ be a (possibly, infinite) set. By~$R[M]$ denote the algebra of polynomials over $R$ in formal variables $x_m$ that correspond bijectively to elements $m\in M$. We also denote this algebra by \mbox{$R[x_m;\,m\in M]$} if needed to specify the formal variables. Note that the set of monomials in $x_m$, $m\in M$, is in bijection with the $d$-th symmetric power $\Sym^d(M)$ of the set $M$, that is, with the set of unordered $d$-tuples of elements in~$M$. Thus there is an isomorphism
$$
\mbox{$R[M]\simeq \bigoplus\limits_{d\geqslant 0}R^{\oplus\Sym^d(M)}$}\,,
$$
where for a set $P$, by $R^{\oplus P}$ we denote the group of all functions with finite support from~$P$ to~$R$.

We say that the affine scheme $\Ab^M_R:=\Spec\big(R[M]\big)$ is an {\it affine space} over $R$. Thus the formal variables~$x_m$,~$m\in M$, are interpreted as coordinates on the affine space $\Ab^M_R$. For short, we denote $\Ab^M_R$ as $\Ab^M$ if it is clear over which ring the affine space is considered. In what follows, we consider affine spaces over $R$. Given a subset \mbox{$M'\subset M$}, we have a closed embedding $\Ab^{M'}\subset \Ab^M$ such that the corresponding epimorphism \mbox{$R[M]\to R[M']$} is identical on $M'$ and vanishes on $M\smallsetminus M'$. Let \mbox{$(M)\subset R[M]$} denote the ideal generated by all $x_m$, $m\in M$.

\begin{defin}\label{defin:powerseries}
An algebra of {\it power series over $R$ in formal variables $x_m$, $m\in M$,} is defined by the formula
$$
R[[M]]:=\varprojlim\limits_{(M',d)}R[M']/(M')^d\,,
$$
where $M'$ runs over all finite subsets of $M$ and $d\in\N$.
\end{defin}

We also denote this algebra by \mbox{$R[[x_m;\,m\in M]]$}. If~$M$ is at most countable, then elements in $R[[M]]$ are countable sums of pair-wise different monomials in $x_m$, $m\in M$, with coefficients in $R$. For example, the infinite sums $\sum\limits_{m\in M}x_m$, $\sum\limits_{d\geqslant 0}x^d$, and $\sum\limits_{m\in M,\,d\geqslant 0}x_m^d$ are power series, while the infinite sum $x+2x+3x+\ldots$ is not. For a general set $M$, elements in~$R[[M]]$ are transfinite sums of such monomials. In any case, there are isomorphisms
\begin{equation}\label{eq:series}
\mbox{$R[[M]]\simeq \prod\limits_{d\geqslant 0}\varprojlim\limits_{M'}\,R^{\oplus\Sym^d(M')}\simeq
\prod\limits_{d\geqslant 0}R^{\Sym^d(M)}$}\,,
\end{equation}
where, as above, $M'$ runs over all finite subsets of $M$ and for a set $P$, by $R^{P}$ we denote the group of all functions from $P$ to $R$. Thus a power series $\varphi\in R[[M]]$ is decomposed into a countable sum of {\it homogenous power series}: \mbox{$\varphi=\sum\limits_{d\geqslant 0}\varphi_d$}, $\varphi_d\in R^{\Sym^d(M)}$. A {\it support of a power series $\varphi$} is the set of monomials that have a non-zero coefficient in~$\varphi$. In other words, the support is a subset of $\coprod\limits_{d\geqslant 0}\Sym^d(M)$ that corresponds to non-zero coefficients in $\varphi$ with respect to decomposition~\eqref{eq:series}.

\medskip

\begin{defin}\label{defin-conv}
\hspace{0cm}
\begin{itemize}
\item[(i)]
Let $A$ be an $R$-algebra and let $d\geqslant 0$ be a natural number. A {\it homogenous power series $\varphi_d\in R^{\Sym^d(M)}$ converges algebraically at an $A$-point} $p\in\Ab^M(A)=A^M$ if all but finitely many monomials in the support of $\varphi_d$ vanish at $p$. If this holds, then~$\varphi_d(p)$ is a well-defined element in the algebra $A$.
\item[(ii)]
An arbitrary {\it power series $\varphi\in R[[M]]$ converges algebraically at an $A$-point} $p$ as in item~(i) if all its homogenous components $\varphi_d$, $d\geqslant 0$, converge algebraically at $p$ and there is $d_0\in\N$ such that $\varphi_d(p)=0$ when $d\geqslant d_0$. If this holds, then $\varphi(p)=\sum\limits_{d<d_0}\varphi_d(p)$ is a well-defined element in the algebra~$A$.
\item[(iii)]
A power series $\varphi$ {\it converges algebraically on an ind-closed subscheme} $X\subset\Ab^M$ if $\varphi$ converges algebraically at all its $A$-points for any $R$-algebra $A$. By~$\Ac(X)$ denote the set of all power series in $R[[M]]$ that converge algebraically on~$X$.
\end{itemize}
\end{defin}

\begin{examp}\label{examp:convalg}
\hspace{0cm}
\begin{itemize}
\item[(o)]
Any polynomial converges algebraically at any point.
\item[(i)]
Any power series converges algebraically at a point with finitely many non-zero coordinates all of which are nilpotent.
\item[(ii)]
The series $\sum\limits_{m\in M}x_m$ converges algebraically at a point if and only if all but finitely many coordinates of the point are equal to zero.
\item[(iii)]
The series $\sum\limits_{d\geqslant 0}(x_1^d-x_2^d)$ converges algebraically (and, actually, vanishes) at the point \mbox{$(1,1)\in\Ab^2$}, while the series $\sum\limits_{d\geqslant 0}(-1)^dx_1^d$ does not.
\item[(iv)]
For a closed subscheme $V\subset \Ab^M$, a power series converges algebraically on $V$ if and only if it converges algebraically at the point ${p\in \Ab^M\big(\OO(V)\big)}$ that corresponds to the restriction homomorphism ${\OO(\Ab^M)\to \OO(V)}$.
\item[(v)]
For a ind-closed subscheme $X=\mbox{``$\varinjlim\limits_{i\in I}$''}X_i$ of $\Ab^M$, a power series converges algebraically on $X$ if and only if it converges algebraically at all points ${p_i\in \Ab^M\big(\OO(X_i)\big)}$ that correspond to the restriction homomorphisms ${\OO(\Ab^M)\to \OO(X_i)}$, $i\in I$.
\end{itemize}
\end{examp}

Given an ind-closed subscheme $X\subset\Ab^M$, the set $\Ac(X)$ is an $R$-algebra and we have a canonical homomorphism of $R$-algebras
\begin{equation}\label{eq:evalmap}
\Ac(X)\lrto\OO(X)
\end{equation}
given by evaluation of algebraically convergent series at points of $X$. This homomorphism being an isomorphism means that any regular function on $X$ is uniquely expanded as an algebraically convergent power series in $x_m$, $m\in M$. For example, this holds for the affine space $\Ab^M$.

\subsection{Ind-affine spaces}\label{subsection:indaffspace}

As above, we continue considering affine spaces over $R$. We shall work with ind-closed subchemes not only in affine spaces but also in ind-affine spaces. For this goal, we need to introduce one more notion. First, recall that an ind-set is an ind-object in the category of sets.

\begin{defin}\label{defin:strictset}
A {\it strict ind-set} is an ind-set $\mbox{``$\varinjlim\limits_{i\in I}$''}E_i$ such that all structure maps $E_i\to E_j$, $i\leqslant j$, are injective. Given a strict ind-set $E=\mbox{``$\varinjlim\limits_{i\in I}$''}E_i$, by $\overline{E}$ denote the set~$\bigcup\limits_{i\in I}E_i$ and by $E_f$ denote the strict ind-set formed by all finite subsets of~$\overline{E}$.
\end{defin}

In particular, any set is also a strict ind-set.

\medskip

In what follows, $E=\mbox{``$\varinjlim\limits_{i\in I}$''}E_i$ is a strict ind-set. We call the ind-affine scheme ${\Ab^E:=\mbox{``$\varinjlim\limits_{i\in I}$''}\Ab^{E_i}}$ an {\it ind-affine space}. Clearly, $\Ab^E$ is an ind-closed subscheme of the affine space $\Ab^{\overline{E}}$ and there is an isomorphism of $R$-algebras
\begin{equation}\label{eq:isomAE}
\mbox{$\OO(\Ab^E)\simeq\varprojlim\limits_{i\in I}\bigoplus\limits_{d\geqslant 0}R^{\oplus\Sym^d(E_i)}$}\,.
\end{equation}

\medskip

\begin{defin}\label{defin:hat}
Given a set $M$ and a natural number $d$, define the following closed subscheme of $\Ab^M$:
$$
\mbox{$\Ab^M_{(d)}:=\Spec\big(R[M]/(M)^{d}\big)$}\,.
$$
Define also the following ind-closed subscheme of the ind-affine space $\Ab^E$:
$$
\widehat{\Ab}^E:=\mbox{``$\varinjlim\limits_{(i,d)}$''}\Ab^{E_i}_{(d)}\,,
$$
where \mbox{$(i,d)\in I\times \N$}.
\end{defin}

In particular, for a set $M$, we have that $\widehat{\Ab}^M=\mbox{``$\varinjlim\limits_{d\in\N}$''}\Ab^{M}_{(d)}$ and there is an isomorphism of $R$-algebras
$$
\mbox{$\OO(\widehat{\Ab}^M)\simeq\prod\limits_{d\geqslant 0}R^{\oplus\Sym^d(M)}$}\,.
$$
Further, we have the equalities
\begin{equation}\label{eq:Ehat}
\widehat{\Ab}^E=\mbox{``$\varinjlim\limits_{i\in I}$''}\widehat{\Ab}^{E_i}=\mbox{``$\varinjlim\limits_{d\geqslant 0}$''}\Ab^{E}_{(d)}\,,
\end{equation}
where for each natural number $d$, we put $\Ab^E_{(d)}:=\mbox{``$\varinjlim\limits_{i\in I}$''}\Ab^{E_i}_{(d)}$, and there are isomorphisms of $R$-algebras
\begin{equation}\label{eq:functionas}
\mbox{$\OO(\widehat{\Ab}^E)\simeq\varprojlim\limits_{i\in I}\prod\limits_{d\geqslant 0}R^{\oplus\Sym^d(E_i)}\simeq
\prod\limits_{d\geqslant 0}\varprojlim\limits_{i\in I}R^{\oplus\Sym^d(E_i)}$}\,.
\end{equation}
Also, $\Ab^{E_f}$ is an ind-closed subscheme of the ind-affine scheme $\Ab^{E}$ and there is an isomorphism of $R$-algebras
\begin{equation}\label{eq:Ef}
\OO(\widehat{\Ab}^{E_f})\simeq R[[\overline{E}]]\,.
\end{equation}
Explicitly, given an $R$-algebra $A$, a point $p\in \Ab^{\overline{E}}(A)=A^{\overline{E}}$ belongs to the ind-closed subscheme $\Ab^{E_f}\subset \Ab^{\overline{E}}$ if and only if $p$ has finitely many non-zero coordinates all of which are nilpotent.

\begin{rmk}\label{rmk:exampAO}
Formulas~\eqref{eq:isomAE} and~\eqref{eq:functionas} imply that $\OO(\Ab^E)$ and $\OO(\widehat{\Ab}^E)$ are subalgebras of the algebra of power series $R[[\overline{E}]]$. Moreover, there are equalities $\Ac(\Ab^E)=\OO(\Ab^E)$ and $\Ac(\widehat{\Ab}^E)=\OO(\widehat{\Ab}^E)$, where $\Ac(\Ab^E)$ and $\Ac(\widehat{\Ab}^E)$ are defined by means of the embeddings of~$\Ab^E$ and $\widehat{\Ab}^E$ into the affine space $\Ab^{\overline{E}}$. Thus the canonical homomorphism~\eqref{eq:evalmap} is an isomorphism for the ind-closed subschemes $\Ab^E$ and $\widehat{\Ab}^E$ of the affine space $\Ab^{\overline{E}}$.
\end{rmk}

\subsection{Thick ind-cones}\label{subsect:thickindcone}

As above, we continue considering affine spaces over $R$, $M$ is a set, and $E$ is a strict ind-set.

\begin{defin}\label{defin:indcone}
A closed subscheme $V\subset \Ab^M$ is a {\it cone} if $V$ corresponds to a homogenous ideal in $R[M]=R[x_m;\,m\in M]$, where all $x_m$ have degree~$1$. An ind-closed subscheme $X\subset\Ab^E$ is an {\it ind-cone} if there is an isomorphism $X\simeq\mbox{``$\varinjlim\limits_{i\in I}$''}X_i$ such that $X_i\subset \Ab^{\overline{E}}$ is a cone for all $i\in I$.
\end{defin}

One can show that $X$ as above is a cone if and only if $X$ is invariant under homotheties, that is, for any $R$-algebra $A$ and an element $a\in A$, the subset $X(A)\subset \Ab^{\overline{E}}(A)=A^{\overline{E}}$ is preserved under multiplication of all coordinates by $a$.

\medskip

The following lemma contains a crucial property of cones, which is of utmost importance for the main results of the paper.

\begin{lemma}\label{lemma-injcone}
Let $V\subset\Ab^M$ be a cone.
\begin{itemize}
\item[(i)]
The restriction homomorphism $\theta_V\colon \OO(V)\to \OO(V\cap \widehat{\Ab}^M)$ is injective.
\item[(ii)]
The following commutative diagram is a Cartesian square:
$$
\begin{CD}
\Ac(V)\cap\OO(\widehat{\Ab}^M)@>>> \OO(\widehat{\Ab}^M) \\
@VV V @VV{{\xi}_V}V \\
\OO(V) @>{\theta_V}>> \OO(V\cap \widehat{\Ab}^M)\,,
\end{CD}
$$
where the intersection in the left upper corner is taken in $R[[M]]$, the intersection in the right bottom corner is taken among ind-closed subschemes of $\Ab^M$, and $\xi_V$ is the restriction homomorphism.
\end{itemize}
\end{lemma}

\begin{rmk}
Explicitly, Lemma~\ref{lemma-injcone}$(ii)$ claims that a regular function $\varphi\in \OO(\widehat{\Ab}^M)$ viewed as a power series in $R[[M]]$ (see Remark~\ref{rmk:exampAO}) converges algebraically on~$V$ if and only if the restriction of~$\varphi$ to $V\cap \widehat{\Ab}^M$ extends to a regular function on $V$. Also, by Lemma~\ref{lemma-injcone}$(i)$, the latter regular function on $V$ is unique if it exists.
\end{rmk}

\begin{proof}[Proof of Lemma~\ref{lemma-injcone}]
$(i)$ Let $I_V\subset R[M]$ be the ideal that defines the closed subscheme $V\subset {\Ab}^M$. Since~$V$ is a cone, the ring $\OO(V) = R[M]/I_V$ is graded, that is, we have a decomposition ${\OO(V)\simeq\bigoplus\limits_{d\geqslant 0}A_d}$. Therefore the ring
$$
\OO(V\cap \widehat{\Ab}^M)\simeq \varprojlim_{d\in\N}\OO(V\cap\Ab^M_{(d)})\simeq
\varprojlim_{d\in\N} R[M]/ \big(I_V + (M)^d\big)
$$
is isomorphic to $\prod\limits_{d\geqslant 0}A_d$ and the homomorphism $\theta_V$ is nothing but the natural embedding \mbox{$\bigoplus\limits_{d\geqslant 0}A_d\hookrightarrow \prod\limits_{d\geqslant 0}A_d$}.

$(ii)$ Consider a series $\varphi\in\OO(\widehat{\Ab}^M)\subset R[[M]]$. Let $p\in\Ab^M\big(\OO(V)\big)$ be the point that corresponds to the restriction homomorphism $\OO(\Ab^M)\to \OO(V)$. By Example~\ref{examp:convalg}(iv), $\varphi$ converges algebraically on $V$ if and only if it converges algebraically at $p$. We see that each homogenous component $\varphi_d\in R^{\oplus\Sym^d(M)}$ of $\varphi$ converges algebraically at $p$. Hence~$\varphi$ converges algebraically at $p$ if and only if \mbox{$\varphi_d(p)\in A_d$} equals zero for sufficiently large~$d$. By item~$(i)$, the latter is equivalent to the fact that $\xi_V(\varphi)=\varphi(p) \in \prod\limits_{d\geqslant 0}A_d $ is in the image of the ring $\OO(V)\simeq\bigoplus\limits_{d\geqslant 0}A_d$ under the map $\theta_V$. This proves the lemma.
\end{proof}

\medskip

\begin{defin}\label{defin:dset}
An ind-closed subscheme $X\subset\Ab^E$ is {\it thick} if $X$ contains the ind-closed subscheme~$\widehat{\Ab}^E\subset\Ab^E$ (see Definition~\ref{defin:hat}).
\end{defin}

In particular,~$\Ab^{E_f}$ is thick in $\Ab^E$ if and only if $E_f\simeq E$.

\medskip

Here is the main property of thick ind-cones, which generalizes Remark~\ref{rmk:exampAO}.

\begin{prop}\label{prop-injectiveseries}
Let $X\subset\Ab^E$ be a thick ind-cone. Then the canonical homomorphism \mbox{$\Ac(X)\to\OO(X)$} is an isomorphism (see formula~\eqref{eq:evalmap}), where~$\Ac(X)$ is defined by means of the embedding of $X$ into the affine space $\Ab^{\overline{E}}$.
\end{prop}
\begin{proof}
We may assume that $E\simeq \mbox{``$\varinjlim\limits_{i\in I}$''}E_i$, $X\simeq\mbox{``$\varinjlim\limits_{i\in I}$''}X_i$, and $X_i$ is a cone in the affine space $\Ab^{E_i}$ for any~$i\in I$. For each $i\in I$, apply Lemma~\ref{lemma-injcone} to the cone $X_i\subset\Ab^{E_i}$ and get a Cartesian square
$$
\begin{CD}
\Ac(X_i)\cap\OO(\widehat{\Ab}^{E_i})@>>> \OO(\widehat{\Ab}^{E_i}) \\
@VV V @VVV \\
\OO(X_i) @>>> \OO(X_i\cap \widehat{\Ab}^{E_i})\,.
\end{CD}
$$
By formula~\eqref{eq:Ehat}, we have the isomorphisms
$$
\varprojlim\limits_{i\in I}\Ac(X_i)\cap \OO(\widehat{\Ab}^{E_i})\simeq \Ac(X)\cap \OO(\widehat{\Ab}^E)\,,\qquad
\varprojlim\limits_{i\in I}\OO(\widehat{\Ab}^{E_i})\simeq\OO(\widehat{\Ab}^E)\,.
$$
Evidently, $\mbox{``$\varinjlim\limits_{i\in I}$''}X_i\cap \widehat{\Ab}^{E_i}\simeq X\cap \widehat{\Ab}^E$. Thus,
passing to the inverse limit over $i\in I$ and using that inverse limits are left exact, we obtain a Cartesian square
$$
\begin{CD}
\Ac(X)\cap\OO(\widehat{\Ab}^E)@>>> \OO(\widehat{\Ab}^E) \\
@VV V @VVV \\
\OO(X) @>>> \OO(X\cap \widehat{\Ab}^E)\,.
\end{CD}
$$
Since $X$ is thick in $\Ab^E$, we have the equality $\widehat{\Ab}^E\cap X=\widehat{\Ab}^E$ and the restriction homomorphism $\OO(\widehat{\Ab}^E)\to\OO(X\cap \widehat{\Ab}^E)$ is an isomorphism. Also, since $\widehat{\Ab}^{E}\subset X$, there is an equality \mbox{$\Ac(X)\cap\Ac(\widehat{\Ab}^E)=\Ac(X)$}. Since \mbox{$\Ac(\widehat{\Ab}^E)=\OO(\widehat{\Ab}^E)$} (see Remark~\ref{rmk:exampAO}), this implies the equality \mbox{$\Ac(X)\cap\OO(\widehat{\Ab}^E)=\Ac(X)$}. All together, this proves the proposition.
\end{proof}

\begin{rmk}\label{rmk:inverse}
It is easy to check that the composition of the inverse to the isomorphism $\Ac(X)\stackrel{\sim}\lrto \OO(X)$ from Proposition~\ref{prop-injectiveseries} with the embedding $\Ac(X)\subset R[[\overline{E}]]$ coincides with the restriction of functions to the ind-closed subscheme $\widehat{\Ab}^{E_f}\subset \widehat{\Ab}^E\subset X$ (see formula~\eqref{eq:Ef}).
\end{rmk}

\medskip

We will also use the following simple properties of thick ind-cones.

\begin{lemma}\label{lemma:decomposethickcone}
Let $E_1$, $E_2$ be strict ind-sets and let $X_1\subset \Ab^{E_1}$, $X_2\subset \Ab^{E_2}$ be ind-closed subschemes of the corresponding ind-affine spaces over $R$. Put
$$
\mbox{$E:=E_1\coprod E_2\,,\qquad X:=X_1\times X_2\subset \Ab^{E_1}\times\Ab^{E_2}\simeq\Ab^{E}$}\,.
$$
Then $X\subset \Ab^E$ is thick (respectively, is an ind-cone) if and only if $X_1\subset \Ab^{E_1}$ and $X_2\subset \Ab^{E_2}$ are thick (respectively, are ind-cones).
\end{lemma}
\begin{proof}
For each $a=1,2$, we have equalities
$$
X_a=X\cap \Ab^{E_a}\,,\qquad \widehat{\Ab}^{E_a}=\widehat{\Ab}^{E}\cap \Ab^{E_a}
$$
between ind-closed subschemes of $\Ab^E$, which proves one implication. For the other implication, use the isomorphism
$\widehat{\Ab}^{E_1}\times \widehat{\Ab}^{E_2}\simeq \widehat{\Ab}^E$.
\end{proof}

The proof of the following statement is straightforward.

\begin{lemma}\label{lemma:extescalthickcone}
Let $X\subset \Ab^E$ be an ind-closed subset and $R\to S$ be a homomorphism of rings. If $X$ is thick (respectively, is an ind-cone) in $\Ab^E$, then $X_{S}$ is thick (respectively, is an ind-cone) in $\Ab_{S}^E$.
\end{lemma}

\medskip

Here are further remarkable properties of functions on thick ind-cones.

\begin{prop}\label{prop:generdenis}
Let $X\subset \Ab^E$ be a thick ind-cone and let $\alpha,\alpha'\colon X\to Y$ be two morphisms to an affine scheme $Y$ over $R$. Suppose that for any $R$-algebra $A$ and any point \mbox{$p\in X(A)\subset \Ab^{\overline{E}}(A)$} with finitely many non-zero coordinates all of which are nilpotent, there is an equality $\alpha(p)=\alpha'(p)\in Y(A)$. Then $\alpha=\alpha'$.
\end{prop}
\begin{proof}
Using a closed embedding $Y\subset \Ab^M$ for a set $M$, we easily reduce the proposition to the case $Y=\Ab^1$. Then the statement follows directly from Proposition~\ref{prop-injectiveseries} and Remark~\ref{rmk:inverse}.
\end{proof}

\begin{prop}\label{lemma:analytic}
Let $X\subset \Ab^E$ be a thick ind-cone and let $R\subset S$ be an embedding of rings.
\begin{itemize}
\item[(i)]
The natural homomorphism $\OO(X)\to \OO(X_{S})$ is injective.
\item[(ii)]
Consider a power series $\varphi\in S[[\overline{E}]]$ and a regular function $\phi\in\OO(X)$. Suppose that for any $S$-algebra $A$ and any point \mbox{$p\in X(A)\subset \Ab^{\overline{E}}(A)$} with finitely many non-zero coordinates all of which are nilpotent, there is an equality
$$
\varphi(p)=\phi(p)\in A\,.
$$
Then $\varphi$ has coefficients from $R$, that is, $\varphi\in R[[\overline{E}]]$. Besides, the power series~$\varphi$ converges algebraically on $X$, that is, $\varphi\in\Ac(X)$, and $\varphi$ goes to $\phi$ under the canonical homomorphism $\Ac(X)\to\OO(X)$.
\end{itemize}
\end{prop}
\begin{proof}
$(i)$ Use Proposition~\ref{prop-injectiveseries}, Lemma~\ref{lemma:extescalthickcone}, and the fact that
the natural map from $\Ac(X)\subset R[[\overline{E}]]$ to $\Ac(X_{S})\subset S[[\overline{E}]]$ is injective.

$(ii)$ This follows directly from Proposition~\ref{prop-injectiveseries} and Remark~\ref{rmk:inverse}.
\end{proof}

An ind-affine scheme $Y$ over $R$ is {\it ind-flat over $R$} if there is an isomorphism ${Y\simeq\mbox{``$\varinjlim\limits_{j\in J}$''}Y_j}$ such that $Y_j$ is a flat affine scheme over $R$ for any $j\in J$. The following fact is a generalization of Proposition~\ref{lemma:analytic}$(i)$.

\begin{prop}\label{prop:geninj}
Let $X$ be a thick ind-cone in an ind-affine space over $R$, $Y$ be an ind-flat ind-affine scheme over $R$, and $Z$ be an affine scheme over $R$. Let $R\subset S$ be an embedding of rings. Then the natural map
$$
\Hom_R(X\times Y,Z)\lrto \Hom_{S}\big((X\times Y)_{S},Z_{S}\big)
$$
is injective.
\end{prop}
\begin{proof}
First, using a closed embedding $Z\subset \Ab^M$ for a set $M$, we easily reduce the proposition to the case $Z=\Ab^1$. In other words, it is enough to prove that the natural map
$$
\OO(X\times Y)\lrto \OO\big((X\times Y)_{S}\big)
$$
is injective.

Let $Y\simeq\mbox{``$\varinjlim\limits_{j\in J}$''}Y_j$ be such that $Y_j$ is a flat affine scheme over $R$ for any $j\in J$. We have canonical isomorphisms
$$
\OO(X\times Y)\simeq\varprojlim\limits_{j\in J}\OO(X\times Y_j)\,,\qquad
\OO\big((X\times Y)_{S}\big)\simeq\varprojlim\limits_{j\in J}\OO\big((X\times Y_j)_{S}\big)\,.
$$
Therefore it is enough to consider the case when $Y$ is a flat affine scheme over $R$, because inverse limits are left exact.

Let $Y\simeq\Spec(A)$, where $A$ is a flat $R$-algebra. Put $B:=A\otimes_R S$. Clearly, we have canonical isomorphisms
$$
\OO(X\times Y)\simeq\OO(X_A)\,,\qquad \OO\big((X\times Y)_{S}\big)\simeq\OO(X_{B})\,.
$$
Since the homomorphism $R\to S$ is injective and $A$ is flat over $R$, the natural homomorphism $A\to B$ is injective as well. Thus we finish the proof by Lemma~\ref{lemma:extescalthickcone} and Proposition~\ref{lemma:analytic}$(i)$ applied to the embedding of rings $A\subset B$.
\end{proof}

\subsection{Connectedness}\label{subsect:connect}

\begin{defin}\label{defin:connect}
An ind-scheme $X$ over $R$ is {\it connected over $R$} if any idempotent in~$\OO(X)$ is the image of an idempotent in $R$ under the natural homomorphism of rings ${R\to\OO(X)}$. An ind-scheme $X$ over $R$ is {\it absolutely connected over $R$} if for any homomorphism of rings $R\to S$, the ind-scheme $X_S$ is connected over $S$.
\end{defin}

\begin{rmk}\label{rmk:connectred}
Idempotents in a ring $A$ correspond bijectively to decompositions of the topological space $\Spec(A)$ into a disjoint union of two closed subsets. Therefore any surjective homomorphism $A\to B$ such that all elements in the kernel are nilpotent induces a bijection between the sets of idempotent in $A$ and $B$. It follows that an affine scheme $X$ over $R$ is (absolutely) connected over $R$ if and only if $X_{R_{\rm red}}$ is (absolutely) connected over $R_{\rm red}$, where $R_{\rm red}$ is the quotient of $R$ over the nilradical $\Nil(R)$ (cf. the proof of~\cite[Lem.\,3.2]{OZ1}).
\end{rmk}

One easily checks that if $\Spec(R)$ is connected as a topological space, then an ind-scheme $X$ over $R$ is connected over $R$ if and only if any morphism \mbox{$X\to\Spec(R)\coprod\Spec(R)$} over $R$ factors through a copy of $\Spec(R)$. If $X$ is a scheme over $R$, then the latter is also equivalent to connectedness of $X$ as a topological space.

\begin{examp}\label{examp:connsimple}
\hspace{0cm}
\begin{itemize}
\item[(i)]
The affine line $\Ab^1$ is absolutely connected over $\z$. Indeed, by Remark~\ref{rmk:connectred}, it is enough to show that $\Ab^1_R$ is connected over $R$ for any ring $R$ without non-zero nilpotents. A non-zero idempotent in $\OO(\Ab^1_R)$ corresponds to a polynomial
$$
f=a_0+a_1x+\ldots+a_dx^d\in R[x]
$$
such that $f^2=f$, where $a_d\ne 0$. If $d>0$, then $a_d^2=0$, which contradicts the condition $a_d\ne 0$. This shows that $\Ab^1_R$ is connected over $R$. More generally, for any set $M$, the affine space~$\Ab^M$ is absolutely connected over $\z$ (the proof is similar).
\item[(ii)]
One shows similarly as in item~(i) that $\gm$ is absolutely connected over $\z$.
\end{itemize}
\end{examp}

\begin{rmk}\label{rmk:connect}
Let $X\simeq\mbox{``$\varinjlim\limits_{i\in I}$''}X_i$  be an ind-scheme over $R$. Suppose that for any $i\in I$, the scheme $X_i$ is connected over $R$ and the homomorphism of rings $R\to \OO(X_i)$ is injective (for example, the latter condition holds if $X_i(R)$ is not empty). Then $X$ is connected over $R$.
\end{rmk}

\medskip

The following simple lemma says that ind-cones are absolutely connected. Geometrically speaking, this holds because one can join any point on a cone with the vertex.

\begin{lemma}\label{lemma:coneconn}
Let $X$ be an ind-cone in an ind-affine space over $R$. Then~$X$ is absolutely connected over $R$.
\end{lemma}
\begin{proof}
By Lemma~\ref{lemma:extescalthickcone}, it is enough to prove that $X$ is connected over $R$. Any cone has the point $0:=(0,0,\ldots)$ over any ring. Therefore by Remark~\ref{rmk:connect}, it is enough to suppose that $X$ is a cone $V$ in an affine space $\Ab^M$ over~$R$, where $M$ is a set. Consider an idempotent $\phi\in\OO(V)$. Let $e\in R$ be the value of~$\phi$ at the point $0\in V(R)$. Then $e$ is also an idempotent. Let us show that $\phi$ is equal to the image of $e$ under the homomorphism $R\to \OO(V)$.

Let $A$ be an $R$-algebra and let $x\in V(A)\subset \Ab^M(A)=A^M$ be an $A$-point on $V$. Since~$V$ is a cone, we have a morphism of schemes over $A$
$$
\gamma\;:\;\Ab^1_A\lrto V_A\,,\qquad a\longmapsto a\cdot x\,,
$$
which is a line on the cone joining $0$ and $x$. By Example~\ref{examp:connsimple}(i), $\Ab^1_A$ is connected over $A$. Therefore the idempotent $\gamma^*(\phi)\in \OO(\Ab^1_A)$ is the image of an idempotent in $A$ under the homomorphism $A\to \OO(\Ab^1_A)$. In particular, the function $\gamma^*(\phi)$ on $\Ab^1_A$ is constant and we have the equalities in $A$:
$$
\phi(0)=\gamma^*(\phi)(0)=\gamma^*(\phi)(1)=\phi(x)\,.
$$
On the other hand, $\phi(0)$ is the image in $A$ of $e\in R$ under the homomorphism $R\to A$. All together, this proves that $\phi$ is the image of $e$ under the homomorphism $R\to \OO(V)$.
\end{proof}

\medskip

Recall that the group functor $\uz$ (see Definition~\ref{defin:uz}) is represented by an ind-affine scheme (see Subsection~\ref{subsect:indschemes}), which is clearly not connected over $\z$. We will use the following property of connected ind-schemes.

\begin{prop}\label{lemma:conntriv}
Let $X$ be a connected ind-scheme over $R$ and let $\alpha\colon X\to \uz_R$ be a morphism of ind-schemes over $R$. Suppose that there exists an $R$-point $x\in X(R)$ such that ${\alpha(x)=0}$. Then $\alpha=0$.
\end{prop}
\begin{proof}
Let $\phi\in\OO(\uz_R)$ be the characteristic function of the zero element, that is, $\phi$ equals one on the component $\Spec(\z)\subset \uz$ that corresponds to $0\in \uz(R)$ and $\phi$ equals zero on all the other components. In particular, $\phi$ is an idempotent. The idempotent $\alpha^*(\phi)\in \OO(X)$ is the image of an idempotent $e\in R$, because $X$ is connected over $R$. The idempotent $e\in R$ defines a decomposition $R\simeq R_1\times R_2$ such that the image of $\alpha_{R_1}\colon X_{R_1}\to \uz_{R_1}$ in $\uz_{R_1}$ is equal to zero and the image of $\alpha_{R_2}\colon X_{R_2}\to \uz_{R_2}$ in $\uz_{R_2}$ is contained in the union of all components in $\uz$ except for the component of the zero element. If $R_2$ is non-zero, then $x$ defines an $R_2$-point $x_2\in X_{R_2}(R_2)$ such that $\alpha_{R_2}(x_2)=0$, which gives a contradiction. Therefore, $R=R_1$, $e=1$, and $\alpha=0$.
\end{proof}

\subsection{Density}\label{subsection:density}

\begin{defin}\label{defin-dense}
An ind-closed subscheme $X$ of an ind-affine scheme $Y$ is {\it dense} if the natural homomorphism of rings $\OO(Y)\to \OO(X)$ is injective. An ind-closed subscheme $X$ of an ind-affine scheme $Y$ is {\it absolutely dense} if for any ring $R$, we have that $X_R$ is dense in $Y_R$.
\end{defin}

Note that (absolute) density is not a relative with respect to a base notion, that is, (absolute) density does not depend on a base ring over which given ind-schemes are considered.

If $Y$ is an affine scheme, then an ind-closed subscheme $X$ of $Y$ is dense if and only if there is no a closed affine subscheme $Z\subset Y$, $Z\ne Y$, such that $X\subset Z$. In particular, given a closed embedding of affine schemes $X\subset Y$, we have that $X$ is dense in $Y$ if and only if~$X=Y$.

\begin{lemma}\label{lemma:equivdense}
An ind-closed subscheme $X$ of an ind-affine scheme $Y$ is absolutely dense if and only if for any ind-scheme~$T$ and any affine scheme $Z$, the natural map
\begin{equation}\label{eq:injmap}
\Hom(Y\times T,Z)\lrto \Hom(X\times T,Z)
\end{equation}
is injective.
\end{lemma}
\begin{proof}
The proof uses similar ideas as in the proof of Proposition~\ref{prop:geninj}.

Suppose that the map~\eqref{eq:injmap} is injective for all $T$ and $Z$ as above. Given a ring~$R$, put ${T=\Spec(R)}$ and $Z=\Ab^1$. Then the isomorphisms $\Hom(X\times T,Z)\simeq \OO(X_R)$ and ${\Hom(Y\times T,Z)\simeq \OO(Y_R)}$ imply that $X$ is absolutely dense in $Y$.

Let us prove another implication. Using a closed embedding $Z\subset \Ab^M$ for a set $M$, we easily reduce the lemma to the case $Z=\Ab^1$. Let $\alpha,\,\beta\colon Y\times T\to \Ab^1$ be two morphisms of ind-schemes such that $\alpha\ne\beta$. There is a ring $R$ and an $R$-point $x\in T(R)$ such that the morphisms of ind-schemes over $R$
$$
\alpha_R(-,x),\,\beta_R(-,x)\;:\; Y_R\lrto\Ab^1_R
$$
are note equal. Since $X$ is absolutely dense in $Y$, the images of $\alpha_R(-,x)$ and~$\beta_R(-,x)$ under the natural homomorphism of rings $\OO(Y_R)\to \OO(X_R)$ are also not equal. Therefore the compositions of the initial morphisms $\alpha$ and $\beta$ with the embedding $X\subset Y$ are not equal as well.
\end{proof}

\medskip

\begin{examp}\label{examp:gmdense}
\hspace{0cm}
\begin{itemize}
\item[(i)]
Note that the group functor $\Nil$ is represented by the ind-affine scheme
$$
\widehat{\Ab}^1\simeq\mbox{``$\varinjlim\limits_{n\in \N}$''}\Spec\big(\Z[x]/(x^n)\big)\,,
$$
which is an ind-closed subscheme of the affine space $\Ab^1\simeq \Spec\big(\z[x]\big)$ representing the group functor $\ga$.
We have that $\Nil$ is absolutely dense in $\ga$. Indeed, for any ring $R$, the natural homomorphism of $R$-algebras
$$
\OO\big((\ga)_R\big)\simeq R[x] \lrto \OO(\Nil_R)\simeq R[[x]]
$$
is injective. Similarly, for any natural number $N$, the ind-closed subscheme $\Nil^{\times N}$ of the affine space $(\ga)^{\times N}\simeq \Ab^N$ is absolutely dense.
\item[(ii)]
The group subfunctor $1+\Nil\subset\gm$ is represented by an ind-closed subscheme of~$\gm$ and is absolutely dense in it. Indeed, for any ring $R$, the natural homomorphism of $R$-algebras
$$
\OO\big((\gm)_R\big)\simeq R[y,y^{-1}]\simeq R[x,(1+x)^{-1}] \lrto \OO\big((1+\Nil)_R\big)\simeq R[[x]]
$$
is injective, where $y=1+x$. Similarly, for any natural number $N$, the ind-closed subscheme ${(1+\Nil)^{\times N}}$ is absolutely dense in the affine scheme $(\gm)^{\times N}$.
\end{itemize}
\end{examp}

The following proposition generalizes Example~\ref{examp:gmdense}(i).

\begin{prop}\label{lemma:thickindconedense}
Let $E$ be a strict ind-set and let $X\subset\Ab^E$ be an ind-closed subscheme of the corresponding ind-affine space. Suppose that $\widehat{\Ab}^{E_f}\subset X$ (see Definitions~\ref{defin:strictset} and~\ref{defin:hat}). Then $X$ is absolutely dense in $\Ab^E$.
\end{prop}
\begin{proof}
We have embeddings of ind-closed subschemes in~$\Ab^{\overline{E}}$
$$
\widehat{\Ab}^{E_f}\subset X\subset \Ab^E\,,
$$
which induce homomorphisms of rings
$$
\OO(\Ab^E)\lrto \OO(X)\lrto \OO(\widehat{\Ab}^{E_f})\simeq \Z[[\overline{E}]]\,.
$$
Since the composition of these homomorphisms is injective (see Remark~\ref{rmk:exampAO}), the first homomorphism is injective as well, whence $X$ is dense in $\Ab^E$. A similar argument applies to the embedding $X_R\subset \Ab^E_R$ over an arbitrary ring~$R$.
\end{proof}

In particular, Proposition~\ref{lemma:thickindconedense} implies that any thick ind-closed subscheme in $\Ab^E$ (see Definition~\ref{defin:dset}) is absolutely dense in it.

\section{Geometric properties of some iterated loop groups}\label{section:reprfunc}

\subsection{Ordered product of strict ind-sets}

We define an operation on strict ind-sets which allows us to treat representability of iterated loop functors. Recall that given a strict ind-set $E$, we have the associated set $\overline{E}$ (see Definition~\ref{defin:strictset}).

\quash{We use upper indices in order to distinguish different strict ind-sets, while lower indices belong to the corresponding partially ordered sets. For example, we consider two strict ind-sets $E^1$ and $E^2$ .The corresponding partially ordered sets are denoted by $I^1$ and $I^2$, respectively.}

\begin{defin}\label{defin:starindind}
Given strict ind-sets $E_1=\mbox{``$\varinjlim\limits_{i_1\in I_1}$''}E_{1,i_1}$ and $E_2=\mbox{``$\varinjlim\limits_{i_2\in I_2}$''}E_{2,i_2}$, their {\it ordered product}~$E_1*E_2$ is a strict ind-set given by the formula
$$
E_1*E_2:=\mbox{``$\varinjlim\limits_{\lambda\in \Lambda}$''}(E_1*E_2)_{\lambda}\,,
$$
where
$$
\Lambda:=\big\{(\lambda_1,\lambda_2)\;\mid\;\lambda_1\colon\overline{E}_2\to I_1,\quad \lambda_2\in I_2\big\}
$$
and for $\lambda=(\lambda_1,\lambda_2)\in \Lambda$, we put
$$
(E_1*E_2)_{\lambda}:=\big\{(e_1,e_2)\in\overline{E}_1\times\overline{E}_2\;\mid\; e_2\in E_{2,\lambda_2},\, e_1\in E_{1,\lambda_1(e_2)}\big\}\,.
$$
A partial order on $\Lambda$ is defined by ${\lambda=(\lambda_1,\lambda_2)\leqslant\lambda'=(\lambda'_1,\lambda'_2)}$ if and only if $\lambda_1\leqslant\lambda'_1$ and $\lambda_2\leqslant\lambda'_2$, where the order on functions is defined point-wise, that is, $\lambda_1\leqslant \lambda_1'$ if and only if $\lambda_1(e_2)\leqslant \lambda_1'(e_2)$ for all $e_2\in \overline{E}_2$.
\end{defin}

Note that $\overline{E_1*E_2}=\overline{E}_1\times\overline{E}_2$, that is, the strict ind-set $E_1*E_2$ is formed by subsets of the set $\overline{E}_1\times \overline{E}_2$.

\begin{examp}\label{examp:iterprod}
\hspace{0cm}
\begin{itemize}
\item[(i)]
Let $M$ be a set and let $E=\mbox{``$\varinjlim\limits_{i\in I}$''}E_i$ be a strict ind-set. Then there are canonical isomorphisms of strict ind-sets
$$
M*E\simeq\mbox{``$\varinjlim\limits_{i\in I}$''}M\times E_i\,,
$$
$$
E*M\simeq\mbox{``$\varinjlim\limits_{\lambda\in \Lambda}$''}(E*M)_{\lambda}\,,
$$
where $\Lambda:=\big\{\lambda\colon M\to I\}$ and $(E*M)_{\lambda}:=\big\{(e,m)\in\overline{E}\times M\;\mid\; e\in E_{\lambda(m)}\big\}$.
\item[(ii)]
If $E_1$ and $E_2$ are just sets, then by item~(i), the ordered product $E_1*E_2$ coincides with the Cartesian product $E_1\times E_2$.
\item[(iii)]
By $U$ denote the one element set. Then by item~(i), for any strict ind-set $E$, there are canonical isomorphisms $U*E\simeq E*U\simeq E$.
\end{itemize}
\end{examp}

\medskip

Clearly, the ordered product is functorial with respect to isomorphisms of strict ind-sets. One checks directly that the ordered product $E_1*E_2$ is functorial with respect to morphisms of strict ind-sets for the first argument, that is, any morphism of strict ind-sets $E_1\to E'_1$ defines in a natural way a morphism of strict ind-sets $E_1*E_2\to E'_1*E_2$. Also, the ordered product $E*M$ between a strict ind-set $E$ and a set $M$ is functorial with respect to injective maps of sets, that is, an injective map of sets $M\to M'$ defines in a natural way a morphism of strict ind-sets $E*M\to E*M'$.

\begin{rmk}\label{rmk:orderedprod}
One easily checks that for strict ind-sets $E_1$ and $E_2$, there is a canonical isomorphism of strict ind-sets
$$
E_1*E_2\simeq\varinjlim_{i_2\in I_2}E_1*E_{2,i_2}\,,
$$
where the limit in the right hand side is taken in the category of strict ind-sets.
\end{rmk}

However, the ordered product is not naturally functorial with respect to arbitrary morphisms of strict ind-sets as the following example shows.

\begin{examp}
Let $I$ be a directed partially ordered set without a final element (in particular,~$I$ is infinite) and let $E=\mbox{``$\varinjlim\limits_{i\in I}$''}E_i$ be a strict ind-set such that $E_i\ne E_j$ for~$i\ne j$. As above, let $U$ be the one element set. The unique map of sets $I\to U$ induces a projection ${p\colon\overline{E}\times I\to\overline{E}\times U\simeq \overline{E}}$. We claim that the projection $p$ does not extend to a morphism of strict ind-sets from $E*I$ to $E*U\simeq E$. Indeed, let $\lambda$ be the identical function from~$I$ to itself. Then the subset $(E*I)_{\lambda}$ of $\overline{E}\times I$ (see Example~\ref{examp:iterprod}(i)) is equal to
$$
\big\{(e,i)\in\overline{E}\times I\;\mid\;e\in E_i\big\}\,.
$$
In particular, $(E*I)_{\lambda}$ maps surjectively to~$\overline{E}$ under the projection $p$, while no $E_i$ maps surjectively to~$\overline{E}$ for $i\in I$. Therefore there are no morphisms of ind-sets from $E*I$ to~$E$.
\end{examp}

\medskip

The ordered product defines a monoidal structure on the category of strict ind-sets with isomorphisms. In other words, given three strict ind-sets ${E_1=\mbox{``$\varinjlim\limits_{i_1\in I_1}$''}E_{1,i_1}}$, ${E_2=\mbox{``$\varinjlim\limits_{i_2\in I_2}$''}E_{2,i_2}}$, and~${E_3=\mbox{``$\varinjlim\limits_{i_3\in I_3}$''}E_{3,i_3}}$, there is a functorial isomorphism
$$
(E_1*E_2)*E_3\simeq E_1*(E_2*E_3)
$$
that satisfies the standard coherence conditions (the unit object is the one element set~$U$). This holds because both $(E_1*E_2)*E_3$ and $E_1*(E_2*E_3)$ are functorially isomorphic to the strict ind-set $\mbox{``$\varinjlim\limits_{\delta\in \Delta}$''}(E_1*E_2*E_3)_{\delta}$, where
$$
\Delta:=\big\{(\delta_1,\delta_2,\delta_3)\;\mid\;\delta_1\colon \overline{E}_2\times\overline{E}_3\to I_1,\quad \delta_2\colon\overline{E}_3\to I_2,\quad \delta_3\in I_3 \big\}\,,
$$
and for $\delta=(\delta_1,\delta_2,\delta_3)\in\Delta$, we put
$$
(E_1*E_2*E_3)_{\delta}:=\big\{(e_1,e_2,e_3)\in\overline{E}_1\times\overline{E}_2\times
\overline{E}_3\;\mid\; e_3\in E_{3,\delta_3},\, e_2\in E_{2,\delta_2(e_3)},\, e_1\in E_{1,\delta_1(e_2,e_3)}\big\}\,.
$$
In particular, ordered powers of strict ind-sets are well-defined.

\begin{examp}\label{examp:D}
Define the following strict ind-set:
$$
D:=\mbox{``$\varinjlim\limits_{m\in \z}$''}\,\Z_{\geqslant -m}\,,
$$
where $\z_{\geqslant -m}$ denotes the set of all integers $l$ such that $l\geqslant -m$. Note that $\overline{D}=\Z$. We use notation from Subsection~\ref{subsect:Laurent}. The $n$-th ordered power~$D^{*n}$ of~$D$ is canonically isomorphic to the strict ind-set $\mbox{``$\varinjlim\limits_{\lambda\in \Lambda_n}$''}\,\z^n_{\lambda}$, where a partial order on~$\Lambda_n$ is defined as follows:
$$
\lambda=(\lambda_1,\ldots,\lambda_n)\leqslant \lambda'=(\lambda'_1,\ldots,\lambda'_n)
$$
if and only if $\lambda_1\geqslant \lambda_1',\ldots, \lambda_n\geqslant \lambda_n'$ (note the reverse order). Also, we have that $\overline{D^{*n}}=\z^n$.
\end{examp}

\medskip

Note that the functor $E\mapsto \overline{E}$ from the category of strict ind-sets with isomorphisms to the category of sets is naturally monoidal, where we consider the ordered product of strict ind-sets and the Cartesian product of sets.

We stress that the ordered product is not naturally commutative. More precisely, the above monoidal functor $E\mapsto \overline{E}$ is not symmetric, that is, given strict ind-sets $E_1$ and $E_2$, the canonical bijection $\overline{E}_1\times\overline{E}_2\simeq \overline{E}_2\times\overline{E}_1$ does not necessarily extend to an isomorphism of strict ind-sets between $E_1*E_2$ and $E_2*E_1$. For instance, this happens for the case in Example~\ref{examp:iterprod}(i) with general $M$ and $E$ or for the case~${E_1=E_2=D}$. This non-commutativity motivates the term ``ordered''.

\subsection{Representability of loop functors}\label{subsect:represent}

Let us discuss representability of loop functors (see Definition~\ref{defin:loopgroup}). Below, we consider affine spaces over $\Z$.

The functor $\lo \Ab^1\simeq\lo\ga$ is represented by an ind-affine space~$\Ab^D$ (see Subsection~\ref{subsection:indaffspace} and Example~\ref{examp:D}). Indeed, given a ring $A$, an $A$-point on $\Ab^D$ is given by a collection~$(a_l)_{l\in\z}$, where $a_l\in A$ and there is $m\in\z$ such that $a_l=0$ for $l<-m$. The point~$(a_l)_{l\in\z}$ corresponds to the Laurent series~$\sum\limits_{l\in\z} a_l t^l\in\LL(A)$.

\begin{prop}\label{lemma-repraffine}
\hspace{0cm}
\begin{itemize}
\item[(i)]
Given a strict ind-set $E$ (see Definition~\ref{defin:strictset}), the functor $\lo\Ab^E$ is represented by the ind-affine space~$\Ab^{D*E}$ (see Definition~\ref{defin:starindind}).
\item[(ii)]
Let $E$ be a strict ind-set and $X\subset \Ab^E$ be an ind-closed subscheme. Then the functor~$\lo X$ is represented by an ind-closed subscheme of $\Ab^{D*E}$ (cf.~\cite[\S\,1a]{PR}).
\item[(iii)]
In notation of item~$(ii)$, if $X$ is thick (respectively is an ind-cone) in $\Ab^E$, then $\lo X$ is thick (respectively, is an ind-cone) in $\Ab^{D*E}$ (see Definitions~\ref{defin:indcone} and \ref{defin:dset}).
\end{itemize}
\end{prop}
\begin{proof}
$(i)$ One easily checks that the assignment $F\mapsto \lo F$ commutes with direct limits of functors. Therefore, by Remark~\ref{rmk:orderedprod}, we can assume that $E$ is a set $M$ and it is enough to show an isomorphism $\lo\Ab^M\simeq \Ab^{D*M}$. Given a ring $A$, an element from the set  $\lo\Ab^M(A)$ is the same as a collection $\{f_m\}$, ${m\in M}$, of Laurent series from $\LL(A)$. Clearly, such a collection defines a function $\lambda\colon M\to\Z$ and an $A$-point on the affine space $\Ab^{(D*M)_{\lambda}}$ and vice versa.

$(ii)$ Similarly as in item~$(i)$, we can assume that $X$ is a closed subscheme $V$ of an affine space $\Ab^M$, where $M$ is a set. Note that $V$ is the fiber over the point $(0,0,0,\ldots)$ of a morphism $\Ab^M\to\Ab^{M'}$, where~$M'$ is a set of equations of $V$ in $\Ab^M$. Since the assignment $F\mapsto \lo F$ commutes with fibred products of functors, we see by item~$(i)$ that~$\lo V$ is the fiber over the point $(0,0,0,\ldots)$ of the corresponding morphism between ind-affine spaces $\Ab^{D*M}\to \Ab^{D*M'}$. Hence the functor~$\lo V$ is represented by an ind-closed subscheme of~$\Ab^{D*M}$.

$(iii)$ Note that for Laurent series $f_1,\ldots,f_d$, the coefficients of their product $f_1\cdot\ldots\cdot f_d$ are degree $d$ homogenous polynomials in coefficients of $f_i$'s. This implies item~$(iii)$ for being an ind-cone. Further, for any set $M$, we have an embedding $\Ab^{D*M}_{(d)}\subset \lo \Ab^M_{(d)}$ in $\Ab^{D*M}\simeq\lo \Ab^M$, which gives item~$(iii)$ for being thick.
\end{proof}

\begin{examp}\label{examp:loopga}
Applying Proposition~\ref{lemma-repraffine}$(i)$ iteratively to $\Ab^1$, we obtain that ${\lo^n\Ab^1\simeq\lo^n\ga}$ is represented by the ind-affine space $\Ab^{D^{*n}}$ (see Example~\ref{examp:D}). This can be also shown directly using the isomorphisms
$$
D^{*n}\simeq \mbox{``$\varinjlim\limits_{\lambda\in \Lambda_n}$''}\,\z^n_{\lambda}\,,\qquad\overline{D^{*n}}\simeq\z^n\,,
$$
and the description of iterated Laurent series in $\LL^n(A)=\lo^n\ga(A)$ for a ring $A$ in terms of the sets $\z^n_{\lambda}$, $\lambda\in\Lambda_n$, given in Subsection~\ref{subsect:Laurent}.
\end{examp}

\subsection{Representability of $\lo^n\gm$ and its special subgroups}

Let us show that the group functor $\lo^n\gm$, its group subfunctors $(\lo^n\gm)^0$ and $(\lo^n\gm)^{\sharp}$ (see Definition~\ref{defin:specialmult}), and the group functor $(\lo^n\ga)^{\sharp}$ (see Definition~\ref{defin:sharpaddfunctor}) are represented by ind-affine schemes, which are close to be ind-flat over $\z$.

\medskip

Let $\z^n_{>0}$ be the set of all elements $l\in\z^n$ such that $l> 0$ (see Subsection~\ref{subsection:decomp} for the lexicographical order on $\z^n$) and define a strict ind-set $D^{*n}_{>0}:=D^{*n}\cap \z^n_{>0}$ (see Example~\ref{examp:D} for the strict ind-set $D^{*n}$). Similarly, define the sets $\z^n_{\geqslant 0}$, $\z^n_{< 0}$, $\z^n_{\leqslant 0}$ and the strict ind-sets $D^{*n}_{\geqslant 0}$, $D^{*n}_{<0}$, $D^{*n}_{\leqslant 0}$. Clearly, we have decompositions
$$
\mbox{$\z^n=\z^n_{>0}\coprod\z^n_{\leqslant 0}\,,\qquad D^{*n}=D^{*n}_{>0}\coprod D^{*n}_{\leqslant 0}$}\,.
$$

\begin{prop}\label{lemma:sharprepr}
\hspace{0cm}
\begin{itemize}
\item[(i)]
The functor $\vv_{n,+}$ (see formula~\eqref{eq:v+}) is represented by the ind-affine space $\Ab^{D^{*n}_{>0}}$ and the functor $\vv_{n,-}$ (see formula~\eqref{eq:v-}) is represented by a thick ind-cone in the ind-affine space $\Ab^{D^{*n}_{<0}}$.
\item[(ii)]
The isomorphic functors $(\lo^n\gm)^{\sharp}\simeq (\lo^n\ga)^{\sharp}$ are represented by a thick ind-cone in the ind-affine space $\Ab^{D^{*n}}\simeq \lo^n\ga$ (see Example~\ref{examp:loopga}).
\item[(iii)]
The functors $(\lo^n\gm)^{0}$ and $\lo^n\gm$ are represented by ind-affine schemes of type~${X\times X'}$, where $X$ is a thick ind-cone in an ind-affine space over $\z$ and $X'$ is an ind-affine scheme which is ind-flat over $\z$.
\end{itemize}
\end{prop}
\begin{proof}
$(i)$ The assertion for $\vv_{n,+}$ follows directly from its definition. Let us prove the assertion for $\vv_{n,-}$. Recall that the functor $\Nil$ is represented by the ind-closed subscheme of $\Ab^1$
$$
\widehat{\Ab}^1\simeq\mbox{``$\varinjlim\limits_{d\in \N}$''}\Spec\big(\Z[x]/(x^d)\big)\,,
$$
which is definitely a thick ind-cone in $\Ab^1$. By Proposition~\ref{lemma-repraffine}$(iii)$ and Example~\ref{examp:loopga}, the functor $L^n\Nil$ is represented by a thick ind-cone in the ind-affine space $\lo^n\Ab^1\simeq\Ab^{D^{*n}}$ over $\z$.

Further, define $\lo^n\Nil_{\geqslant 0}$ as the intersection $\lo^n\Nil\cap \Ab^{D^{*n}_{\geqslant 0}}$ between ind-closed subschemes of $\Ab^{D^{*n}}\simeq \lo^n\ga$, and similarly for $\lo^n\Nil_{< 0}$. It follows from Remark~\ref{rmk:nilpseries} that there is an equality
$$
\mbox{$\lo^n\Nil=\lo^n\Nil_{\geqslant 0}\times\lo^n\Nil_{< 0}$}
$$
between ind-closed subschemes of $\Ab^{D^{*n}}=\Ab^{D^*_{\geqslant 0}}\times \Ab^{D^*_{< 0}}$. Now the required result follows from Lemma~\ref{lemma:decomposethickcone} and the isomorphism of functors $\vv_{n,-}\simeq L^n\Nil_{<0}$.

$(ii)$ By Definitions~\ref{defin:specialmult} and~\ref{defin:sharpaddfunctor}, we have the decompositions
$$
(\lo^n\gm)^{\sharp}\simeq \vv_{n,-}\times (1+\Nil)\times\vv_{n,+}\,,\qquad (\lo^n\ga)^{\sharp}\simeq \vv_{n,-}\times \Nil\times\vv_{n,+}\,.
$$
Also, we have the decomposition ${\Ab^{D^{*n}}=\Ab^{D^{*n}_{<0}}\times \Ab^1\times\Ab^{D^{*n}_{>0}}}$, where $\Ab^1$ in the middle corresponds to the element $0\in\z^n=\overline{D^{*n}}$. Thus we conclude by item~$(i)$ and Lemma~\ref{lemma:decomposethickcone}.

$(iii)$ By Definition~\ref{defin:specialmult}, we have the decomposition ${(\lo^n\gm)^{0}\simeq \vv_{n,-}\times\vv_{n,+}\times \gm}$ and by Proposition~\ref{prop-decomp}, we have the decomposition ${\lo^n\gm\simeq \vv_{n,-}\times\vv_{n,+}\times \gm\times\uz^n}$. Clearly, the ind-affine space $\vv_{n,+}\simeq\Ab^{D^{*n}_{>0}}$ and the ind-affine scheme $\uz$ (see Subsection~\ref{subsect:indschemes}) are ind-flat over $\z$. Hence the ind-affine scheme $\vv_{n,+}\times\gm\times\uz^n$ is ind-flat over $\z$ as well, which finishes the proof by item~$(i)$.
\end{proof}

\begin{rmk}
\hspace{0cm}
\begin{itemize}
\item[(i)]
Propositions~\ref{lemma:sharprepr}$(i)$ and~\ref{prop-injectiveseries} imply that the ring of regular functions ${\OO(\vv_{n,-})\simeq \Ac(\vv_{n,-})}$ is flat over $\Z$, or, equivalently, is torsion free over $\z$, because it is a subring of the ring of power series over $\Z$. On the other hand, Proposition~\ref{lemma:sharprepr}$(i)$ fixes a presentation of $\vv_{n,-}$ as an ind-affine scheme $\mbox{``$\varinjlim\limits_{i\in I}$''}X_i$, but the rings $\OO(X_i)$ are not flat over $\Z$, see a discussion after~\cite[Lem.\,3.4]{OZ1} (this flatness fails already for the case $n=1$).
\item[(ii)]
It is an interesting open question whether the ind-affine scheme $\vv_{n,-}$ is ind-flat over~$\z$, that is, whether there is another presentation $\mbox{``$\varinjlim\limits_{j\in J}$''}Y_j$ of~$\vv_{n,-}$ as an ind-affine scheme such that all rings $\OO(Y_j)$, $j\in J$, are flat over $\Z$ (note that the answer is positive for the case $n=1$, see, e.g.,~\cite[\S\,2]{OZ1}). This is the lack of an answer to this question that motivated us to develop the theory of thick ind-cones.
\end{itemize}
\end{rmk}

\medskip

Now we obtain a series of applications combining Proposition~\ref{lemma:sharprepr} with various properties of thick-ind cones from Section~\ref{representab}. Let $N$ be a natural number. Consider the set $M:=\coprod\limits_{i=1}^N \z^n$ and the corresponding set of formal variables $x_{i,l}$, where $1\leqslant i\leqslant N$, $l\in\z^n$, which are coordinates on the affine space $\Ab^M$ (see Subsection~\ref{subsect:convalg}). For short, by $X$ denote the ind-affine scheme $\big((\lo^n\ga)^{\sharp}\big)^{\times N}$. Thus $X$ is an ind-closed subscheme of the affine space~$\Ab^M$ such that a collection of iterated Laurent series $(f_1,\ldots,f_N)$ cooresponds to the point $(a_{i,l})$, where $f_i=\sum\limits_{l\in\z^N}a_{i,l} t^l$ for~${1\leqslant i\leqslant N}$.

\begin{theor}\label{prop-key}
Let $R\subset S$ be an embedding of rings. Consider a power series (see Definition~\ref{defin:powerseries})
$$
\varphi\in S[[x_{i,l};\,1\leqslant i\leqslant N,\,l\in\z^N]]
$$
and a regular function $\phi\in \OO(X_R)$. Suppose that for any \mbox{$S$-algebra}~$A$ and any collection of Laurent polynomials (not just series) $f_1,\ldots,f_N$ with nilpotent coefficients in~$A$, there is an equality
$$
\varphi(f_1,\ldots,f_N)=\phi(f_1,\ldots,f_N)\in A\,.
$$
Then $\varphi$ has coefficients in $R$, converges algebraically on $X_R$ (see Definition~\ref{defin-conv}(iii)), and goes to $\phi$ under the canonical homomorphism $\Ac(X_R)\to \OO(X_R)$ (see formula~\eqref{eq:evalmap}).
\end{theor}
\begin{proof}
Define a strict ind-set $E:=\coprod\limits_{i=1}^ND^{*n}$. By Proposition~\ref{lemma:sharprepr}$(ii)$ and Lemma~\ref{lemma:decomposethickcone}, $X$ is a thick ind-cone in $\Ab^E$. By Lemma~\ref{lemma:extescalthickcone}, $X_R$ is a thick ind-cone in $\Ab^E_R$. Thus we conclude by Proposition~\ref{lemma:analytic}$(ii)$.
\end{proof}

\begin{rmk}
Combining Propositions~\ref{lemma:sharprepr}$(ii)$ and~\ref{prop:generdenis}, one also obtains a higher-dimensional generalization of~\cite[Prop.\,3.7]{OZ1}.
\end{rmk}

\medskip

\begin{theor}\label{cor:uniq}
Let $R\subset S$ be an embedding of rings, $Y$ be an ind-flat ind-affine scheme over $R$, $Z$ be an affine scheme over $R$, and $N$ be a natural number. Then the natural map
$$
\Hom_R\big((L^n\gm)_R^{\times N}\times Y,Z\big)\lrto \Hom_{S}\big((L^n\gm)^{\times N}_{S}\times Y_S,Z_{S}\big)
$$
is injective.
\end{theor}
\begin{proof}
By Proposition~\ref{lemma:sharprepr}$(iii)$ and Lemma~\ref{lemma:decomposethickcone}, the functor $(L^n\gm)^{\times N}$ is represented by an ind-affine scheme of type~${X\times X'}$, where $X$ is a thick ind-cone in an ind-affine space over~$\z$ and $X'$ is an ind-flat ind-affine scheme over $\z$. By Lemma~\ref{lemma:extescalthickcone}, the functor $(L^n\gm)^{\times N}_R$ is represented by a similar ind-affine scheme over $R$. Thus we conclude by Proposition~\ref{prop:geninj}.
\end{proof}

\medskip

Here are further geometric properties of $\lo^n\ga$ and the special subgroups of $\lo^n\gm$.

\begin{prop}\label{prop:densegm}
Let $N$ be a natural number.
\begin{itemize}
\item[(i)]
The ind-closed subscheme ${\big((\lo^n\ga)^{\sharp}\big)^{\times N}\subset
(\lo^n\ga)^{\times N}}$ is absolutely dense.
\item[(ii)]
The ind-scheme $\big((\lo^n\ga)^{\sharp}\big)^{\times N}$ is absolutely connected over $\z$.
\item[(iii)]
The ind-closed subscheme
${\big((\lo^n\gm)^{\sharp}\big)^{\times N}\subset\big((\lo^n\gm)^0\big)^{\times N}}$ is absolutely dense.
\item[(iv)]
\label{lemma-geomgroups}
The ind-schemes $\big((\lo^n\gm)^{\sharp}\big)^{\times N}$ and $\big((\lo^n\gm)^{0}\big)^{\times N}$ are absolutely connected over $\Z$.
\end{itemize}
\end{prop}
\begin{proof}
$(i)$ The ind-closed subscheme $(\lo^n\ga)^{\sharp}$ in $\lo^n\ga\simeq\Ab^{D^{*n}}$ contains the ind-closed subscheme $\Ab^{(D^{*n})_f}$ (see Definition~\ref{defin:strictset}), because $\Ab^{(D^{*n})_f}$ corresponds to Laurent polynomials with nilpotent coefficients. Thus by Proposition~\ref{lemma:thickindconedense}, $(\lo^n\ga)^{\sharp}$ is absolutely dense in $(\lo^n\ga)$. The case of an arbitrary $N$ is treated similarly.

$(ii)$ Clearly, the ind-closed subscheme $(\lo^n\ga)^{\sharp}$ in the ind-affine space $\lo^n\ga\simeq\Ab^{D^{*n}}$ is an ind-cone (see a discussion after Definition~\ref{defin:indcone} or use Proposition~\ref{lemma:sharprepr}$(ii)$). Thus we conclude by Lemma~\ref{lemma:coneconn}.

$(iii)$ Combine Definition~\ref{defin:specialmult} with Example~\ref{examp:gmdense}(ii) and Lemma~\ref{lemma:equivdense}.

$(iv)$ By item~$(iii)$, it is enough to prove that $\big((\lo^n\gm)^{\sharp}\big)^{\times N}$ is absolutely connected over~$\z$. This follows from item~$(ii)$ and the isomorphism of ind-schemes ${(\lo^n\gm)^{\sharp}\simeq (\lo^n\ga)^{\sharp}}$.
\end{proof}

\begin{rmk}
All notions from Subsections~\ref{subsect:connect} and~\ref{subsection:density} make sense when ind-schemes are replaced just by functors (see Section~\ref{sect:notation}). Moreover all statements from Subsections~\ref{subsect:connect} and~\ref{subsection:density} still hold in this generality. Thus, in fact, Proposition~\ref{prop:densegm} is based neither on representability of iterated loop groups, nor on the theory of thick ind-cones.
\end{rmk}

\subsection{Tangent space to $\lo^n\gm$}

Let us discuss the tangent space at the neutral element of the group ind-scheme~$\lo^n\gm$. In what follows we will also use tangent spaces to group functors, so we give a more general definition. In the sequel, let $\varepsilon$ be a formal variable that satisfies~$\varepsilon^2=0$.

\begin{defin}\label{defin:tangfunctor}
Given a group functor $F$ over a ring $R$ (see Section~\ref{sect:notation}), a {\it tangent space}~$TF$ to~$F$ is a group functor over $R$ defined by the formula
$$
TF(A):=\Ker\big(F\big(A[\varepsilon]\big)\to F(A)\big)\,,
$$
where $A$ is an $R$-algebra.
\end{defin}

In particular, given a morphism of group functors $\Phi\colon F'\to F$ over $R$, we have its {\it differential} $T\Phi\colon TF'\to TF$, which is also a morphism of group functors over $R$.

\begin{examp}\label{exam:tgm}
\hspace{0cm}
\begin{itemize}
\item[(i)]
There is an isomorphism of group functors ${T\ga\stackrel{\sim}\lrto\ga}$, ${a\,\varepsilon \mapsto a}$, where $a$ is an element in a ring $A$.
\item[(ii)]
There is an isomorphism of group functors ${T\gm\stackrel{\sim}\lrto\ga}$, ${1+a\,\varepsilon \mapsto a}$, where $a$ is an element in a ring $A$.
\end{itemize}
\end{examp}

\medskip

\begin{lemma}\label{lemma:TL}
\hspace{0cm}
\begin{itemize}
\item[(i)]
For any group functor $F$, there is a canonical isomorphism of group functors
$TLF\simeq LTF$.
\item[(ii)]
We have isomorphisms of group functors
$$
T(\lo^n\gm)^{\sharp}\simeq T(\lo^n\gm)^0\simeq T\lo^n\gm\simeq \lo^n\ga\,.
$$
\item[(iii)]
We have isomorphisms of group functors
$$
T(\lo^n\ga)^{\sharp}\simeq T(\lo^n\ga)\simeq \lo^n\ga\,.
$$
\end{itemize}
\end{lemma}
\begin{proof}
Let $A$ be a ring.

$(i)$ We have an isomorphism of rings $A[\varepsilon]((t))\simeq A((t))[\varepsilon]$ (cf. Lemma~\ref{lemma:top}$(iii)$). This implies the following sequence of isomorphisms of groups
$$
TLF(A)\simeq \Ker\big(LF(A[\varepsilon])\to LF(A)\big)\simeq \Ker\big(F\big(A[\varepsilon]((t))\big)\to F\big(A((t))\big)\big)\simeq
$$
$$
\simeq \Ker\big(F\big(A((t))[\varepsilon]\big)\to F\big(A((t))\big)\big)\simeq LTF(A)\,.
$$

$(ii)$ Let $f\in T\lo^n\gm(A)$. Explicitly, we have that ${f\in \Ker\big(\lo^n\gm(A[\varepsilon])\to \lo^n\gm(A)\big)}$. Therefore all coefficients of $f$ are nilpotent except for the constant term, which belongs to $1+A\,\varepsilon\subset A[\varepsilon]^*$. Lemma~\ref{expl-form}$(iii)$ implies that $f\in (\lo^n\gm)^{\sharp}\big(A[\varepsilon]\big)$. This proves the first and the second isomorphisms in item~$(ii)$. To show the third isomorphism, use item~$(i)$ and the isomorphism of group functors $T\gm\simeq\ga$ (see Example~\ref{exam:tgm}(ii)).

$(iii)$ The proof is similar to item~$(ii)$.
\end{proof}

\begin{rmk}
The first and the second isomorphisms in Lemma~\ref{lemma:TL}$(ii)$ correspond to the facts that $(\lo^n\gm)^{\sharp}$ is absolutely dense in $(\lo^n\gm)^{0}$ (see Proposition~\ref{prop:densegm}$(iii)$) and that the quotient $\lo^n\gm/(\lo^n\gm)^{0}\simeq\uz^n$ is discrete (see formula~\eqref{eq:decommult} in Subsection~\ref{sp-sub}), respectively. The first isomorphism in Lemma~\ref{lemma:TL}$(iii)$ corresponds to the fact that~$(\lo^n\ga)^{\sharp}$ is absolutely dense in~$\lo^n\ga$ (see Proposition~\ref{prop:densegm}$(i)$).
\end{rmk}

\subsection{Characters of $(\lo^n\gm)^0$}

\begin{defin}\label{defin:char}
Given a group functor $F$ over a ring $R$, the group of {\it characters} of $F$ is defined by the formula ${{\rm X}(F):=\Hom^{gr}_R\big(F,(\gm)_R\big)}$ and the group of {\it linear functionals} on $F$ is defined by the formula ${{\rm X}^+(F):=\Hom^{gr}_R\big(F,(\ga)_R\big)}$.
\end{defin}

Recall the following well-known facts.

\begin{lemma}\label{lemma:diffgagm}
\hspace{0cm}
\begin{itemize}
\item[(i)]
For any ring $R$ of zero characteristic, there is an isomorphism of groups ${R\stackrel{\sim}\lrto{\rm X}^{+}\big((\ga)_R\big)}$, ${r\longmapsto \big(a\mapsto ar\big)}$, where $a$ is an element in an \mbox{$R$-algebra}~$A$.
\item[(ii)]
For any ring $R$, there is an isomorphism of groups
${\uz(R)\stackrel{\sim}\lrto{\rm X}\big((\gm)_R\big)}$, ${\underline{i}\longmapsto\big(a\mapsto a^{\underline{i}}\big)}$,
where $a$ is an invertible element in an $R$-algebra $A$.
\item[(iii)]
For any $\Q$-algebra $R$, there is an isomorphism of groups
${\Nil(R)\stackrel{\sim}\lrto{\rm X}\big((\ga)_R\big)}$, ${r\longmapsto\big(a\mapsto \exp(ar)\big)}$,
where $a$ is an element in an $R$-algebra $A$. The inverse isomorphism sends a character $\chi\colon (\ga)_R\to(\gm)_R$ to its differential (see item~$(i)$)
$$
T\chi\in\Hom^{gr}_R\big(T(\ga)_R,T(\gm)_R\big)\simeq \Hom^{gr}_R\big((\ga)_R,(\ga)_R\big)\simeq R\,,
$$
which is necessarily a nilpotent element in $R$.
\end{itemize}
\end{lemma}
\begin{proof}
$(i)$ A linear functional $(\ga)_R\to (\ga)_R$ is given by a polynomial $f\in R[x]$ such that $f(0)=0$ and $f(x+y)=f(x)+f(y)$. Thus we conclude by the binomial theorem.

$(ii)$ A character $(\gm)_R\to (\gm)_R$ is given by a Laurent polynomial $f=\sum_i a_i x^i$ in~$R[x,x^{-1}]^*$ such that $f(1)=1$ and $f(x)\cdot f(y)=f(x\cdot y)$. These identities immediately imply that the coefficients $a_i$ are orthogonal idempotents in $R$ and their sum is equal to~$1$.

$(iii)$ A character $(\ga)_R\to (\gm)_R$ is given by a polynomial $f\in R[x]^*$ such that $f(0)=1$ and $f(x)\cdot f(y)=f(x+y)$. Applying $\log$, we get an equality ${\log(f)(x)+\log(f)(y)=\log(f)(x+y)}$ between power series in $R[[x]]$. Similarly as in item~$(i)$, we obtain that $\log(f)(x)=rx$ for an element $r\in R$. Since $f=\exp\big(\log(f)\big)$ is a polynomial, we see that $r$ is nilpotent.
\end{proof}

\medskip

\begin{prop}\label{lemma:difflngmga}
Let $R$ be a $\Q$-algebra and let $\chi\in{\rm X}\big((\lo^n\gm)^{\sharp}_R\big)$ be a character. Then the subfunctor (see Lemma~\ref{lemma:TL}$(ii)$)
$$
(\lo^n\ga)^{\sharp}_R\subset (\lo^n\ga)_R\simeq T(\lo^n\gm)^{\sharp}_R
$$
is sent by $T\chi$ to the subfunctor
$$
\Nil_R\subset (\ga)_R\simeq T(\gm)_R
$$
and the following diagram commutes (see Proposition~\ref{log-map} for the map $\exp$):
$$
\begin{CD}
(L^n \ga)^{\sharp}_R @>{T\chi}>>  \Nil_R   \\
 @V\exp VV @ V \exp VV \\
  (L^n\gm)^{\sharp}_R  @>{\chi}>> (\gm)_R\,.
  \end{CD}
$$
\end{prop}
\begin{proof}
Let $A$ be an $R$-algebra and consider an iterated Laurent series $g\in (L^n \ga)^{\sharp}(A)$. We obtain a character
$$
\kappa\;:\;(\ga)_A\stackrel{g}\lrto(\lo^n\ga)^{\sharp}_A\stackrel{\exp}\lrto(\lo^n\gm)^{\sharp}_A\stackrel{\chi}\lrto(\gm)_A\,,
$$
where the first morphism of group ind-schemes sends an element $b$ from an $A$-algebra $B$ to the iterated Laurent series $bg\in \LL^n(B)$. Thus we have that $\kappa(b)=\chi\big(\exp(bg)\big)\in B^*$. Let us find the differential~$T\kappa\in \Nil(A)$ (see Example~\ref{exam:tgm} and Lemma~\ref{lemma:diffgagm}$(i)$).

Recall that by Example~\ref{exam:tgm} and Lemma~\ref{lemma:TL}, we have isomorphisms
$$
T\ga\simeq\ga\,,\qquad T(\lo^n\ga)^{\sharp}\simeq\lo^n\ga\,,\qquad T(\lo^n\gm)^{\sharp}\simeq\lo^n\ga\,,\qquad T\gm\simeq\ga\,.
$$
It is easily seen that under these isomorphisms $Tg$ corresponds to $g$ and $T\exp$ corresponds to the identity morphism. Indeed, for any ring~$C$ and any iterated Laurent series ${h\in\lo^n\ga(C)}$, we have an equality ${\exp(h\,\varepsilon )=1+h\,\varepsilon}$ in $\lo^n\gm\big(C[\varepsilon]\big)$, where $\varepsilon^2=0$. It follows that the differential~$T\kappa$ corresponds to the element $(T\chi)(g)\in A$. By Lemma~\ref{lemma:diffgagm}$(iii)$ applied over the ring $A$, we see that the element~$(T\chi)(g)$ is nilpotent and there is an equality $\kappa(1)=\exp\big((T\chi)(g)\big)$. On the other hand, ${\kappa(1)=\chi\big(\exp(g)\big)}$, which proves the proposition.
\end{proof}

\begin{rmk}
Proposition~\ref{lemma:difflngmga} is a particular case of an algebraic version of the well-known analytic statement that a homomorphism between Lie groups commutes with its differential through the exponential map.
\end{rmk}

\begin{prop}\label{prop:charactgm}
Let $R$ be a $\Q$-algebra. Then the natural homomorphism of groups
$$
{\rm X}\big((\lo^n\gm)^0_R\big)\lrto{\rm X}^+\big((\lo^n\ga)_{R}\big)\,,\qquad \chi\longmapsto T\chi\,,
$$
is injective. In addition, for any character $\chi\colon (\lo^n\gm)^0_R\to(\gm)_R$, an $R$-algebra $A$, and an element
${f\in (\lo^n\gm)^{\sharp}(A)}$, there is an equality in $A^*$
$$
\chi(f)=\exp(a)\,,
$$
where a nilpotent element $a\in\Nil(A)$ is defined by the equality in~$A[\varepsilon]^*$
$$
\chi\big(1+\log(f)\,\varepsilon\big)=1+a\,\varepsilon\,.
$$
\end{prop}
\begin{proof}
By Proposition~\ref{prop:densegm}$(iii)$ and Lemma~\ref{lemma:equivdense}, the restriction map
$$
{\rm X}\big((\lo^n\gm)^0_R\big)\lrto{\rm X}\big((\lo^n\gm)^{\sharp}_R\big)
$$
is injective. Now we conclude by Propositions~\ref{log-map} and~\ref{lemma:difflngmga}.
\end{proof}

\quash{
\begin{prop}\label{prop:rigid}
For any ring $R$, any multilinear morphism of functors
$$
(\gm)_R\times\big((\lo^n\gm)^{\sharp}\big)^{\times n}_R\lrto(\gm)_R
$$
is identically one.
\end{prop}
\begin{proof}
A multilinear morphism as above defines a morphism of functors
$$
\big((\lo^n\gm)^{\sharp}\big)^{\times n}_R\to \underline{\Hom}_R\big((\gm)_R,(\gm)_R\big)\simeq\uz_R\,,
$$
where for functors $F$ and $F'$ over $R$, by $\underline{\Hom}_R(F',F)$ we denote the functor over $R$ that sends an $R$-algebra $A$ to $\Hom_A(F'_A,F_A)$. Moreover, this morphism of functors sends the point $(1,\ldots,1)\in \big((\lo^n\gm)^{\sharp}\big)^{\times n}(R)$ to $0$. Thus Propositions~\ref{prop:densegm}$(iv)$ and~\ref{lemma:conntriv} imply the needed statement.
\end{proof}
}

\subsection{Linear functionals on $L^n\Omega^i$}

Define the following morphism of group functors:
$$
\Delta\;:\;\lo^n\ga\lrto \lo^n\Omega^1\,,\quad f\longmapsto
\mbox{$df-\sum\limits_{i=1}^n\frac{\partial f}{\partial t_i}dt_i$}\,,
$$
where $f\in\lo^n\ga(A)=\LL^n(A)$ for a ring $A$.

\begin{lemma}\label{lemma:vanish}
Let $R$ be a ring, $\Phi\colon(\lo^n\Omega^1)_R\to\Ab^1_R$ be a morphism of functors over $R$ (which does not necessarily respect the group structure), and let $\phi\in\OO\big((\lo^n\ga)_R\big)$ be the composition $\Phi\circ\Delta_R$. Then for any $R$-algebra $A$ and an iterated Laurent series $f\in\LL^n(A)$, there is an equality $\phi(f)=\phi(0)$.
\end{lemma}
\begin{proof}
We use notation from Subsection~\ref{subsect:Laurent}. Recall from Example~\ref{examp:loopga} that $\lo^n\ga$ is represented by the ind-affine space ${\Ab^{D^{*n}}\simeq\mbox{``$\varinjlim\limits_{\lambda\in \Lambda_n}$''}\,\Ab^{\z^n_{\lambda}}}$. Let $\lambda\in\Lambda_n$ be such that ${f\in \Ab^{\z^n_{\lambda}}(A)}$, that is, we have that $f=\sum\limits_{l\in\z^n_{\lambda}}a_lt^l$, where~$a_l\in A$. The restriction of $\phi$ to the affine space~$(\Ab^{\z^n_{\lambda}})_R$ over $R$ is given by a polynomial $\phi_{\lambda}\in R[\z^n_{\lambda}]$. Let $T\subset\z^n_{\lambda}$ be a finite subset such that the polynomial $\phi_{\lambda}$ depends only on variables that correspond to elements in~$T$. Let $f'=\sum\limits_{l\in\z^n_{\lambda}}a'_lt^l$ be a Laurent polynomial such that $a'_l=a_l$ if $l\in T$ and $a'_l=0$ if $l\notin T$. By construction, we have that $\phi(f)=\phi(f')$. On the other hand, one easily checks that $\Delta(f')=\Delta(0)=0$, whence $\phi(f')=\phi(0)$. This finishes the proof.
\end{proof}

For an integer $i\geqslant 0$, by $\widetilde{L^n\Omega}{}^i$ denote the group functor that sends a ring $A$ to~$\widetilde{\Omega}^i_{\LL^n(A)}$ (see Definition~\ref{defin:tildeforms}). Note that we have a canonical morphism of group functors ${L^n\Omega^i\to \widetilde{L^n\Omega}{}^i}$.

\begin{prop}\label{lemma:algiterforms}
For any ring $R$ and an integer $i\geqslant 0$, the natural homomorphism of groups
$$
{\rm X}^+\big((\widetilde{\lo^n\Omega}{}^i)_R\big)\lrto
{\rm X}^+\big((\lo^n\Omega^i)_R\big)
$$
is an isomorphism.
\end{prop}
\begin{proof}
This follows directly from Definition~\ref{defin:tildeforms} and Lemma~\ref{lemma:vanish}.
\end{proof}

\section{Auxiliary results from algebraic $K$-theory}\label{sect:Ktheory}

Let $m\geqslant 0$ be an integer.

\subsection{Milnor $K$-groups and algebraic $K$-groups}

\label{Milnor-K-group}

 We will use the following version of Milnor $K$-groups for rings.

\begin{defin}\label{defin:Milnor}
The {\it $m$-th Milnor $K$-group} $K_m^M(A)$ of a ring $A$ is the quotient group of the group~$(A^*)^{\otimes\,m}$ by the subgroup generated by all elements of type
\begin{equation}\label{eq:Steintensor}
a_1\otimes\ldots\otimes a_i\otimes a\otimes(1-a)\otimes a_{i+3}\otimes\ldots\otimes a_m\,,
\end{equation}
which are called {\it Steinberg relations}.
\end{defin}

Notice that in tensor~\eqref{eq:Steintensor}, the multiples $a$ and $1-a$ come one after another.
As usual, we denote by~$\{a_1,\ldots,a_m\}$ and we call a {\it symbol} the class in $K^M_m(A)$ of a tensor \mbox{$a_1\otimes\ldots\otimes a_m$}, $a_i\in R^*$.

\medskip

Clearly, $K_m^M$ is a group functor (see Section~\ref{sect:notation}). By $\Omega^m$ denote the group functor that sends a ring~$A$ to the group of $m$-th absolute K\"ahler differentials $\Omega^m_A$. It is easily checked that there is a morphism of group functors
$$
d\log\;:\; K_m^M\lrto \Omega^m\,,\qquad \{a_1,\ldots,a_m\}\longmapsto \frac{da_1}{a_1}\wedge\ldots\wedge\frac{da_m}{a_m}\,.
$$

\medskip

We also have algebraic $K$-groups $K_m(A)$, $m\geqslant 0$, which are functorial with respect to $A$. There are canonical decompositions of group functors
$$
K_0\simeq \uz\times\widetilde{K}_0\,,\qquad K_1\simeq \gm\times SK_1\,.
$$
Let $\rk\colon K_0\to\uz$ and $\det\colon K_1\to \gm$ be the corresponding projections. Explicitly, they are defined by taking the rank of a finitely generated projective module and by taking the determinant of a matrix, respectively.

\medskip

Loday~\cite{L} constructed a graded-commutative product between algebraic $K$-groups, which is functorial with respect to $A$:
$$
K_i(A)\otimes_{\mathbb{Z}} K_j(A)\lrto K_{i+j}(A)\,,\qquad  \alpha\otimes \beta \longmapsto \alpha\cdot \beta\,,\qquad i,j\geqslant 0\,.
$$
The Loday product between the group subfunctors $\gm\subset K_1$ factors through the Milnor $K$-groups (see, e.g.,~\cite[\S\S\,1,\,2]{S}), that is, we have a morphism of group functors
$$
K^M_m\lrto K_m\,,\qquad \{a_1,\ldots,a_m\}\longmapsto a_1\cdot\ldots\cdot a_m\,,
$$
where $a_i\in A^*$ for a ring $A$.

\subsection{Tangent space to Milnor $K$-groups}\label{tangent-space}

Let us state the main result in~\cite{GOMilnor}. First recall some notions from op.cit. Let $\varepsilon$ be a formal variable that satisfies $\varepsilon^2=0$. See Definition~\ref{defin:tangfunctor} for the tangent space~$TF$ to a group functor~$F$.

\begin{examp}
For any ring $A$, we have that $d(\varepsilon^2)=2\varepsilon d\varepsilon=0$ in $\Omega^1_{A[\varepsilon]}$ and a calculation shows that there is an isomorphism of $A$-modules
\begin{equation}\label{eq:decomforms}
T\Omega^{m+1}(A)\simeq \big(\varepsilon\,\Omega^{m+1}_A\big)\oplus \big(d\varepsilon\wedge\Omega^m_A\big)\oplus \left(\big(\varepsilon d\varepsilon\wedge\Omega^m_A\big)/2\big(\varepsilon d\varepsilon\wedge\Omega^m_A\big) \right)\,.
\end{equation}
\end{examp}

Following a construction of Bloch~\cite{Blo}, we give the following definition.

\begin{defin}\label{defin:Blochmap}
Let
$$
Bl\;:\;TK^M_{m+1}\lrto \Omega^{m}
$$
be the composition of the morphism of group functors $T(d\log)\colon TK^M_{m+1}\to T\Omega^{m+1}$ and the projection to the direct summand $\Omega^{m}\simeq d\varepsilon\wedge\Omega^m$ in decomposition~\eqref{eq:decomforms}.
\end{defin}

For instance, given a collection of invertible elements $a_1,\ldots,a_{m}\in A^*$ and an element $b\in A$, the map $Bl$ sends the symbol $\{1+b\,\varepsilon,a_1,\ldots,a_m\}$ to the differential form \mbox{$b\,\frac{da_1}{a_1}\wedge\ldots \wedge \frac{da_m}{a_m}$}.

\medskip

It turns out that the homomorphism of groups $Bl\colon TK^M_{m+1}(A)\to \Omega^m_A$ is an isomorphism when a ring $A$ has sufficiently many invertible elements in the following sense (see Definition~3.1 from the paper of Morrow~\cite{Mor}).

\begin{defin}\label{defin:wstable}
Given a natural number $k\geqslant 2$, a ring $A$ is called {\it weakly $k$-fold stable} if for any collection of elements $a_1,\ldots,a_{k-1}\in A$, there is an invertible element $a\in A^*$ such that the elements $a_1+a,\ldots,a_{k-1}+a$ are invertible in $A$.
\end{defin}

\begin{rmk}\label{exam:wstable}
For any ring $A$, the ring of Laurent series $\LL(A)=A((t))$ is weakly $k$-fold stable for all $k\geqslant 2$. Actually, an invertible element $a$ in Definition~\ref{defin:wstable} can be taken of the form~$t^i$ for a suitable~\mbox{$i\in \Z$}. \end{rmk}

For the following result see~\cite[Theor.\,2.9]{GOMilnor} (see also Example~\ref{exam:tgm}(ii) for the case $m=0$).

\begin{theor}\label{thm:tangentMilnor}
Let $A$ be a weakly $5$-fold stable ring such that $\frac{1}{2}\in A$. Then for any~$m\geqslant 0$, the homomorphism $Bl\colon TK^M_{m+1}(A)\to\Omega_A^{m}$ is an isomorphism.
\end{theor}

Combining Remark~\ref{exam:wstable}, Theorem~\ref{thm:tangentMilnor}, and Lemma~\ref{lemma:TL}$(i)$, we obtain the following isomorphism, which is crucial for our main results.

\begin{prop}\label{prop:LBl}
There is an isomorphism of group functors
$$
(TL^nK^M_{m+1})_{\z[\frac{1}{2}]}\stackrel{\sim}\lrto (L^n\Omega^m)_{\z[\frac{1}{2}]}\,.
$$
For any $\z[\frac{1}{2}]$-algebra $A$, a collection $f_1,\ldots,f_m\in\LL(A)^*$ of invertible iterated Laurent series, and an element $g\in\LL^n(A)$, this isomorphism sends a symbol $\{1+g\,\varepsilon,f_1,\ldots,f_m\}$ to the differential form ${g\,\frac{df_1}{f_1}\wedge\ldots\wedge\frac{df_m}{f_m}}$.
\end{prop}

\medskip

We will also use the following fact about Milnor $K$-groups, which is proved in~\cite[Lem.\,3.6]{Mor} with a method of Nesterenko and Suslin from~\cite[Lem.\,3.2]{NS} (see also Lemma~2.2 from the paper of Kerz~\cite{Ker} and~\cite[Lem.\,3.5]{GOMilnor}).

\begin{lemma}\label{lemma:Morrow}
Let $A$ be a weakly $5$-fold stable ring. Then for all elements $a,b\in A^*$, there are equalities in $K_2^M(A)$
$$
\{a,b\}=-\{b,a\}\,,\qquad \{a,a\}=\{-1,a\}\,.
$$
\end{lemma}

\subsection{Boundary map for algebraic $K$-groups}\label{subsect:bound}

Let $A$ be an arbitrary ring. We recall the construction of a boundary map (see below for the corresponding references)
$$
\partial_{m+1}\;:\;K_{m+1}\big(A((t))\big)\lrto K_{m}(A)\,.
$$
We show that the map $\partial_{m+1}$ is functorial with respect to $A$, that is, $\partial_{m+1}$ is a morphism of group functors from $LK_{m+1}$ to $K_{m}$ (see Definition~\ref{defin:loopgroup}). Explicitly, this means that for any homomorphism of rings $A\to B$, the following diagram is commutative:
$$
\begin{CD}
K_{m+1}\big(A((t))\big) @>{\partial_{{m+1},A}}>> K_m(A) \\
@VV V @VVV \\
K_{m+1}\big(B((t))\big) @>{\partial_{{m+1}, B}}>> K_m(B)\,,
\end{CD}
$$
where we use the indices ``$A$'' and ``$B$'' in order to specify the maps $\partial_{m+1}$ for the corresponding rings.

\medskip

Let $\Hc(A)$ be the exact category of $A[[t]]$-modules that admit a resolution of length one by finitely generated projective $A[[t]]$-modules and are annihilated by a power of the element~$t$.
A result of Gersten~\cite{Ger} implies the existence of a boundary map ${\tilde\partial_{m+1}\colon K_{m+1}\big(A((t))\big)\to K_m\big(\Hc(A)\big)}$. Gersten's result was later generalized by Grayson~\cite{Gr} (for a more detailed exposition see~\cite[Theor.\,9.1]{S}). In order to show functoriality of $\tilde\partial_{m+1}$, we use the construction from~\cite{Gr}, which we explain below.

\begin{rmk}
A resolution of a module in $\Hc(A)$ is an example of a perfect complex of $R[[t]]$-modules with support on the closed subscheme $\Spec(A)\subset \Spec\big(A[[t]]\big)$ given by the equation $t=0$. Perfect complexes of this type are crucial in the work of Thomason--Trobaugh~\cite{TT}. Results of the latter paper are used by Musicantov and Yom Din~\cite[\S\,4.1]{MY} in order to compare the boundary map for $K$-groups in low degrees with the valuation, the tame symbol, and the residue map when $A$ is a field.
\end{rmk}

Let us introduce some more notation. Given a ring $R$, let $\Pc(R)$ be the exact category of finitely generated projective $R$-modules and let $\Pc^1(R)$ be the exact category of \mbox{$R$-modules} that admit a resolution of length one by finitely generated projective \mbox{$R$-modules}. Let $\Vc(A)$ be the exact category of finitely generated projective \mbox{$A((t))$-modules} that are localizations of finitely generated projective $A[[t]]$-modules. Let~$\Vc^1(A)$ be the exact category of $A((t))$-modules that admit a resolution of length one by $A((t))$-modules in~$\Vc(A)$. By $\rm BQ$ denote the classifying space of the $\rm Q$-construction of an exact category.

By Quillen's resolution theorem~\cite[\S\,4, Theor.\,3, Cor.\,1]{Q}, the embedding of exact categories $\Vc(A)\to\Vc^1(A)$ induces a homotopy equivalence ${\rm BQ}\Vc(A)\to{\rm BQ}\Vc^1(A)$, whence we obtain isomorphisms of $K$-groups
$$
K_m\big(\Vc(A)\big)\simeq K_m\big(\Vc^1(A)\big)\,.
$$
Note that all exact sequences in the category $\Pc\big(A((t))\big)$ split. Besides, the subcategory $\Vc(A)$ is cofinal in $\Pc\big(A((t))\big)$. In other words, any module in $\Pc\big(A((t))\big)$ is a direct summand of a module in~$\Vc(A)$, being a direct summand of a finitely generated free $A((t))$-module. Therefore by~\cite[Prop.\,1.1, 1.3]{Ger}, the corresponding map ${{\rm BQ}\Vc(A)\to{\rm BQ}\Pc\big(A((t))\big)}$ is homotopy equivalent to a covering and we obtain isomorphisms of $K$-groups $$
K_m\big(\Vc(A)\big)\simeq K_m\big(A((t))\big)\,,\qquad m\geqslant 1\,.
$$
(The map $K_0\big(\Vc(A)\big)\to K_0\big(A((t))\big)$ is injective but not surjective in general.)
Finally, consider the following diagram of exact categories:
\begin{equation}\label{eq:seqcats}
\Hc(A)\lrto \Pc^1\big(A[[t]]\big)\lrto \Vc^1(A)\,.
\end{equation}
The point is that~\eqref{eq:seqcats} induces a homotopy fibration
$$
{\rm BQ}\Hc(A)\lrto {\rm BQ}\Pc^1\big(A[[t]]\big)\lrto {\rm BQ}\Vc^1(A)\,.
$$
This statement is essentially ``the Localization Theorem for projective modules'' in~\cite{Gr} and follows from the beginning and the end of the proof of this theorem. Consequently we obtain a long exact sequence of homotopy groups and, in particular, we get a boundary map ${K_{m+1}\big(\Vc^1(A)\big)\to K_m\big(\Hc(A)\big)}$. Applying the above isomorphisms between $K$-groups, we obtain the map
$$
\tilde{\partial}_{m+1}\;:\; K_{m+1}\big(A((t))\big)\lrto K_m\big(\Hc(A)\big)\,.
$$

\medskip

We need the following simple lemma.

\begin{lemma}\label{lemma:tor}
Given a homomorphism of rings $A\to B$, for any $A[[t]]$-module $M$ in~$\Hc(A)$, we have that
$\Tor_1^{A[[t]]}\big(M,B[[t]]\big)=0$.
\end{lemma}
\begin{proof}
Let
$$
0\lrto P\stackrel{\alpha}\lrto Q\lrto M\lrto 0
$$
be a resolution of $M$ by finitely generated projective $A[[t]]$-module. We need to show that the map
$$
\alpha':=\alpha\otimes_{A[[t]]}B[[t]]\;:\; P\otimes_{A[[t]]}B[[t]]\lrto Q\otimes_{A[[t]]}B[[t]]
$$
is injective. Since $M$ is annihilated by a power of $t$, the map $\alpha\otimes_{A[[t]]}A((t))$ is an isomorphism, whence the map $\alpha\otimes_{A[[t]]}B((t))=\alpha'\otimes_{B[[t]]}B((t))$ is an isomorphism as well. Therefore all elements in the $B[[t]]$-module $\Ker(\alpha')$ are annihilated by powers of $t$. On the other hand, $\Ker(\alpha')$ is a submodule of the projective $B[[t]]$-module $P\otimes_{A[[t]]}B[[t]]$ and $t$ is not a zero divisor in the ring $B[[t]]$. Therefore, $\Ker(\alpha')=0$.
\end{proof}

\begin{prop}\label{prop:functtilde}
The map $\tilde\partial_{m+1}\colon K_{m+1}\big(A((t))\big)\to K_m\big(\Hc(A)\big)$ is functorial with respect to a ring~$A$.
\end{prop}
\begin{proof}
We use Kato's idea from the proof of~\cite[\S\,2.1, Lem.\,2]{K1}. Observe that the second and the third categories in diagram~\eqref{eq:seqcats} are not functorial with respect to $A$. To overcome this, we define auxiliary exact subcategories. Let $\varphi\colon A\to B$ be a homomorphism of rings. Let $\Pc^1_{\varphi}\big(A[[t]]\big)$ be the exact category of $A[[t]]$-modules $M$ in $\Pc^1\big(A[[t]]\big)$ such that
$\Tor_1^{A[[t]]}\big(M, B[[t]]\big)=0$. Similarly, let $\Vc^1_{\varphi}(A)$ be the exact category of $A((t))$-modules~$N$ in~$\Vc^1(A)$ such that $\Tor_1^{A((t))}\big(N, B((t))\big)=0$. Lemma~\ref{lemma:tor} asserts that $\Hc(A)$ is a subcategory in $\Pc^1_{\varphi}\big(A[[t]]\big)$. Thus we have the following diagram of exact categories with arrows being exact functors:
$$
\begin{CD}
\Hc(A) @>>>  \Pc^1\big(A[[t]]\big) @>>>  \Vc^1(A) \\
@AAA @AAA @AAA \\
\Hc(A) @>>>  \Pc^1_{\varphi}\big(A[[t]]\big) @>>>  \Vc^1_{\varphi}(A) \\
@VVV @VVV @VVV \\
\Hc(B) @>>>  \Pc^1\big(B[[t]]\big) @>>>  \Vc^1(B) \\
\end{CD}
$$
By Quillen's resolution theorem, the up going arrows in this diagram induce homotopy equivalences between the corresponding $\rm BQ$-spaces. Therefore by functoriality of the boundary map in a long exact sequence of homotopy groups, we obtain functoriality for the map $\tilde\partial_{m+1}$.
\end{proof}

\medskip

Now let us construct a map from $K_m\big(\Hc(A)\big)$ to $K_m(A)$. The second named author and Zhu proved in~\cite[Prop.\,7.1]{OZ1} that any $A[[t]]$-module in the category $\Hc(A)$ is a finitely generated projective $A$-module (cf.~\cite[\S\,3.3]{G} and~\cite[\S\,2.1, Lem.\,1]{K1}). Thus we have an exact functor $\Hc(A)\to\Pc(A)$, which induces maps between $K$-groups
$$
I_m\;:\; K_m\big(\Hc(A)\big)\lrto K_m(A)\,.
$$
Let us make the following side remark.

\begin{rmk}
The exact category $\Hc(A)$ is equivalent to the exact category $\underline{\rm Nil}(A)$ from~\cite[p.\,236]{Gr}, whose objects are pairs $(M,f)$, where
$M$ is a finitely generated projective \mbox{$A$-module} and $f$ is a nilpotent
endomorphism of $M$. Indeed, if an $A[[t]]$-module~$M$ is in the category $\Hc(A)$, then, as mentioned above, $M$ is a finitely generated projective $A$-module and, by definition, $t$ acts on $M$ as a nilpotent endomorphism. Conversely, for any pair $(M,f)$ as above, we have an exact sequence of $A[[t]]$-modules
$$
0 \lrto M[[t]]  \stackrel{\alpha}{\lrto}  M[[t]]  \stackrel{\beta}{\lrto} M_f  \lrto 0\,,
$$
where $M_f$ is an $A[[t]]$-module $M$ with $t$ acting as $f$ and we put
$$
\alpha\Big(\,\mbox{$\sum\limits_{l\geqslant 0} a_l t^l$}\Big):=\mbox{$\sum\limits_{l\geqslant 0}\big(a_{l-1}-f(a_l)\big)t^l$}\,,
\qquad \beta\Big(\,\mbox{$\sum\limits_{l\geqslant 0} b_l t^l$}\Big):=\mbox{$\sum\limits_{l\geqslant 0} f^l(b_l)$}\,,
$$
where we put $a_{-1}:=0$. Note that if $\beta\Big(\sum\limits_{l\geqslant 0} b_l t^l\Big) =0$, then $\sum\limits_{l\geqslant 0} b_l t^l = \alpha\Big(\sum\limits_{l\geqslant 0} a_l t^l\Big)$, where
$a_l = \sum\limits_{j \geqslant 0} f^j (b_{l+j+1})$ for any integer $l \geqslant 0$. Since $f$ is nilpotent, the $A[[t]]$-module $M_f$ is from $\Hc(A)$.
\end{rmk}

We need the following simple lemma.

\begin{lemma}\label{lemma:commuteextscal}
Let $A\to B$ be a homomorphism of rings and let $M$ be a module over~$A[[t]]$ such that $t^l\cdot M=0$ for some integer $l\geqslant 0$. Then the natural homomorphism of \mbox{$B$-modules}
$M \otimes_A B \to M \otimes_{A[[t]]}B[[t]]$ is an isomorphism.
\end{lemma}
\begin{proof}
Define an $A[[t]]$-module $T:=A[[t]] / t^l A[[t]]$. Then the lemma holds for $T$, that is, the natural homomorphism of $B$-modules
$$
T \otimes_A B  \lrto  T \otimes_{A[[t]]} B[[t]]
$$
is an isomorphism. Indeed, both $B$-modules are isomorphic to the $B$-module ${B[[t]]/ t^l B[[t]]}$. To prove the lemma in the general case, we take the composition of the following isomorphisms of $B$-modules:
\begin{multline}  \nonumber
\hspace{-0.3cm}
M \otimes_A B \stackrel{\sim}\lrto (M \otimes_{A[[t]]} T)  \otimes_A B  \stackrel{\sim}\lrto   M \otimes_{A[[t]]} (T \otimes_A B)  \stackrel{\sim}\lrto \\  \stackrel{\sim}\lrto M \otimes_{A[[t]]}  (T \otimes_{A[[t]]} B[[t]])  \stackrel{\sim}\lrto (M \otimes_{A[[t]]}  T  )  \otimes_{A[[t]]}  B[[t]]  \stackrel{\sim}\lrto M \otimes_{A[[t]]} B[[t]]\,.
\end{multline}
\end{proof}

Lemma~\ref{lemma:commuteextscal} directly implies that the map $I_m$ is functorial with respect to $A$.

\medskip

Finally, define a boundary map for algebraic $K$-groups by the formula
\begin{equation}\label{eq:bounKtheory}
\partial_{m+1}:=I_m\circ\tilde\partial_{m+1}\;:\;K_{m+1}\big(A((t))\big)\lrto K_m(A)\,.
\end{equation}
Combining functoriality of $I_m$ together with functoriality of $\tilde\partial_{m+1}$ given by Proposition~\ref{prop:functtilde}, we obtain the following important result.

\begin{prop}  \label{nat-tr}
The boundary map $\partial_{m+1}$ is a morphism of group functors from~$LK_{m+1}$ to $K_{m}$.
\end{prop}

\medskip

We will also need the following special property of the map $\partial_{m+1}$. By $\lambda_m\colon K_m\big(A[[t]]\big)\to K_m\big(A((t))\big)$ denote the map induced by the homomorphism of rings $A[[t]]\to A((t))$. By $\alpha\mapsto\bar\alpha$ denote the map $K_m\big(A[[t]]\big)\to K_m(A)$ induced by the homomorphism of rings ${A[[t]]\to A[[t]]/(t)\simeq A}$. It follows from~\cite[\S\,2.4, Prop.\,5]{K1} that for any element \mbox{$\alpha\in K_m\big(A[[t]]\big)$}, there is an equality in $K_m(A)$
\begin{equation}\label{projection-formula}
\partial_{m+1}\big(\lambda_m(\alpha)\cdot t\big)=\bar{\alpha}\,,
\end{equation}
where we consider $t$ as an element in $K_1\big(A((t))\big)$. In particular, we have that $\partial_1(t)=1$, where we consider $1$ as an element in $K_0(A)$.

\section{Main results}\label{int-CC-sym}\label{sect:main}

For short, we will use the following expression: given two sets $P$, $Q$, and a map $\alpha\colon P\to Q$, we say that ``an element $p\in P$ is uniquely determined by the element $\alpha(p)\in Q$'' if we have that $\alpha^{-1}\big(\alpha(p)\big)=\{p\}$. Thus the map $\alpha$ is injective if and only if this condition holds for any element $p\in P$. In the situations that follow, the sets $P$, $Q$, and the map~$\alpha$ will be clear from the context and they will not be mentioned explicitly.

\subsection{Additive symbol}\label{subsect:add}

First we discuss a simpler version of what we shall do later for the Contou-Carr\`ere symbol. The general plan in this simpler case is quite similar to that in the case of the Contou-Carr\`ere symbol.

\begin{defin}\label{def:addsymb}
An {\it additive symbol} $\nu_n$ is the following composition of morphisms of group functors (see formula~\eqref{eq:bounKtheory} and Proposition~\ref{nat-tr}):
$$
\begin{CD}
\nu_{n}\; : \; \lo^nK_{n}^M @>>> \lo^n K_{n} @>{\partial_1 \cdot \ldots \cdot \partial_n}>> K_0 @>{\rk}>> \uz\,.
\end{CD}
$$
\end{defin}

It turns out that all morphisms of group functors from $\lo^nK_{n}^M$ to $\uz$ can be described explicitly as follows. The symbol $\{t_1,\ldots,t_n\}$ defines a morphism of group functors
$$
\Xi\;:\;\uz\lrto\lo^nK^M_n\,,\qquad \underline{i}\longmapsto \underline{i}\cdot\{t_1,\ldots,t_n\}\,.
$$
Note that for any ring $R$, there is an isomorphism of groups $\Hom^{gr}_R(\uz_R,\uz_R)\simeq\uz(R)$.

\begin{prop}\label{prop-addsymb}
For any ring $R$, the natural homomorphism
$$
\Xi_R^*\;:\;\Hom^{gr}_R\big((\lo^nK^M_n)_R,\uz_R\big)\lrto \uz(R)\,,\qquad \Phi{\longmapsto} \Phi\circ \Xi_R\,,
$$
is an isomorphism.
\end{prop}

Before we give a proof of Proposition~\ref{prop-addsymb}, let us discuss some of its corollaries. Explicitly, Proposition~\ref{prop-addsymb} means the following: suppose that the values of two morphisms of group functors $\Phi,\Phi'\colon (L^nK^M_n)_R\to \z_R$ coincide at the symbol $\{t_1,\ldots,t_n\}$. Then we have that $\Phi=\Phi'$.

By formula~\eqref{projection-formula} in Subsection~\ref{subsect:bound}, for any ring $A$, there is an equality ${\nu_{n}\{t_1,\ldots,t_n\}=1}$ in $\z(A)$.
Hence there is an equality
\begin{equation}\label{eq:evaladd}
\Xi^*(\nu_{n})=1\in\uz(\z)=\z\,.
\end{equation}
Thus Proposition~\ref{prop-addsymb} implies the following universal property of the additive symbol~$\nu_n$: for any morphism of group functors $\Phi\colon (L^nK^M_n)_R\to \uz_R$, there is a locally constant integer $\underline{i}\in\uz(R)$ such that $\Phi=\underline{i}\cdot\nu_n$.

\medskip

Now let us give the proof of Proposition~\ref{prop-addsymb}, which is based on the absolute connectedness of~$(\lo^n\gm)^0$ over $\z$ (see Proposition~\ref{prop:densegm}$(iv)$).

\begin{proof}[Proof of Proposition~\ref{prop-addsymb}]
Formula~\eqref{eq:evaladd} implies that the homomorphism $\Xi^*_R$ is surjective. Thus we need to show that $\Xi^*_R$ is injective. In other words, given a morphism of group functors $\Phi\colon(L^nK^M_n)_R\to \uz_R$, we need to show that $\Phi$ is uniquely determined by the composition $\Phi\circ\Xi_R$.

Consider the following composition of morphisms of functors:
$$
\widetilde{\Phi}\;:\; (L^n\gm)^{\times n}_R\lrto (\lo^nK^M_n)_R\stackrel{\Phi}\lrto \uz_R\,,
$$
where the first morphism sends a collection of invertible iterated Laurent series $(f_1,\ldots,f_{n})$ to the symbol $\{f_1,\ldots,f_{n}\}$. Clearly, $\Phi$ is uniquely determined by $\widetilde{\Phi}$.

Actually, $\widetilde{\Phi}$ is a morphism of ind-schemes over $R$. For any $R$-algebra $A$, a collection $f_2,\ldots,f_n\in L^n\gm(A)$ defines a morphism of ind-schemes over $A$
$$
\alpha\;:\;(\lo^n\gm)^0_A\lrto \uz_A\,,\qquad f\longmapsto \widetilde{\Phi}(f,f_2,\ldots,f_n)\,.
$$
By Proposition~\ref{prop:densegm}$(iv)$, the ind-scheme~$(\lo^n\gm)^0$  is absolutely connected over $\z$. Thus by Proposition~\ref{lemma:conntriv}, the morphism~$\alpha$ is equal to zero. Hence the restriction of $\widetilde{\Phi}$ to the ind-closed subscheme $(\lo^n\gm)^{0}_R\times (\lo^n\gm)^{\times (n-1)}_R$ is equal to zero as well.

By Remark~\ref{exam:wstable} and Lemma~\ref{lemma:Morrow}, symbols in $L^nK^M_n(A)=K^M_{n}\big(\LL^n(A)\big)$ are antisymmetric. Thus we obtain that $\widetilde{\Phi}$ factors through the composition of morphisms
\begin{equation}\label{eq:compmorphisms}
(\lo^n\gm)^{\times n}_R\stackrel{\nu^{\times n}}\lrto(\uz^n_R)^{\times n} \stackrel{\det}\lrto\uz_R\,,
\end{equation}
because by Definition~\ref{defin:specialmult}, there is an isomorphism $\nu\colon\lo^n\gm/(\lo^n\gm)^0\stackrel{\sim}\lrto \uz^n$. Here for any ring $A$, we consider elements in $\uz^n(A)$ as columns with (locally constant) integral entries and elements in $(\uz^n)^{\times n}(A)$ as $(n\times n)$-matrices with (locally constant) integral entries.

We obtain that $\widetilde{\Phi}$ is uniquely determined by its value at the collection $(t_1,\ldots,t_n)$, because this collection is sent to $1\in\uz_R(R)$ under the above composition of morphisms~\eqref{eq:compmorphisms}. This proves the proposition.
\end{proof}

\medskip

Note that in the course of the proof of Proposition~\ref{prop-addsymb} we have also obtained an explicit formula for the additive symbol $\nu_n$. Namely, for any ring $A$ and a collection ${f_1,\ldots,f_n\in \lo^n\gm(A)}$ of invertible iterated Laurent series, there is an equality
\begin{equation}\label{eq:explicadd}
\nu_n\{f_1,\ldots,f_n\}=\det\big(\nu(f_1),\ldots,\nu(f_n)\big)\,.
\end{equation}
Indeed, the proof of Proposition~\ref{prop-addsymb} implies that the generator of the cyclic group $\Hom^{gr}_R\big((\lo^nK^M_n)_R,\uz_R\big)\simeq \z$ is given by the right hand side of formula~\eqref{eq:explicadd}. On the other hand, as noticed before the proof, $\nu_n$ is this generator.

\medskip

Actually, one can define the morphism of group functors $\nu_n\colon \lo^nK^M_n\to\uz$ explicitly just by the right hand side of formula~\eqref{eq:explicadd}, not using the boundary map for \mbox{$K$-groups} and connectedness of $(\lo^n\gm)^0$. When following this approach, one needs to check directly the Steinberg relations, which is done explicitly as follows.

\begin{lemma}\label{lemma:Steindet}
Let $A$ be a ring and let $f_1,\ldots,f_n\in \lo^n\gm(A)$ be a collection of invertible iterated Laurent series such that $f_1+f_2=1$. Then there is an equality ${\det\big(\nu(f_1),\ldots,\nu(f_n)\big)=0}$.
\end{lemma}
\begin{proof}
Changing the ring $A$ if needed, without loss of generality, we may assume that~$\nu(f_1)$ and $\nu(1-f_1)$ belong to the subgroup $\Z^n \subset \uz(A)^n$. Further, we can assume that $\nu(f_1)\ne 0$ (otherwise, there is nothing to prove). If $\nu(f_1)>0$, then $\nu(1-f_1)=0$ and if $\nu(f_1)<0$, then $\nu(f_1)=\nu(1-f_1)$ (see Proposition~\ref{prop-decomp} and Example~\ref{examp:projections}(i)). This proves the lemma.
\end{proof}

\medskip

Let us mention one more property of the right hand side of formula~\eqref{eq:explicadd}, which we will use later. We have a natural morphism of group functors $\zeta\colon\uz\to\ga$. Note that if a ring $A$ has positive characteristic, then the map $\zeta\colon\uz(A)\to A$ is not injective.

\begin{prop}\label{prop-integer-res}
For any ring $A$ and a collection $f_1, \ldots, f_n  \in \lo^n\gm(A)$ of invertible iterated Laurent series, there is an equality in $A$
\begin{equation}  \label{integer-res}
\zeta\Big(\det\big(\nu(f_1), \ldots, \nu(f_n)\big)\Big) = \res \left(\frac{df_1}{f_1}  \wedge \ldots \wedge \frac{df_n}{f_n} \right)\,.
\end{equation}
\end{prop}
\begin{proof}
Both sides of formula~\eqref{integer-res} define morphisms of ind-schemes from $(L^n \gm)^{\times n}$ to~$\ga$. Therefore by Theorem~\ref{cor:uniq} applied to $R=\Z$ and $S=\Q$, it is enough to prove the proposition when $A$ is a $\Q$-algebra, which we assume from now on.

Both sides of formula~\eqref{integer-res} are multilinear alternating maps. Therefore we can assume that $f_i$ belong to subgroups of $\lo^n\gm(A)$ from decomposition~\eqref{eq:decommult} in Subsection~\ref{sp-sub}. The case~${f_1\in \gm(A)}$ is trivial. Suppose that $f_1 \in (\lo^n\gm)^{\sharp}(A)$. Then by Definition~\ref{defin:specialmult}, we have that $\nu(f_1)=0$ and thus, ${\zeta\Big(\det\big(\nu(f_1), \ldots, \nu(f_n)\big)\Big)=0}$. On the other hand, using Lemmas~\ref{dif-form}$(i)$ and~\ref{lemma:cohom}, we obtain the equalities
$$
\res \left(\frac{df_1}{f_1}   \wedge \ldots \wedge \frac{df_n}{f_n} \right) = \res \left( d \left( \log (f_1 )  \frac{df_2}{f_2}    \wedge \ldots \wedge \frac{df_n}{f_n}  \right)          \right)=0\,.
$$
It remains to consider the case when all $f_i$ are from the subgroup $\uz^n(A)\subset L^n\gm(A)$. By multilinearity and the alternating property, it is enough to consider the case ${f_1=t_1,\ldots,f_n=t_n}$. Clearly, we have the equalities
$$
\zeta\Big(\det\big(\nu(t_1), \ldots, \nu(t_n)\big)\Big) = 1= \res \left( \frac{dt_1}{t_1}  \wedge \ldots \wedge  \frac{dt_n}{t_n}  \right)  \,,
$$
which finishes the proof.
\end{proof}

\begin{rmk}
When $n=2$ and $A$ is a $\Q$-algebra, Proposition~\ref{prop-integer-res} coincides with~\cite[Lem.\,4.1]{OZ1}.
\end{rmk}

\subsection{Contou-Carr\`ere symbol}\label{subsect:key}

Here is our main object of study.

\begin{defin}\label{defin:CC}
A {\it Contou-Carr\`ere symbol $CC_n$} is the following composition of morphisms of group functors (see formula~\eqref{eq:bounKtheory} and Proposition~\ref{nat-tr}):
\begin{equation}\label{eq:eta1}
\begin{CD}
CC_n\;:\;\lo^nK_{n+1}^M
@>>> \lo^nK_{n+1} @>{\partial_2  \cdot \ldots \cdot \partial_{n+1}}>>  K_1
@>{\det}>> \gm\,.
\end{CD}
\end{equation}
\end{defin}

Thus $CC_n$ is a character of the group functor $L^nK^M_{n+1}$. In particular, for any ring $A$, we have a homomorphism of groups, which we denote similarly
$$
CC_n\;:\; K^M_{n+1}\big(A((t_1))\ldots((t_n))\big)\lrto A^*\,.
$$
It turns out that all characters of $L^nK^M_{n+1}$ can described explicitly as follows. Define a morphism of group functors
$$
\Theta\;:\;\gm\lrto L^nK^M_{n+1}\,,\qquad a\longmapsto \{a,t_1,\ldots,t_n\}\,,
$$
where $a\in A^*$ for a ring $A$. Recall that for any ring $R$, there is an isomorphism of groups ${\rm X}\big((\gm)_R\big)\simeq\uz(R)$ (see Definition~\ref{defin:char} and Lemma~\ref{lemma:diffgagm}$(ii)$).

\medskip

A ring $R$ is said to be {\it torsion free over $\Z$} if $R$ has no torsion as a $\Z$-module or, equivalently, if the natural homomorphism of rings $R\to R\otimes_{\Z}\Q$ is injective. In particular, if this holds, then the natural homomorphism of rings $\z\to R$ is also injective (however, the latter condition is not equivalent to being torsion free over $\z$). Note that $R$ is torsion free over $\z$ if and only if~$R$ is flat over $\Z$.

\begin{theor}\label{theor-key}
Suppose that a ring $R$ is torsion free over $\Z$. Then the natural homomorphism of groups
$$
\Theta_R^*\;:\;{\rm X}\big((L^nK^M_{n+1})_R\big)\lrto \uz(R)\,,\qquad\Phi\longmapsto \Phi\circ \Theta_R\,,
$$
is an isomorphism.
\end{theor}

Theorem~\ref{theor-key} is proved in Subsection~\ref{subsect:proofrigid} below. Combining Theorem~\ref{theor-key} with the theory of thick ind-cones (Theorem~\ref{cor:uniq}), we obtain the following explicit result.

\begin{corol}\label{corol:rigid}
Let
$$
\Psi,\Psi'\;:\; (L^n\gm)^{\times(n+1)}_R\lrto (\gm)_R
$$
be multilinear morphisms of functors over a ring $R$. Suppose the following conditions:
\begin{itemize}
\item[(i)]
There is an $R$-algebra $A_0$ and an element $a_0\in A_0^*$ such that $a_0$ is not a root of unit and there is an equality ${\Psi\{a_0,t_1,\ldots,t_n\}=\Psi'\{a_0,t_1,\ldots,t_n\}}$ in $A_0^*$.
\item[(ii)]
There is an embedding of rings $R\subset S$ such that $S$ is torsion free over $\z$ and the morphisms of functors $\Psi_S$ and $\Psi'_S$ satisfy the Steinberg relations, that is, $\Psi_S$ and $\Psi'_S$ are compositions of the natural multilinear morphism of functors ${(L^n\gm)^{\times(n+1)}_S\to (L^nK^M_{n+1})_S}$ with morphisms of group functors $\Phi_S,\Phi'_S\colon (L^nK^M_{n+1})_S\to (\gm)_S$.
\end{itemize}
Then we have that $\Psi=\Psi'$.
\end{corol}
\begin{proof}
Define a morphism of group functors
$$
\Gamma\;:\;\gm\lrto(\lo^n\gm)^{\times(n+1)}\,,\qquad a\longmapsto (a,t_1,\ldots,t_n)\,,
$$
where $a\in A^*$ for a ring $A$. Condition~$(i)$ implies the equality $\Psi\circ\Gamma_R=\Psi'\circ\Gamma_R$ between the endomorphisms of $(\gm)_R$. Using condition~$(ii)$ and taking the extension of scalars from $R$ to $S$, we obtain the equality $\Phi_S\circ \Theta_S=\Phi'_S\circ\Theta_S$. By Theorem~\ref{theor-key} applied over the ring $S$, we see that $\Phi_S=\Phi'_S$. This directly implies that $\Psi_S=\Psi'_S$. Finally, by Theorem~\ref{cor:uniq}, we get the required equality~$\Psi=\Psi'$.
\end{proof}

\medskip

\begin{rmk}\label{examp:eta}
By formula~\eqref{projection-formula} in Subsection~\ref{subsect:bound}, for any ring $A$ and an element $a\in A^*$, there is an equality $CC_{n}\{a,t_1,\ldots,t_n\}=a$ in $A^*$. Hence there is an equality
$$
\Theta^*(CC_{n})=1\in\uz(\z)=\z\,.
$$
\end{rmk}

Theorem~\ref{theor-key} with Remark~\ref{examp:eta} imply the following universal property of the Contou-Carr\`ere symbol.

\begin{theor}\label{theor:intCC}
Suppose that a ring $R$ is torsion free over $\z$. Then for any morphism of group functors $\Phi\colon (\lo^nK_{n+1}^M)_R\to(\gm)_R$, there is a locally constant integer $\underline{i}\in\uz(R)$ such that $\Phi=(CC_n)^{\underline{i}}$.
\end{theor}

\medskip

Let us make a remark concerning another definitions of the Contou-Carr\`ere symbol. For simplicity, below we consider various multilinear morphisms to $\gm$ up to taking the inverse.

\begin{rmk}
\hspace{0cm}
\begin{itemize}
\item[(i)]
The classical one-dimensional symbol was constructed by Contou-Carr\`ere himself in~\cite{CC1} as an unexplicitly defined bilinear morphism of functors $(L\gm)^{\times 2}\to \gm$. Also, it was given in~\cite{CC1} an explicit formula for the one-dimensional symbol over $\Q$-algebras. Another, but equivalent, explicit formula over $\Q$-algebras was given by Deligne~\cite[\S\,2.9]{Del}. Beilinson, Bloch, and Esnault constructed in~\cite[\S\,3.1]{BBE} a bilinear morphism of functors $(L\gm)^{\times 2}\to\gm$ by using the commutator in a super central extension of the group functor~$L\gm$. It was shown in~\cite[Prop.\,3.3]{BBE} that this morphism coincides with the Contou-Carr\`ere symbol. Independently of the work~\cite{BBE}, Anderson and Pablos Romo~\cite{AP} studied in a similar way a particular case for Laurent series over Artinian rings. It was shown in~\cite[Theor.\,7.2]{OZ1} that the one-dimensional Contou-Carr\`ere symbol is equal to the symbol from Definition~\ref{defin:CC} with $n=1$. This fact can be proved differently by using Corollary~\ref{corol:rigid} and Remark~\ref{examp:eta}. With this approach, one also needs to check that the classical one-dimensional Contou-Carr\`ere symbol satisfies the Steinberg relations for Laurent series over \mbox{$\Q$-algebras}, which can be done explicitly as in the proof of Proposition~\ref{CC-Q-Milnor} below.
\item[(ii)]
In~\cite[Def.\,3.4]{OZ1} it was defined by an explicit formula a multilinear morphism of functors ${(L^2\gm)^{\times 3}_{\Q}\to (\gm)_{\Q}}$ on the category of $\Q$-algebras. Besides, in~\cite[\S\,5.3]{OZ1} it was defined a multilinear morphism of functors $(L^2\gm)^{\times 3}\to\gm$ using a commutator in a central extension of $L^2\gm$ by a Picard category. The equality between this morphism and the above morphism over $\Q$-algebras was proved in~\cite[Theor.\,5.9]{OZ1}. The equality between these morphisms and the symbol from Definition~\ref{defin:CC} with $n=2$ was proved in~\cite[Theor.\,7.2]{OZ1}. Similarly, as in item~(i), this equality can be also proved by using Corollary~\ref{corol:rigid}, Remark~\ref{examp:eta}, and the Steinberg property from Proposition~\ref{CC-Q-Milnor} below.
\item[(iii)]
Braunling, Groechenig, and Wolfson constructed in~\cite[Def.\,4.9]{BGW} a multilinear morphism of functors
${(\cdot, \ldots, \cdot )\colon (L^n \gm)^{\times (n+1)}\to\gm}$ using a commutator in a spectral extension of $L^n\gm$ from their previous work~\cite{BGW0}. It follows from the construction that the morphism $(\cdot,\ldots,\cdot)$ satisfies the Steinberg relations. Thus Corollary~\ref{corol:rigid} and Remark~\ref{examp:eta} imply that $(\cdot,\ldots,\cdot)$ is equal to the higher Contou-Carr\`ere symbol~$CC_n$ as defined in the present paper (see Definition~\ref{defin:CC}).
\end{itemize}
\end{rmk}

\subsection{Proof of Theorem~\ref{theor-key}}\label{subsect:proofrigid}

In this subsection we prove Theorem~\ref{theor-key}. We will uses several auxiliary results. In what follows $R$ is a ring and
$$
\Phi\;:\; (\lo^nK_{n+1}^M)_R\lrto(\gm)_R
$$
is a character. Define the composition of morphisms of functors
$$
\widetilde{\Phi}\;:\;(\lo^n\gm)^{\times (n+1)}_R\lrto (L^nK_M^{n+1})_R\stackrel{\Phi}\lrto (\gm)_R\,,
$$
where the first morphism sends a collection of invertible iterated Laurent series $(f_1,\ldots,f_{n+1})$ to the symbol $\{f_1,\ldots,f_{n+1}\}$. Actually, $\widetilde{\Phi}$ is a morphism of ind-affine schemes.

\medskip

The following result is based on the theory of thick ind-cones (see Proposition~\ref{prop:geninj} and Theorem~\ref{cor:uniq}).

\begin{lemma}\label{prop:overQ}
Suppose that a ring $R$ is torsion free over $\z$. Then any character $\Phi$ as above is uniquely determined by the morphism of group functors ${\Phi_{S}\colon (\lo^nK_{n+1}^M)_S\to(\gm)_S}$ over~$S$, where $S:=R\otimes_\z \Q$.
\end{lemma}
\begin{proof}
Clearly, $\Phi$ is uniquely determined by $\widetilde{\Phi}$. Since $R$ is torsion free over~$\z$, the natural homomorphism of rings $R\to S$ is injective. Therefore by Theorem~\ref{cor:uniq}, the morphism of functors~$\widetilde{\Phi}$ is uniquely determined by the morphism of functors~$\widetilde{\Phi}_{S}$, whence $\Phi$ is uniquely determined by $\Phi_{S}$.
\end{proof}

\medskip

The following result is based on the antisymmetric property of symbols in the Milnor $K$-groups of rings of iterated Laurent series (see Remark~\ref{exam:wstable} and Lemma~\ref{lemma:Morrow}).

\begin{lemma}\label{lemma:udsharp}
For any ring $R$, we have that any character $\Phi$ as above is uniquely determined by the restriction of $\widetilde{\Phi}$ to the ind-closed subscheme ${(\lo^n\gm)^{0}_R\times (\lo^n\gm)^{\times n}_R\subset (\lo^n\gm)^{\times (n+1)}_R}$.
\end{lemma}
\begin{proof}
By Remark~\ref{exam:wstable} and Lemma~\ref{lemma:Morrow}, symbols in $L^nK^M_{m}(A)=K^M_{m}\big(\LL^n(A)\big)$, $m\geqslant 2$, are antisymmetric. Thus decomposition~\eqref{eq:decommult} in Subsection~\ref{sp-sub} implies that $\Phi$ is uniquely determined by the restrictions of $\widetilde{\Phi}$ to the ind-closed subschemes ${(\lo^n\gm)^{0}_R\times (\lo^n\gm)^{\times n}_R}$ and ${(\uz^n)^{\times(n+1)}_R}$ of ${(\lo^n\gm)^{\times (n+1)}_R}$.

By multilinearity, we see that the restriction of $\widetilde{\Phi}$ to ${(\uz^n)^{\times(n+1)}_R}$ is uniquely determined by the values of $\Phi$ at symbols $\{f_1,\ldots,f_{n+1}\}$, where each $f_i$, $1\leqslant i\leqslant n+1$, is equal to some $t_j$, $1\leqslant j\leqslant n$. By Dirichlet's box principle, we have that $f_i=f_j$ for some $i\ne j$. By the above antisymmetric property of symbols, we may assume that $f_1=f_2$. By Remark~\ref{exam:wstable} and Lemma~\ref{lemma:Morrow} again, for any element $f\in\lo^n\gm(A)$, there is an equality $\{f,f\}=\{-1,f\}$ in $K_2^M\big(\LL^n(A)\big)$. Therefore, $\{f_1,f_2,\ldots,f_{n+1}\}=\{-1,f_2,\ldots,f_{n+1}\}$. Since $-1\in\gm(A)\subset (\lo^n\gm)^0(A)$, this proves the lemma.
\end{proof}

\medskip

The following result is based on the absolute connectedness of $(\lo^n\gm)^0$ over $\z$ (see Proposition~\ref{prop:densegm}$(iv)$).

\begin{lemma}\label{lemma:trivgm}
For any ring $R$ and any character $\Phi$ as above, the restriction of $\widetilde{\Phi}$ to the ind-closed subscheme ${(\gm)_R\times \big((\lo^n\gm)_R^0\big)^{\times n}\subset (\lo^n\gm)^{\times (n+1)}_R}$ is a constant morphism whose value is equal to $1\in \gm(R)$.
\end{lemma}
\begin{proof}
By multilinearity, the restriction of $\widetilde{\Phi}$ to the ind-closed subscheme ${(\gm)_R\times \big((\lo^n\gm)_R^0\big)^{\times n}}$ defines a morphism of group ind-schemes
$$
\alpha\;:\;\big((\lo^n\gm)_R^0\big)^{\times n}\lrto \underline{\Hom}^{gr}_R\big((\gm)_R,(\gm)_R\big)\simeq \uz_R\,,
$$
where $\underline{\Hom}^{gr}$ is the internal Hom for group functors (see Section~\ref{sect:notation} and Lemma~\ref{lemma:diffgagm}$(ii)$). By Proposition~\ref{prop:densegm}$(iv)$, we have that  $\big((\lo^n\gm)_R^0\big)^{\times n}$ is a connected ind-scheme over~$R$. Thus by Proposition~\ref{lemma:conntriv}, we have that $\alpha=0$. This proves the lemma.
\end{proof}

\medskip

Now we are ready to prove Theorem~\ref{theor-key}.

\begin{proof}[Proof of Theorem~\ref{theor-key}]

\hspace{0cm}

\medskip

{\it Step 1.}
The existence of a functorial boundary map for algebraic $K$-groups (see Subsection~\ref{subsect:bound}) implies that the map~$\Theta_R^*$ is surjective. More precisely, this follows from Remark~\ref{examp:eta}. Thus we need to show that the map~$\Theta_R^*$ is injective, that is, that a morphism of group functors ${\Phi\colon(\lo^nK^M_{n+1})_R\to (\gm)_R}$ is uniquely determined by the composition~${\Phi\circ\Theta_R}$.

\medskip

{\it Step 2.}
By Lemma~\ref{prop:overQ}, we may assume that $R$ is a $\Q$-algebra, which we do from now on.

Let $A$ be an $R$-algebra and let $f_2,\ldots,f_{n+1}\in \lo^n\gm(A)$ be a collection of $n$ invertible iterated Laurent series. We obtain a character
$$
\chi\;:\;(\lo^n\gm)^{0}_A\lrto (\gm)_A\,,\qquad f\longmapsto \Phi\{f,f_2,\ldots,f_{n+1}\}\,.
$$
By Lemma~\ref{lemma:udsharp}, $\Phi$ is uniquely determined by all characters $\chi$ of this kind.

Since $R$ is a $\Q$-algebra, $A$ is a $\Q$-algebra as well. Thus a property of the exponential map from Proposition~\ref{prop:charactgm} implies that each character~$\chi$ as above is uniquely determined by its differential $T\chi$.  This implies that $\Phi$ is uniquely determined by its differential~$T\Phi\colon (TL^nK^M_{n+1})_R\to T(\gm)_R$.

Now we use the description of the tangent space to Milnor $K$-groups in terms of differential forms (see Subsection~\ref{tangent-space}). Let
$$
\Psi\colon(L^n\Omega^n)_R\lrto (\ga)_R
$$
be the morphism of group functors that corresponds to the differential $T\Phi$ under the isomorphism from Proposition~\ref{prop:LBl} and the isomorphism from Example~\ref{exam:tgm}(ii) (here we use that $2$ is invertible in the $\Q$-algebra $R$). What was said above implies that $\Phi$ is uniquely determined by~$\Psi$.

\medskip

{\it Step 3.} Using Lemma~\ref{lemma:trivgm}, we shall show in this step that $\Psi$ equals zero on exact differential forms, that is, that the following morphism of ind-affine schemes
$$
\psi\;:\;(\lo^n\ga)^{\times n}_R\lrto (\ga)_R\,,\qquad (g_1,\ldots,g_n)\longmapsto \Psi(dg_1\wedge\ldots\wedge dg_n)
$$
is equal to zero.

By definition, $\psi$ is a regular function in $\OO\big((\lo^n\ga)^{\times n}_R\big)$. By Proposition~\ref{prop:densegm}$(i)$, the function $\psi$ is uniquely determined by its restriction to the dense ind-closed subscheme ${\big((\lo^n\ga)^{\sharp}_R\big)^{\times n}\subset(\lo^n\ga)^{\times n}_R}$. Thus it is enough to show that for any $R$-algebra $A$ and elements ${g_1,\ldots,g_n\in (\lo^n\ga)^{\sharp}(A)}$, we have $\Psi(dg_1\wedge\ldots\wedge dg_n)=0$.

By Lemma~\ref{dif-form}$(i)$ and Proposition~\ref{log-map}, for any element $g\in (\lo^n\ga)^{\sharp}(A)$, there is an equality $\frac{d\exp(g)}{\exp(g)}=dg$. Therefore the isomorphism from Proposition~\ref{prop:LBl} sends the symbol $\{1+\varepsilon,\exp(g_1),\ldots,\exp(g_n)\}$ to the differential form $dg_1\wedge\ldots\wedge dg_{n}$, where $\varepsilon^2=0$. Together with Example~\ref{exam:tgm}(ii), this implies the following equality in the ring~$A[\varepsilon]^*$:
$$
1+\Psi(dg_1\wedge\ldots \wedge dg_n)\,\varepsilon=\Phi\{1+\varepsilon,\exp(g_1),\ldots,\exp(g_n)\}\,.
$$
Since $1+\varepsilon\in \gm\big(A[\varepsilon]\big)$ and $\exp(g_i)\in(\lo^n\gm)^{\sharp}(A)\subset(\lo^n\gm)^{0}\big(A[\varepsilon]\big)$ for any $i$, $1\leqslant i\leqslant n$, applying Lemma~\ref{lemma:trivgm}, we obtain the equality $\Psi(dg_1\wedge\ldots\wedge dg_n)=0$.

\medskip

{\it Step 4.}
By Proposition~\ref{lemma:algiterforms}, $\Psi$ factors uniquely through the natural morphism of group functors $(L^n\Omega^n)_R\to(\widetilde{L^n\Omega}{}^n)_R$. Combining Step 3 with Lemma~\ref{lemma:cohom}, we obtain that the morphism of group functors $\Psi\colon (L^n\Omega^n)_R\to (\ga)_R$ factors through the residue map ${\res\colon (L^n\Omega^n)_R\to (\ga)_R}$. In particular, $\Psi$ is uniquely determined by the composition
$$
\lambda\;:\;(\ga)_R\lrto (L^n\Omega^n)_R\stackrel{\Psi}\lrto (\ga)_R\,,
$$
where the first morphism sends an element $a\in A$ to the differential form $a\,\frac{dt_1}{t_1}\wedge\ldots\wedge\frac{dt_n}{t_n}$ for an $R$-algebra $A$.

The isomorphism from Proposition~\ref{prop:LBl} sends the symbol $\{1+a\,\varepsilon,t_1,\ldots,t_n\}$ to the differential form $a\,\frac{dt_1}{t_1}\wedge\ldots\wedge\frac{dt_n}{t_n}$. This implies that $\lambda=T\big(\Phi\circ\Theta_R\big)$, which finishes the proof of the theorem.
\end{proof}

\subsection{Explicit formula}    \label{CC-over-Q}

In the course of the proof of Theorem~\ref{theor-key} we have essentially obtained an explicit formula for the Contou-Carr\`ere symbol ${CC_n\colon K^M_{n+1}\big(\LL^n(A)\big)\to A^*}$ when~$A$ is a $\Q$-algebra. In order to give this explicit formula, we introduce the following auxiliary map.

\begin{prop}\label{lem-sign}
There is a unique multilinear antisymmetric (or, equivalently, symmetric) map
$$
\sgn  \; \colon \; (\z^n)^{\times (n+1)}\lrto \z/2\z
$$
such that for any $l\in \z^n$ and $r\in (\z^n)^{\times(n-1)}$, we have $\, \sgn(l,l,r)\equiv\det(l,r)\pmod{2}$.
\end{prop}
\begin{proof}
The above properties define uniquely the values of the map $\sgn$ on the elements of the standard basis in $(\z^n)^{\otimes(n+1)}$ by Dirichlet's box principle. This gives uniqueness of the map $\sgn$. Existence is implied by the following explicit formula of Vostokov and Fesenko, which we take from~\cite[Intr.]{P2}:
\begin{equation}\label{eq:Parshin}
\mbox{$\sgn(l_1,\ldots,l_{n+1})\equiv\sum\limits_{1 \leqslant i <j \leqslant n+1}  \det(l_1,\ldots,\hat{l}_i,\ldots,\hat{l}_j,\ldots,l_{n+1},l_i\cdot l_j) \pmod{2} \,,$}
\end{equation}
where $l_i\cdot l_j$ is the coordinate-wise product of the elements $l_i$ and $l_j$ in $\z^n$ and the hat symbol corresponds to omitting an element. One easily checks that this map satisfies all conditions of the proposition.
\end{proof}

\begin{rmk}
Another explicit formula for the map $\sgn$ is given by Khovanskii~\cite{Kh} (see also~\cite[Prop.\,11]{O1} for the case $n=2$):
\begin{multline*}
\mbox{$\sgn(l_1, \ldots, l_{n+1}) \equiv 1 + \sum\limits_{i=1}^{n+1} \det(l_1, \ldots, \hat{l}_i,\ldots, l_{n+1}) +  $}\\ \mbox{$ +\prod\limits_{i=1}^{n+1} \big(1 + \det(l_1, \ldots,\hat{l}_i, \ldots, l_{n+1})\big)  \pmod{2}$}
\end{multline*}
An equality between this formula and formula~\eqref{eq:Parshin} is proved in~\cite[$\S$\,1]{Kh}.
\end{rmk}

\medskip

Now we give an explicit formula for the Contou-Carr\`ere symbol over $\Q$-algebras.

\begin{theor}\label{prop-CCrational}
Let $A$ be a ring. The map
$$
\lo^n\gm(A)^{\times (n+1)}\lrto \gm(A)\,,\qquad (f_1,\ldots,f_{n+1})\longmapsto CC_n\{f_1,\ldots,f_{n+1}\}\,,
$$
is a multilinear antisymmetric map that satisfies the following conditions:
\begin{itemize}
\item[(i)]
If $A$ is a $\Q$-algebra and $f_1\in (\lo^n\gm)^{\sharp}(A)$, then
\begin{equation}  \label{exp-log-form}
CC_n\{f_1,f_2,\ldots,f_{n+1}\}=\exp\,\res\left(\log(f_1)\,\frac{df_2}{f_2}\wedge\ldots\wedge
\frac{df_{n+1}}{f_{n+1}}\right)\,,
\end{equation}
where the series $\exp$ is applied to a nilpotent element of the ring $A$.
\item[(ii)]
If $f_1\in A^*$, then (cf. formula~\eqref{eq:explicadd})
\begin{equation}  \label{form-nu}
CC_n\{f_1,f_2\ldots,f_{n+1}\}=f_1^{\,\det(\nu(f_2),\,\ldots,\,\nu(f_{n+1}))}\,.
\end{equation}
\item[(iii)]
If $f_1=t^{\underline{l}_1},\ldots, f_{n+1}=t^{\underline{l}_{n+1}}$, where $\underline{l}_i\in\uz^n(A)$, $1 \leqslant i \leqslant n+1$, then
\begin{equation}  \label{form-sign}
CC_n\{f_1,\ldots,f_{n+1}\}=(-1)^{\,\sgn(\underline{l}_1,\,\ldots,\,\underline{l}_{n+1})}\,.
\end{equation}
\end{itemize}
\end{theor}
\begin{proof}
Obviously, the map $CC_n$ is multilinear. It is antisymmetric because of the antisymmetric property for symbols in the Milnor $K$-group $K^M_{n+1}\big(\LL^n(A)\big)$ (see Remark~\ref{exam:wstable} and Lemma~\ref{lemma:Morrow}).

$(i)$ We will refer to steps of the proof of Theorem~\ref{theor-key}. Let $\varepsilon$ be a formal variable that satisfies $\varepsilon^2=0$. Step 2 together with the Proposition~\ref{prop:charactgm} imply that there is an equality ${CC_n\{f_1,\ldots,f_n\}=\exp(a)}$, where $a\in\Nil(A)$ is a nilpotent element that satisfies ${CC_n\{1+\log(f_1)\,\varepsilon,f_2,\ldots,f_n\}=1+a\,\varepsilon}$. The isomorphism from Proposition~\ref{prop:LBl} sends the symbol $\{1+\log(f_1)\,\varepsilon,f_2,\ldots,f_n\}$ to the differential form $\log(f_1)\,\frac{df_2}{f_2}\wedge\ldots\wedge\frac{f_{n+1}}{f_{n+1}}$. Therefore there is an equality
$$
CC_n\{1+\log(f_1)\,\varepsilon,f_2,\ldots,f_n\}=
1+\Psi\left(\log(f_1)\,\frac{df_2}{f_2}\wedge\ldots\wedge\frac{df_{n+1}}{f_{n+1}}\right)\,\varepsilon\,,
$$
where $\Psi\colon (L^n\Omega)_{\Q}\to (\ga)_{\Q}$ corresponds to the differential $T(CC_n)_{\Q}$ of $(CC_n)_{\Q}$ similarly as it corresponds to~$T\Phi$ in Step 2. Finally, Step 4 together with Remark~\ref{examp:eta} imply that there is an equality $$
\Psi\left(\log(f_1)\,\frac{df_2}{f_2}\wedge\ldots\wedge\frac{df_{n+1}}{f_{n+1}}\right)=
\res\left(\log(f_1)\,\frac{df_2}{f_2}\wedge\ldots\wedge\frac{df_{n+1}}{f_{n+1}}\right)\,.
$$

$(ii)$ The Contou-Carr\`ere symbol defines a morphism of group functors (see Section~\ref{sect:notation} and Lemma~\ref{lemma:diffgagm}$(ii)$)
$$
L^nK^M_n\lrto \underline{\Hom}^{gr}(\gm,\gm)\simeq\uz\,,\qquad \{f_2,\ldots,f_{n+1}\}\longmapsto \big(a\mapsto CC_n\{a,f_2,\ldots,f_{n+1}\}\big)\,,
$$
where $f_2,\ldots,f_{n+1}\in\lo^n\gm(R)$ for a ring $R$ and $a\in A^*$ for an $R$-algebra $A$. Thus we conclude by Propositions~\ref{prop-addsymb}, a discussion after its proof, and Remark~\ref{examp:eta}.

$(iii)$ The restriction of $CC_n$ to $(\uz^n)^{\times (n+1)}\subset(\lo^n\gm)^{\times (n+1)}$ is a multilinear antisymmetric morphism from $(\uz^n)^{\times (n+1)}$ to $\gm$. It follows from Dirichlet's box principle that it takes values in $\uz/2\uz\subset \gm$. Let $\underline{l}_1,\ldots\underline{l}_{n+1}\in\uz^n(A)$ for a ring $A$ be such that $\underline{l}_1=\underline{l}_2$. By Remark~\ref{exam:wstable} and Lemma~\ref{lemma:Morrow}, there is an equality in $K^M_{n+1}\big(\LL^n(A)\big)$
$$
\{t^{\underline{l}_1},t^{\underline{l}_2},\ldots,t^{\underline{l}_{n+1}}\}=
\{-1,t^{\underline{l}_2},\ldots,t^{\underline{l}_{n+1}}\}\,.
$$
Combining this with item~$(ii)$, we see that $CC_n\{t^{\underline{l}_1},t^{\underline{l}_2},\ldots,t^{\underline{l}_{n+1}}\}=
(-1)^{\,\det(\underline{l}_2,\,\ldots,\,\underline{l}_{n+1})}$. Thus we conclude by Propostion~\ref{lem-sign}.
\end{proof}

\begin{rmk}
For $n=1$, the right hand sides of formulas~\eqref{exp-log-form},~\eqref{form-nu}, and~\eqref{form-sign} from Theorem~\ref{prop-CCrational} appeared in~\cite[\S\,2.9]{Del}. (More precisely, formula from~\cite[\S\,2.9]{Del} gives an inverse map to the one given by formulas~\eqref{exp-log-form},~\eqref{form-nu}, and~\eqref{form-sign}.) Also, another, but equivalent, formula appeared in~\cite[\S\,6]{CC2} and (with a misprint) in~\cite{CC1}.
\end{rmk}

\medskip

Here are examples of the calculation of the Contou-Carr\`ere symbol, which are based on the explicit formula from Theorem~\ref{prop-CCrational} and on some descend arguments from $\Q$ to~$\Z$.

\begin{prop}
Let $A$ be a ring and $\varepsilon$ be a formal variable that satisfies $\varepsilon^{2}=0$. Define a ring $B:=A[\varepsilon]$. Then for any collection of invertible iterated Laurent series $f_1,\ldots,f_{n}\in \lo^n\gm(A)=\LL^n(A)^*$ and an iterated Laurent series $g\in \LL^n(A)$, we have an equality in $B^*$
\begin{equation}  \label{residue0}
CC_n \{1+g\,\varepsilon,f_1,\ldots,f_n\}=1+\mathop{\rm res}\left(g\,\frac{df_1}{f_1}\wedge\ldots\wedge\frac{df_n}{f_n}\right)\varepsilon\,,
\end{equation}
where the symbol $CC_n$ is applied to elements of the group $K^M_{n+1}\big(\LL^n(B)\big)$.
\end{prop}
\begin{proof}
Both sides of formula~\eqref{residue0}
define morphisms of ind-schemes from ${(L^n\gm)^{\times n}\times L^n \ga}$ to $(\ga)^{\times 2}$
given by coefficients at powers of $\varepsilon$ in the ring $B=A[\varepsilon]$. One easily checks that by Theorem~\ref{prop-CCrational}$(i)$, the values of both sides of formula~\eqref{residue0} coincide when $A$ is a $\Q$-algebra. By Example~\ref{examp:loopga}, the functor $L^n\ga$ is represented by an ind-affine space, which is an ind-flat ind-affine scheme over $\z$. Therefore by Theorem~\ref{cor:uniq} applied to $R=\Z$, $S=\Q$, and $Y=L^n\ga$, we have that the above two morphisms coincide.
\end{proof}

\begin{prop}
Let $A$ be a ring and $\eta$ be a formal variable that satisfies ${\eta^{n+2}=0}$. Define a ring $B:=A[\eta]$. Then for any collection ${g_1,\ldots,g_{n+1}\in \LL^n(A)}$ of iterated Laurent series, we have an equality in $B^*$
\begin{equation}  \label{residue}
CC_n \{1 + g_1 \eta, \ldots, 1 + g_{n+1}  \eta\} = 1 + \mathop{\rm res} (g_1 dg_2 \wedge \ldots \wedge dg_{n+1})  \eta^{n+1}\,,
\end{equation}
where the symbol $CC_n$ is applied to elements of the group $K^M_{n+1}\big(\LL^n(B)\big)$.
\end{prop}
\begin{proof}
Both sides of formula~\eqref{residue} define morphisms of ind-schemes ${(L^n \ga)^{n+1}\to(\g_a)^{\times (n+2)}}$
given by coefficients at powers of $\eta$ in the ring $B=A[\eta]$. One easily checks that by Theorem~\ref{prop-CCrational}$(i)$, the values of both sides of formula~\eqref{residue} coincide when $A$ is a $\Q$-algebra. By Proposition~\ref{lemma-repraffine}$(i)$, we have that the functor $(L^n\ga)^{n+1}$ is represented by an ind-affine space, which is an ind-flat ind-scheme over $\z$. Thus the natural homomorphism $\OO(L^n\Ab^{n+1})\to \OO\big((L^n\Ab^{n+1})_{\Q}\big)$ is injective. This proves the proposition.
\end{proof}

\medskip

\begin{rmk}\label{rmk:alternative}
Let $A$ be a $\Q$-algebra. Decomposition~\eqref{eq:decommult} in Subsection~\ref{sp-sub} implies that formulas~\eqref{exp-log-form},~\eqref{form-nu}, and~\eqref{form-sign} define uniquely a multilinear antisymmetric map
$\lo^n\gm(A)^{\times (n+1)}\to A^*$. Thus one could define a Contou-Carr\`ere symbol over \mbox{$\Q$-algebras} explicitly just by the right hand sides of formulas~\eqref{exp-log-form},~\eqref{form-nu}, and~\eqref{form-sign} and not using the boundary map for $K$-groups and geometric properties of the iterated loop group $\lo^n\gm$ and its special subgroups. When following this approach, one needs to check directly that the Contou-Carr\`ere symbol is well-defined, that is, that formulas ~\eqref{exp-log-form},~\eqref{form-nu}, and~\eqref{form-sign} define a map correctly and that this map satisfies the Steinberg relations. We do this in Propositions~\ref{prop:explCC} and~\ref{CC-Q-Milnor}, respectively, using only decompositions of the group functor~$\lo^n\gm$ given in Section~\ref{multiplicative}. In order to avoid a confusion, we denote the explicitly defined Contou-Carr\`ere symbol by~$CC^{ex}_n$.
\end{rmk}

\begin{prop}\label{prop:explCC}
For any $\Q$-algebra $A$, there is a unique multilinear antisymmetric map
$$
CC^{ex}_{n}  \; \colon  \; \lo^n\gm(A)^{\times (n+1)}\lrto A^*
$$
such that the formulas~\eqref{exp-log-form},~\eqref{form-nu}, and~\eqref{form-sign} are satisfied, where $CC_n$ is replaced by~$CC_n^{ex}$. The map $CC^{ex}_n$ is functorial with respect to $A$.
\end{prop}
\begin{proof}
As explained in Remark~\ref{rmk:alternative}, uniqueness follows from decomposition~\eqref{eq:decommult} in Subsection~\ref{sp-sub}. Thus we need to show that the map $CC^{ex}_n$ is defined correctly, which we do in the following steps.

\medskip

{\it Step 1.}
First we check that the right hand side of formula~\eqref{exp-log-form} is well-defined. By Proposition~\ref{log-map}, the map $\log$ in formula~\eqref{exp-log-form} is well-defined. Thus we need to prove that the value of the expression $\res(\cdots)$ in formula~\eqref{exp-log-form} belongs to $\Nil(A)$, whence the expression $\exp \, \res (\cdots)$ makes sense.

The expression $\res( \cdots)$ defines a multilinear map
$$
\Psi_n\;:\;(\lo^n\gm)^{\sharp}(A)  \times \lo^n\gm(A)^{\times n}   \lrto A  \,,
$$
$$
(f_1,f_2,\ldots,f_{n+1})\longmapsto \res\left(\log(f_1)\,\frac{df_2}{f_2}\wedge\ldots\wedge
\frac{df_{n+1}}{f_{n+1}}\right)\,.
$$
Using Definition~\ref{defin:specialmult} of the group $(\lo^n\gm)^{\sharp}(A)$ as a decomposition and decomposition~\eqref{eq:decomCC} of the group $\lo^n\gm(A)$ from Proposition~\ref{prop-decomp}, we see that it is enough to assume that all~$f_i$, $1 \leqslant  i \leqslant n+1$, belong to the subgroups in the above decompositions.

Suppose first that for some $i$, $1 \leqslant  i \leqslant n+1$, we have that $f_i\in (1+ \Nil)(A) \times \vv_{n,-}(A)$. Then by Remark~\ref{rmk:nilpseries}, the iterated Laurent series $\log(f_i)\in \lo^n\ga(A)$ has nilpotent coefficients and therefore by Lemma~\ref{dif-form}$(i)$, we have that $\Psi_n(f_1,\ldots,f_{n+1})\in\Nil(A)$.

Now suppose that $f_1$ belongs to the subgroup $\vv_{n,+}(A)$ and each $f_i$, $2 \leqslant i \leqslant n+1$, belongs either to the subgroup $G(A):=\gm(A)\times\vv_{n,+}(A)\times\uz^{n-1}(A)$, where $\uz^{n-1}(A)$ consists of elements $t_1^{\underline{l}_1}\cdot\ldots\cdot t_{n-1}^{\underline{l}_{n-1}}$, or to the subgroup $\uz(A)=(t_n)^{\uz(A)}$ that consists of elements~$t_n^{\underline{l}_n}$, where $\underline{l}_i\in \uz(A)$, $1\leqslant i\leqslant n$. We show by induction on $n$ that $\Psi_n(f_1,\ldots,f_{n+1})=0$. The evident base of the induction is $n=0$, where we put, by definition, $\vv_{0,+}(A)=\{1\}$.

If all $f_i$, $2 \leqslant  i \leqslant n+1$, belong to $G(A)$, then
$$
\log(f_1)\,\frac{df_2}{f_2}\wedge\ldots\wedge
\frac{df_{n+1}}{f_{n+1}}\in A((t_1))\ldots((t_{n-1}))[[t_n]]dt_1 \wedge \ldots \wedge dt_n\,.
$$
Hence, $\Psi_n(f_1,\ldots,f_{n+1})=0$.

Suppose that some $f_i$, $2\leqslant i \leqslant n+1$, belongs to $\uz(A)=(t_n)^{\uz(A)}$. Using multilinearity and the alternating property of the wedge product of differential forms, we can suppose, in addition, that $f_{n+1}=t_n$ and that all $f_i$, $2 \leqslant i \leqslant n$, belong to $G(A)$, because otherwise we have again $\Psi_n(f_1,\ldots,f_{n+1})=0$.

Under above assumptions, there is an equality
$$
\Psi_n(f_1,\ldots,f_{n},t_n)=\Psi_{n-1}(\bar{f}_1,\ldots,\bar{f}_{n})\,,
$$
where the map $\Psi_{n-1}$ is defined for collections of invertible elements of the ring $A((t_1))\ldots((t_{n-1}))$ and by $f\mapsto \bar f$ we denote the natural homomorphism from $A((t_1))  \ldots ((t_{n-1}))[[t_n]]$ to $A((t_1))  \ldots ((t_{n-1}))$. One checks directly that the homomorphism $f\mapsto \bar f$ sends $\vv_{n,+}(A)$ to $\vv_{n-1,+}(A)$, where $n\geqslant 1$ and, as we put above, $\vv_{0,+}(A)=\{1\}$. Hence by the inductive hypothesis, we have that ${\Psi_{n-1}(\bar{f}_1,\ldots,\bar{f}_{n})}=0$.

Thus we obtain a well-defined multilinear map
$$
CC^{ex}_n\;:\;(\lo^n\gm)^{\sharp}(A)\times \lo^n\gm(A)^{\times n}\lrto A^*\,,
$$
$$
(f_1,f_2,\ldots,f_{n+1})\longmapsto\exp\,\res\left(\log(f_1)\,\frac{df_2}{f_2}\wedge\ldots\wedge\frac{df_{n+1}}{f_{n+1}}\right)\,.
$$

\medskip

{\it Step 2.}
By Lemma~\ref{expl-form}$(iii)$, we see that the right hand sides of formulas~\eqref{exp-log-form} and~\eqref{form-nu} are overlapped when $f_1 \in (1 + \Nil)(A)$. We check that in this case the values of these two formulas coincide. Indeed, by Proposition~\ref{prop-integer-res} we have
 \begin{multline} \nonumber
 \exp\,\res\left(\log(f_1)\,\frac{df_2}{f_2}\wedge\ldots\wedge
\frac{df_{n+1}}{f_{n+1}}\right) = \exp\, \left(\log(f_1) \res\left(\frac{df_2}{f_2}\wedge\ldots\wedge
\frac{df_{n+1}}{f_{n+1}}\right)\right) =  \\ =
\exp\, \left(\log(f_1) \cdot \det\big( \nu(f_2), \ldots, \nu(f_{n+1})\big)          \right)
= f_1^{\,\det(\nu(f_2),\, \ldots,\,\nu(f_{n+1}))} \, .
 \end{multline}
Notice that, actually, here we use Proposition~\ref{prop-integer-res} only for $\Q$-algebras and that in this case Proposition~\ref{prop-integer-res} is proved explicitly without the theory of thick ind-cones (see the proof of Proposition~\ref{prop-integer-res}).

By decomposition~\eqref{eq:decommult} in Subsection~\ref{sp-sub}, we obtain that formulas~\eqref{exp-log-form} and~\eqref{form-nu} define correctly a multilinear map
\begin{equation}\label{eq:ccexpart}
CC^{ex}_n\;:\;(\lo^n\gm)^{0}(A)\times \lo^n\gm(A)^{\times n}\lrto A^*\,.
\end{equation}

\medskip

{\it Step 3.}
Now we check the antisymmetric property of the map in formula~\eqref{eq:ccexpart}.

Formulas~\eqref{exp-log-form} and~\eqref{form-nu} are antisymmetric with respect to $f_2,\ldots,f_{n+1}$, because the wedge product of differential forms and the determinant are alternating.

Further, formula~\eqref{exp-log-form} is antisymmetric with respect to $f_1$ and~$f_2$. Namely, if ${f_1,f_2 \in (\lo^n\gm)^{\sharp}(A)}$, then by Lemmas~\ref{dif-form}$(i)$ and~\ref{lemma:cohom}, there are equalities
\begin{multline}  \nonumber
0 = \res \left(  d\left( \log(f_1)  \log(f_2) \frac{df_3}{f_3}  \wedge \ldots \wedge \frac{df_{n+1}}{f_{n+1}}  \right)   \right)=   \\
= \res  \left(  \log(f_1)  \frac{df_2}{f_2}  \wedge \frac{df_3}{f_3}  \wedge \ldots \wedge   \frac{df_{n+1}}{f_{n+1}}        \right)
+ \res  \left(  \log(f_2)  \frac{df_1}{f_1}  \wedge \frac{df_3}{f_3}  \wedge \ldots \wedge   \frac{df_{n+1}}{f_{n+1}}        \right)   \, .
\end{multline}

Finally, formula~\eqref{form-nu} is antisymmetric with respect to $f_1$ and~$f_2$ as well. Indeed, if $f_1,f_2\in A^*$, then there is an equality
$$
f_1^{\,\det(\nu(f_2),\,\nu(f_3),\,\ldots,\,\nu(f_{n+1}))}=f_2^{\,\det(\nu(f_1),\,\nu(f_3),\,\ldots,\,\nu(f_{n+1}))}=1\,,
$$
because ${\nu(f_2)=\nu(f_1)=0}$.

\medskip

{\it Step 4.}
Formula~\eqref{form-sign} is antisymmetric by Proposition~\ref{lem-sign}. Combining this with Step 3 and decomposition~\eqref{eq:decommult} in Subsection~\ref{sp-sub}, we obtain a well-defined multilinear antisymmetric
map $CC_{n}^{ex}\colon\lo^n\gm(A)^{\times (n+1)}\to A^*$. Clearly, this map is functorial with respect to a $\Q$-algebra $A$, because decomposition~\eqref{eq:decomCC} is functorial.
\end{proof}

\begin{rmk}
For $n=2$, Proposition~\ref{prop:explCC} was proved in~\cite[\S\,3.3]{OZ1}.
\end{rmk}

\medskip

Now we show that the map $CC^{ex}_n$ in Proposition~\ref{prop:explCC} satisfies the Steinberg relation. First we prove an auxiliary partial result.

\begin{lemma} \label{St-Lemma}
Let $A$ be a $\Q$-algebra. For any collection of $n$ elements $f_1, f_3, \ldots, f_{n+1}$ of the group~$\lo^n\gm(A)$, there is an equality
\begin{equation}  \label{for-lem}
CC^{ex}_n(f_1, -f_1, f_3, \ldots, f_{n+1})=1\,.
\end{equation}
\end{lemma}
\begin{proof}
Clearly, the left hand side of formula~\eqref{for-lem} is multilinear in $f_3, \ldots, f_{n+1}$. Antisymmetricity of the map $CC_n^{ex}$ implies that the left hand side is also linear in $f_1$. Indeed, for any collection $f_1,f_1',f_3,\ldots,f_{n+1}$ of elements in $\lo^n\gm(A)$, we have the equalities
\begin{multline*}
CC_n^{ex}(f_1f'_1,-f_1f'_1,f_3,\ldots,f_{n+1})=CC_n^{ex}(f_1,-f_1,f_3,\ldots,f_{n+1})\cdot \\
\cdot CC_n^{ex}(f_1,f'_1,f_3,\ldots,f_{n+1})\cdot CC_n^{ex}(f'_1,f_1,f_3,\ldots,f_{n+1})\cdot CC_n^{ex}(f'_1,-f'_1,f_3,\ldots,f_{n+1})\,.
\end{multline*}
Thus we can use decomposition~\eqref{eq:decommult} in Subsection~\ref{sp-sub}.

Suppose that $f_i\in(\lo^n\gm)^0(A)$ for some $i=1,3,\ldots,n+1$. If ${f_i\in\gm(A)}$, then the antisymmetric property and formula~\eqref{form-nu} imply that ${CC^{ex}_n(f_1,-f_1,f_3,\ldots,f_{n+1})=1}$. Now suppose that $f_i\in (\lo^n\gm)^{\sharp}(A)$. If $3\leqslant i\leqslant n+1$, then the antisymmetric property and formula~\eqref{exp-log-form} imply that $CC^{ex}_n(f_1,-f_1,f_3,\ldots,f_{n+1})=1$.
If $i=1$, that is, ${f_1 \in (\lo^n\gm)^{\sharp}(A)}$, then $f_1 = 1-h$, where $h \in \LL^n(A)^\sharp$. Then by formula~\eqref{exp-log-form}, we have
\begin{multline}  \nonumber
CC^{ex}_n(f_1,-f_1, f_3, \ldots, f_{n+1})  =
\exp\,\res\left(\log(1-h)\,\frac{dh}{-1+h}\wedge \frac{df_3}{f_3}\wedge\ldots\wedge
\frac{df_{n+1}}{f_{n+1}}\right)=\\
= \exp\,\res\left(d\left( \varphi(h) \frac{df_3}{f_3}\wedge \ldots\wedge
\frac{df_{n+1}}{f_{n+1}}\right)\right)= 1  \, ,
\end{multline}
where a power series $\varphi \in A[[x]]$ exists by Lemma~\ref{dif-form}$(iii)$ and the last equality follows from Lemma~\ref{lemma:cohom}.

It remains to consider the case when all $f_i$ belong to $\uz(A)^n$. By multilinearity and the antisymmetric property, there is an equality
$$
CC_n^{ex}(f_1,-f_1,\ldots,f_{n+1})=CC_n^{ex}(-1,f_1,\ldots,f_{n+1})^{-1}\cdot CC_n^{ex}(f_1,f_1,\ldots,f_{n+1})\,.
$$
Thus we conclude by formulas~\eqref{form-nu},~\eqref{form-sign} and Proposition~\ref{lem-sign}.
\end{proof}

The next proposition and the antisymmetric property of the map $CC^{ex}_n$ imply that~$CC^{ex}_n$ factors through the Milnor $K$-group, that is, we have a well-defined homomorphism of groups
$$
{CC_{n}^{ex}\;:\; K^M_{n+1}\big(A((t_1))\ldots((t_n))\big)\lrto A^*}\,,
$$
which is functorial with respect to a $\Q$-algebra $A$.

\begin{prop}  \label{CC-Q-Milnor}
Let $A$ be a $\Q$-algebra and let $f_1,\ldots,f_n\in \lo^n\gm(A)$ be a collection of invertible iterated Laurent series such that $f_1+f_2=1$. Then there is an equality ${CC^{ex}_n(f_1,\ldots,f_{n+1})=1}$.
\end{prop}
\begin{proof}
The proof is similar to the proof of the case $n=2$ in~\cite[Prop.\,4.2]{OZ1}. For the sake of completeness we briefly repeat it here for arbitrary $n$ (and in another notations than in~\cite{OZ1}).

Changing the $\Q$-algebra $A$ if needed, without loss of generality, we may assume that~$\nu(f_1)$ and $\nu(1-f_1)$ belong to the subgroup $\Z^n \subset \uz(A)^n$. Consider the following cases.

First suppose that $\nu(f_1) > 0$. Then by Proposition~\ref{prop-decomp} and Example~\ref{examp:projections}$(i)$, we have $1-f_1 \in (\lo^n\gm)^{\sharp}(A)$. By Lemmas~\ref{dif-form}$(ii)$ and~\ref{lemma:cohom}, there are equalities
\begin{multline} \nonumber
CC_n^{ex}(f_1, 1-f_1, f_3, \ldots, f_{n+1})= CC_n^{ex}(1-f_1, f_1, f_3, \ldots, f_{n+1})^{-1}= \\ =
\exp\,\res\left(-\log(1-f_1)\,\frac{df_1}{f_1}\wedge \frac{df_3}{f_3}\wedge \ldots\wedge
\frac{df_{n+1}}{f_{n+1}}\right) = \\ = \exp\,\res\left( d \left({\rm Li_2}(f_1) \frac{df_3}{f_3}\wedge \ldots\wedge
\frac{df_{n+1}}{f_{n+1}}\right) \right)=1 \, .
\end{multline}

Now suppose that $\nu(f_1)=0$. Then by Proposition~\ref{prop-decomp} and Example~\ref{examp:projections}$(i)$, we have $\nu(1-f_1) \geqslant 0$. If $\nu(1-f_1) > 0$, then transposing $f_1$ and $1-f_1$ and using the antisymmetric property of the map $CC_n^{ex}$, we reduce to the previous case. Therefore we can suppose that $\nu(1-f_1) = 0$. By decomposition~\eqref{eq:decommult} in Subsection~\ref{sp-sub}, we have that $f_1 = a(1-h)$, where $a \in A^*$, $1-a \in A^*$, and $h \in \LL^n(A)^\sharp$. By formula~\eqref{form-nu}, there are equalities
$$
CC^{ex}_n(a,1-f_1,f_3,\ldots,f_{n+1})= a^{\,\det(\nu(1-f_1),\,\nu(f_3),\,\ldots,\,\nu(f_{n+1}))}=1 \, .
$$
By formula~\eqref{exp-log-form}, there is an equality
$$
CC^{ex}_n(1-h,1-f_1,f_3,\ldots,f_{n+1})=\exp\,\res\left(
\log(1-h)\,\frac{dh}{a^{-1}-1+h}\wedge \frac{df_3}{f_3}\wedge\ldots\wedge
\frac{df_{n+1}}{f_{n+1}}\right)\,,
$$
because $1-f_1=1-a+ah$. By Lemmas~\ref{dif-form}$(iii)$ and~\ref{lemma:cohom}, the last expression equals one.

Finally, suppose that $\nu(f_1) < 0$. By multilinearity, there is an equality
\begin{multline*}
CC^{ex}_n(f_1, 1-f_1, f_3, \ldots, f_{n+1})= \\
=CC^{ex}_n(f_1,-f_1, f_3, \ldots, f_{n+1}) \cdot CC^{ex}_n(f_1^{-1}, 1 - f_1^{-1}, f_3, \ldots, f_{n+1})^{-1}\,.
\end{multline*}
Since $\nu(f_1^{-1})>0$, the last expression equals one by Lemma~\ref{St-Lemma} and the first case considered above.
\end{proof}

\subsection{Completed Contou-Carr\`ere symbol}\label{subsect:vert}

Let now $A=\mbox{``$\varprojlim\limits_{i\in I}$''}A_i$ be a pro-ring, that is, a pro-object in the category of rings. Recall that a~pro-object is the dual notion to an ind-object and the category of pro-rings is anti-equivalent to the category of ind-affine schemes (see Subsection~\ref{subsect:indschemes}). By $\widehat{A}:=\varprojlim\limits_{i\in I}A_i$ denote be the corresponding inverse limit of rings. One easily checks that there is a natural isomorphism of groups $\widehat{A}^*\stackrel{\sim}\lrto \varprojlim\limits_{i\in I}A^*_i$.

We have a pro-ring $\mbox{``$\varprojlim\limits_{i\in I}$''}\LL^n(A_i)$. Let $\widehat{\LL}^n(A):=\varprojlim\limits_{i\in I}\LL^n(A_i)$ be the corresponding inverse limit of rings.

\medskip

\begin{defin}\label{defin:compCC}
Define a {\it completed Contou-Carr\`ere symbol} $\widehat{CC}_n$ as the following composition of homomorphisms of groups:
$$
\widehat{CC}_n\;:\;K^M_{n+1}\big(\widehat{\LL}^n(A)\big)\lrto\varprojlim_{i\in I} K^M_{n+1}\big(\LL^n(A_i))\lrto
\varprojlim\limits_{i\in I} A_i^*\stackrel{\sim}\lrto \widehat{A}^*\,,
$$
where the second homomorphism is the inverse limit of the Contou-Carr\`ere symbols $CC_{n}\colon K^M_{n+1}\big(\LL^n(A_i)\big)\to A_i^*$ taken over the rings $A_i$, $i\in I$.
\end{defin}

Clearly, the homomorphism $\widehat{CC}_n$ is functorial with respect to the pro-ring~$A$.

\begin{examp}\label{examp:CCcompl}
\hspace{0cm}
\begin{itemize}
\item[(i)]
Let $R$ be a ring and let $A=\mbox{``$\varprojlim\limits_{d\in \N}$''}R[x]/(x^d)$. Then $\widehat{A}\simeq R[[x]]$. For each $d\in \N$, we have an isomorphism $\LL^n\big(R[x]/(x)^d\big)\simeq \LL^n(R)[x]/(x)^d$ (cf. Lemma~\ref{lemma:top}$(iii)$). Hence we have an isomorphism $\widehat{\LL}^n(A)\simeq \LL^n(R)[[x]]=R((t_1))\ldots((t_n))[[x]]$ and we obtain a symbol $\widehat{CC}_n\colon K^M_{n+1}\big(R((t_1))\ldots((t_n))[[x]]\big)\to R[[x]]^*$.
\item[(ii)]
More generally, let $M$ be a set and let $A=\mbox{``$\varprojlim\limits_{(M',d)}$''}R[M']/(M')^d$, where $M'$ runs over all finite subsets in $M$ and $d\in\N$ (see Subsection~\ref{subsect:convalg}). Then $\widehat{A}\simeq R[[M]]$ (see Definition~\ref{defin:powerseries}), $\widehat{\LL}^n(A)\simeq R((t_1))\ldots((t_n))[[M]]$, and we obtain a symbol $\widehat{CC}_n\colon K^M_{n+1}\big(R((t_1))\ldots((t_n))[[M]]\big)\to R[[M]]^*$.
\item[(iii)]
Let $A=\mbox{``$\varprojlim\limits_{d\in \N}$''}\Z/(p^d)$, where $p$ is a prime number. Then $\widehat{A}\simeq\Z_p$ is the ring of $p$-adic integers and the ring $\widehat{\LL}^n(A)\simeq \Z_p\{\{t_1\}\}\ldots\{\{t_n\}\}$ consists of infinite series $\sum\limits_{l\in\Z^n}a_lt^l$, where $a_l\in \Z_p$ and for any $d\in \N$, there is $\lambda\in\Lambda_n$ (see Subsection~\ref{subsect:Laurent}) such that if ${l\notin \z^n_{\lambda}}$, then $a_l\in p^d\,\Z_p$. We obtain a symbol ${\widehat{CC}_n\colon K^M_{n+1}\big(\Z_p\{\{t_1\}\}\ldots\{\{t_n\}\}\big)\to\z_p^*}$.
\end{itemize}
\end{examp}

\medskip

The following fact is directly implied by Theorem~\ref{prop-CCrational}$(i)$.

\begin{prop}\label{examp:VS}
Suppose that all $A_i$, $i\in I$, are $\Q$-algebras and consider a collection ${f_1,\ldots,f_{n+1}\in \widehat{\LL}^n(A)^*}$. Suppose that
$$
f_1=\varprojlim_{i\in I}f_{1,i}\in \varprojlim\limits_{i\in I}(\lo^n\gm)^{\sharp}(A_i)\subset \widehat{\LL}^n(A)^*\,.
$$
Then there is an equality in $\widehat{A}^*$
$$
\widehat{CC}_n\{f_1,\ldots,f_{n+1}\}=\exp\,\res\left(\log(f_1)\,\frac{df_2}{f_2}\wedge\ldots\wedge
\frac{df_{n+1}}{f_{n+1}}\right)\,,
$$
where $\log(f_1)=\varprojlim\limits_{i\in I}\log(f_{1,i})$ is a well-defined element of the ring ${\widehat{\LL}^n(A)=\varprojlim\limits_{i\in I}\LL^n(A_i)}$, the residue map is a homomorphism
$$
\res\;:\;\widehat{\LL}^n(A)dt_1\wedge\ldots\wedge dt_n\lrto \widehat{A}\,,
$$
the series $\exp$ is applied to an element of $\varprojlim\limits_{i\in I}\Nil(A_i)\subset \widehat{A}$, and each differential form $df_i$, $1\leqslant i \leqslant n$, belongs to the inverse limit $\varprojlim\limits_{i\in I}\widetilde{\Omega}^1_{\LL^n(A_i)}=\bigoplus\limits_{j=1}^n\widehat{\LL}^n(A)dt_j$.
\end{prop}

For instance, for a pro-ring $A=\mbox{``$\varprojlim\limits_{d\in \N}$''}R[x]/(x^d)$ as in Example~\ref{examp:CCcompl}(i), the residue $\res\colon R((t_1))\ldots((t_n))[[x]]dt_1\wedge\ldots\wedge dt_n\to R[[x]]$ in Proposition~\ref{examp:VS} is taken with respect to the variables~$t_1,\ldots,t_n$. Also, we have that $\varprojlim\limits_{i\in I}\Nil(A_i)=xR[[x]]$ in this case.

\medskip

\begin{rmk}
It is natural to expect the existence of an extension of the completed Contou-Carr\`ere symbol in Example~\ref{examp:CCcompl}(i) to a homomorphism from $K^M_{n+1}\big(R((t_1))\ldots((t_n))((x))\big)$ to~$R((x))^*$. Such a new symbol would be a morphism of group functors $L^{n+1}K^M_{n+1}\to L\gm$. Note that it would definitely differ from the morphism $L(CC_n)\colon L^{n+1}K^M_{n+1}\to L\gm$ as the latter morphism vanishes on the group $K^M_{n+1}\big(R((t_1))\ldots((t_n))[[x]]\big)$.
\end{rmk}

\begin{rmk}\label{rmk:Denisvertical}
A particular case of a completed Contou-Carr\`ere symbol $K^M_2\big(R((t))((x))\big)\to R((x))^*$ when $R$ is a field of characteristic zero was defined by the second named author in~\cite[Theor.\,2]{O97}. There it was considered a surface fibred over a curve and this symbol was constructed as a local direct image map related to a flag on the surface given by a point and the fiber passing through it. Another definition of a completed symbol in a particular case had been given by Kato~\cite{Kato}. A completed Contou-Carr\`ere symbol and its relation to Kato symbol were studied by Asakura~\cite[\S\,4]{Asa}, P\'al~\cite[\S\S\,3,4]{Pal}, Liu~\cite[\S\,3]{Liu}, and Chinburg, Pappas, Taylor~\cite[\S\S\,3e, 3f]{CPT}.
\end{rmk}

\medskip

The following lemma is implied immediately by the construction of the completed Contou-Carr\`ere symbol.

\begin{lemma}\label{lemma-trivkernel}
The completed Contou-Carr\`ere symbol $\widehat{CC}_n\colon K^M_{n+1}\big(\widehat{\LL}^n(A)\big)\to \widehat{A}^*$
sends the kernel $\Ker\Big(K^M_{n+1}\big(\widehat{\LL}^n(A)\big)\to K^M_{n+1}\big(\LL^n(A_i)\big)\Big)$ to the kernel $\Ker\big(\widehat{A}^*\to A_i^*\big)$ for each~$i\in I$.
\end{lemma}

\subsection{Integrality and convergence of the explicit formula}\label{integr-expl-form}

The aim of this subsection is to show that the right hand side of formula~\eqref{exp-log-form} in Theorem~\ref{prop-CCrational}$(i)$ is expressed by some universal power series with integral coefficients. The variables in these universal power series are coefficients of iterated Laurent series to which one applies the Contou-Carr\`ere symbol.

\medskip

Let $q$, $0\leqslant q\leqslant n$, be an integer and let $1\leqslant j_1<\ldots<j_q\leqslant n$ be a collection of~$q$ integers (for $q=0$, the collection is empty). Put $p:=n+1-q$. Thus, we have ${1\leqslant p\leqslant n+1}$. Consider countably many formal variables $x_{i,l}$, where $1\leqslant i\leqslant p$ and $l\in\z^n$. Given a ring $R$, for short, we denote just by~$R[[x_{i,l}]]$ the ring of power series ${R[[x_{i,l};\,1\leqslant i\leqslant p,\,l\in\z^n]]}$ in these formal variables (see Definition~\ref{defin:powerseries}). Consider infinite series
$$
\mbox{$f_i:=1+\sum\limits_{l\in\z^n} x_{i,l} t^{l}\,,\qquad 1\leqslant i\leqslant p$}\,,
$$
which are well-defined elements in the ring $\Z((t_1))\ldots((t_n))[[x_{i,l}]]$ (see Example~\ref{examp:CCcompl}(ii)). Moreover, for each $i$, $1\leqslant i\leqslant p$, we have that
$$
f_i\in \varprojlim\limits_{(M',d)}(\lo^n\gm)^{\sharp}\big(\Z[M']/(M')^d\big)\,,
$$
where~$M'$ runs over all finite subsets in $\{1,\ldots,p\}\times\z^n$ and $d\in\N$. In particular, ${f_i\in \Z((t_1))\ldots((t_n))[[x_{i,l}]]^*}$. Define the following power series (see Definition~\ref{defin:compCC}):
$$
\varphi_{n,j_1,\ldots,j_q}:=\widehat{CC}_n\{f_1,\ldots,f_{p},t_{j_1},\ldots,t_{j_q}\}\in \Z[[x_{i,l}]]^*\,.
$$
Lemma~\ref{lemma-trivkernel} implies that the constant term of the power series $\varphi_{n,j_1,\ldots,j_q}$ is equal to one.

\medskip

Consider the embeddings of rings
$$
\Z[[x_{i,l}]]\subset \Q[[x_{i,l}]]\,,\qquad \z((t_1))\ldots((t_n))[[x_{i,l}]]\subset \Q((t_1))\ldots((t_n))[[x_{i,l}]]\,.
$$
It turns out that the series $\varphi_{n,j_1,\ldots,j_q}$ viewed as an element of $\Q[[x_{i,l}]]^*$ can be constructed explicitly as follows. By funtoriality of the completed Contou-Carr\`ere symbol $\widehat{CC}_n$, Proposition~\ref{examp:VS} implies that there is an equality in $\Q[[x_{i,l}]]^*$
\begin{equation}\label{eq:phiformal}
\varphi_{n,j_1,\ldots,j_q}=\exp\,\res\left(\log(f_1)\,\frac{df_2}{f_2}\wedge\ldots\wedge
\frac{df_{p}}{f_{p}}\wedge\frac{dt_{j_1}}{t_{j_1}}\wedge\ldots\wedge\frac{dt_{j_q}}{t_{j_q}}\right)\,.
\end{equation}
Explicitly, the series $\varphi_{n,j_1,\ldots,j_q}$ is obtained by formally opening brackets in the right hand side of formula~\eqref{eq:phiformal}. More precisely, for every finite subset $M'\subset\{1,\ldots,p\}\times \Z^n$ and $d\in\N$, by $f\mapsto \bar f$ denote the natural homomorphism of rings
$$
\Q((t_1))\ldots((t_n))[[x_{i,l}]]\lrto \Q((t_1))\ldots((t_n))[M']/(M')^d\,.
$$
Consider the right hand side of formula~\eqref{eq:phiformal} over the ring $\Q((t_1))\ldots((t_n))[M']/(M')^d$. The series $\log(\bar f_1)=\sum\limits_{m\geqslant 1}(-1)^{m+1}\frac{(\bar{f}_1-1)^m}{m}$ is a finite sum in the ring $\Q((t_1))\ldots((t_n))[M']/(M')^d$, the expression $\res(\cdots)$ is an element of the ideal $(M')\cdot\Q((t_1))\ldots((t_n))[M']/(M')^d$, and the right hand side of formula~\eqref{eq:phiformal} applied over the ring $\Q((t_1))\ldots((t_n))[M']/(M')^d$ defines correctly an element
$$
\varphi_{n,j_1,\ldots,j_q}^{(M',d)}\in\Q[M']/(M')^d\,.
$$
Passing to the limit over all~$(M',d)$, we obtain the series $\varphi_{n,j_1,\ldots,j_q}=\varprojlim\limits_{(M',d)}\varphi_{n,j_1,\ldots,j_q}^{(M',d)}$.

\medskip

By the first construction of the series~$\varphi_{n,j_1,\ldots,j_q}$, we see that formally opening brackets in the right hand side of formula~\eqref{eq:phiformal}, one gets a series with {\it integral} coefficients (cf.~\cite[\S\,2.9]{Del}). Thus we deduce this integrality from the existence of the Contou-Carr\`ere symbol for all rings, not just $\Q$-algebras, that is, from the existence of the boundary map for algebraic $K$-groups (see Subsection~\ref{subsect:bound}).

\begin{examp}
Let $f=1+\sum\limits_{l\in \z^n}x_lt^l\in \Z((t_1))\ldots((t_n))[[x_l;\,l \in \Z^n]]$. Then the series $\varphi_{n,1,\ldots,n}\in\Z[[x_l;\,l\in\z^n]]$ is obtained by formally opening brackets in the expression
$$
\exp\,\res\left(\log(f)\,\frac{dt_1}{t_1}\wedge\ldots\wedge
\frac{dt_{n}}{t_{n}}\right)=\exp\,\left(\sum\limits_{m\geqslant 1}(-1)^{m+1}\frac{[(f-1)^m]_0}{m}\right)\,,
$$
where $[g]_0$ denotes the constant term of an iterated Laurent series $g$. The integrality of this series was also proved by Kontsevich~\cite[p.\,2, Step\,1]{Kon} by a combinatorial method.
\end{examp}

\medskip

The following lemma is implied directly by the construction of the completed Contou-Carr\`ere symbol (see also Example~\ref{examp:convalg}$(i)$).

\begin{lemma}\label{lemma:nilpolyn}
Let $A$ be a ring, $g_1,\ldots,g_p$ be Laurent polynomials (not just series) with nilpotent coefficients in $A$. Then there is an equality in $A^*$
$$
CC_n\{1+g_1,\ldots,1+g_p,t_{j_1},\ldots,t_{j_q}\}=\varphi_{n,j_1\ldots j_q}(g_1,\ldots,g_p)\,.
$$
\end{lemma}

Combining Lemma~\ref{lemma:nilpolyn} with Theorem~\ref{prop-key}, we obtain the following result (see also\ Definition~\ref{defin-conv}(iii)).

\begin{theor}\label{theor:integalexpl}
The integral power series~$\varphi_{n,j_1,\ldots,j_q}$ converges algebraically on the ind-closed subscheme $((\lo^n\ga)^{\sharp})^{\times p}\subset (\Ab^{\Z^n})^{\times p}$ and for any ring $A$ and a collection ${g_1,\ldots,g_p\in (\lo^n\ga)^{\sharp}(A)}$, there is an equality in $A^*$
$$
{CC_n\{1+g_1,\ldots,1+g_p,t_{j_1},\ldots,t_{j_q}\}=\varphi_{n,j_1\ldots j_s}(g_1,\ldots,g_p)}\,.
$$
\end{theor}


\medskip

Let $x_{i,l}$ be of weight $l\in\z^n$. A weight of a monomial in~$x_{i,l}$ is defined in a natural way. For example, the weight of $x_{i,l}^d$ is $dl$ and the weight of $x_{i,l}\,x_{i',l'}$ is $l+l'$. A weight homogeneous power series in~$\Q[[x_{i,l}]]$ is an infinite sums of monomials of the same weight, which is also called a weight of this series. For example, the polynomial $x_{i,l}^2+x_{i',2l}$ is weight homogenous of weight $2l$. Let us say that a series $\sum\limits_{l\in\z^n}\varphi_l\,t^l\in \Q((t_1))\ldots((t_n))[[x_{i,l}]]$ is weighted if $\varphi_l$ is a weight homogeneous power series of weight $l$ for each $l\in\z^n$. Similarly, for elements in $\Q((t_1))\ldots((t_n))[M']/(M')^d$, where~$M'$ and~$d$ are as above. The following result is proved for the case $n=1$ in~\cite[\S\,2.9]{Del}.

\begin{prop}\label{lemma:weight0}
The power series $\varphi_{n,j_1,\ldots,j_q}$ is weight homogeneous of weight zero.
\end{prop}
\begin{proof}
One easily checks that sums and products of weighted series in $\Q((t_1))\ldots((t_n))[[x_{i,l}]]$ are weighted. Further, if $f=\sum\limits_{l\in\z^n}\varphi_l\,t^l\in \Q((t_1))\ldots((t_n))[[x_{i,l}]]$ is weighted, then $df=\sum\limits_{i=1}^ng_i\,\frac{dt_i}{t_i}$, where $g_i$, $1\leqslant i\leqslant n$, are weighted as well. Indeed, there is an equality $g_i=\sum\limits_{l\in \z^n}l_i\varphi_l\,t^l$ for each $1\leqslant i\leqslant n$. Finally, if $f\in \Q((t_1))\ldots((t_n))[[x_{i,l}]]$ is weighted, then $\res\big(f\,\frac{dt_1}{t_1}\wedge\ldots\wedge\frac{dt_n}{t_n}\big)$ is a weight homogeneous power series in $\Q[[x_{i,l}]]$ of weight zero, being the constant term of $f$.

All together this implies that for each finite subset $M'\subset \{1,\ldots,p\}\times\z^n$ and $d\in\N$, the element ${\varphi^{(M',d)}_{n,j_1,\ldots,j_q}\in\Q[M']/(M')^d}$ is weight homogeneous of weight zero. Now we pass to the limit over all~$(M',d)$.
\end{proof}

\quash{
\begin{lemma}
Given an element $\lambda\in\Lambda_n$ (see Subsection~\ref{subsect:Laurent}) and an natural number $N\in\N$, there is an explicitly defined constant $C'=C'(\lambda,N)$ such that if $l_1+\ldots+l_N=0$ and $l_i\in\Z^n_{\lambda}$, then $|l_i|\geqslant C'$, where $|l_i|:=|l_1|+\ldots|l_n|$.
\end{lemma}

\begin{theor}
For any element $\lambda\in\Lambda_n$ (see Subsection~\ref{subsect:Laurent}) and an natural number $N\in\N$, there is an explicitly defined constant $C=C(\lambda,N)$ that satisfies the following condition. Let $(j_1,\ldots,j_q)$ be a collection as above, $A$ be a ring, and $g_1,\ldots,g_p\in(\lo^n\gm)^{\sharp}(A)$ be such that $g_i=\sum\limits_{l\in\z^n_{\lambda}}a_{i,l}t^l$ and $\big(\sum\limits_{l\in\z^n_{\lambda},l\leqslant 0})a_{i,l}t^l\big)^N=0$ (see Definition~\ref{defin:sharpaddfunctor}). Then $\varphi_{n,j_1,\ldots,j_q}(g_1,\ldots,g_p)$ depends on coefficients in the series $\varphi_{n,j_1,\ldots,j_q}$ only by monomial from a finite set of all monomials of degree at most $p(N-1)$ in variables $x_l$ such that $|l|\leqslant C$.
\end{theor}
\begin{proof}
It is clear the all monomials in $\varphi^{(M',d)}$
\end{proof}
}

\section{Relation to higher local class field theory}\label{hi-cl-f-th}

In this section, we explain how the $n$-dimensional Contou-Carr\`ere symbol $CC_n$ leads to the local reciprocity map in the explicit higher local class field theory for $n$-dimensional local fields of positive characteristic as constructed by Parshin in~\cite{P1}. (For another approach to higher local class field theory see Kato's papers~\cite{K0},~\cite{K1}.)

Proposition~\ref{prop-Witt} below is also of independent interest. It describes a bilinear pairing of group functors $L^n K_n^M \times L^n W_S  \to W_S $, where $W_S$ is the group functor of Witt vectors that depends on a set of positive integers $S$ closed under taking divisors.

\subsection{Unramified, Kummer, and Artin--Schreier--Witt extensions}\label{subsect:localcft}

Let $\F_q$ be a finite field of characteristic $p$ and let $\overline{\F}_q$ be its algebraic closure. Higher local class field theory describes the Galois group of the maximal Abelian extension $K^{ab}$ of the $n$-dimensional local field $K:=\F_q((t_1))\ldots((t_n))$. One has a local reciprocity map
$$
K_n^M(K) \lrto  \Gal(K^{ab}/K)\,.
$$
The image of the local reciprocity map is dense with respect to the profinite topology on the right hand side, see~\cite[\S\,4, Theor.\,1(1)]{P1}. The kernel is described in topological terms in loc.cit. and is described in algebraic terms by Fesenko in~\cite[Intr.]{Fes} as follows: the kernel is the subgroup of all divisible elements in $K_n^M(K)$.

\medskip

Let $K^{nr}$ be the maximal Abelian unramified extension of $K$. Explicitly, we have that $K^{nr}=\overline{\F}_q((t_1))\ldots((t_n))$ and there is an isomorphism of profinite groups
$$
\Gal(K^{nr}/K)\stackrel{\sim}\lrto\widehat{\Z}\,.
$$
Let $K^{Km}$ be the maximal Kummer extension of $K$, that is, the maximal Abelian extension of exponent $q-1$. Explicitly, we have that ${K^{Km}=\F_{q^{q-1}}((\sqrt[q-1]{t_1}))\ldots ((\sqrt[q-1]{t_n}))}$ and this extension is of finite degree over $K$. By Kummer theory, there is an isomorphism of finite groups
$$
\Gal(K^{Km}/K)\stackrel{\sim}\lrto\Hom\big(K^*/(K^*)^{q-1},\F_q^*\big)\,.
$$
Finally, let $K^{ASW}$ be the maximal Artin--Schreier--Witt extension of $K$, that is, the maximal Abelian $p$-extension. The degree of $K^{ASW}$ over $K$ is infinite. By Artin--Schreier--Witt theory, for each $r\geqslant 1$, there is an isomorphism of pro-$p$ groups
$$
\Gal(K^{ASW}/K)/p^r\Gal(K^{ASW}/K)\stackrel{\sim}\lrto\Hom\big(W_{p^r}(K)/(Fr-1)W_{p^r}(K),\Z/p^r\Z\big)\,,
$$
where $W_{p^r}$ denotes the group of $r$-truncated $p$-Witt vectors (see Subsection~\ref{ASWp} below) and $Fr$ denotes the Frobenius homomorphism for Witt vectors.

The maximal Abelian extension $K^{ab}$ is the composite $K^{nr}\cdot K^{Km}\cdot K^{ASW}$. In order to construct the local reciprocity map, one defines three homomorphisms from $K^M_n(K)$ to the above Galois groups and then one checks a compatibility between these homomorphisms. These three homomorphisms are defined in the proof of~\cite[\S\,4, Theor.\,1]{P1}. Below we explain how they are constructed with the help of the Contou-Carr\`ere symbol~$CC_n$.

\medskip

The homomorphism $K^M_n(K)\to\Gal(K^{nr}/K)$ is the composition of a homomorphism~$c_K\colon K^M_n(K)\to\Z$ defined in ~\cite[\S\,3, Def.\,2]{P1} and the natural embedding $\z\to \widehat{\z}$. It follows from~\cite[\S\,3, Prop.\,2]{P1} that $c_K=\nu_n$ (see Definition~\ref{def:addsymb}).

The homomorphism $K^M_n(K)\to \Gal(K^{Km}/K)$ is induced by a bilinear pairing $K_n^M(K)\times K^*\to \F_q^*$, which is defined as the composition of the product between Milnor $K$-groups and a homomorphism
$$
K_{n+1}^M(K)\lrto \F_q^*\,,\qquad \{f_1,\ldots,f_{n+1}\}\longmapsto (f_1,\ldots,f_{n+1})_{K/\F_q}\,,
$$
defined in~\cite[\S\,3, Def.\,2]{P1}. It follows from~\cite[\S\,2, Prop.\,2]{P1} and~\cite[\S\,3, Prop.\,2]{P1} that the latter homomorphism coincides with the Contou-Carr\`ere symbol ${CC_n\colon K^M_{n+1}\big(\LL^n(\F_q)\big)\to \F_q^*}$ over the field $\F_q$.

The homomorphism $K^M_n(K)\to \Gal(K^{ASW}/K)$ is induced by a collection of bilinear pairings $K_n^M(K)\times W_{p^r}(K)\to \z/p^r\Z$, $r\geqslant 1$. For each $r\geqslant 1$, the latter pairing is the composition of a bilinear pairing
\begin{equation}\label{eq:Witt}
K_{n}^M(K)\times W_{p^r}(K)\lrto W_{p^r}(\F_q)
\end{equation}
constructed explicitly in~\cite[\S\,3, Def.\,5]{P1} and the trace map $W_{p^r}(\F_q)\to W_{p^r}(\F_p)\simeq\Z/p^r\Z$. The pairing~\eqref{eq:Witt} is called a Witt pairing and was defined by Witt~\cite{Witt} for the case~$n=1$. We will show in Subsection~\ref{ASWp} (see Proposition~\ref{prop-Witt} and a discussion after it) how to construct the Witt pairing with the help of the Contou-Carr\`ere symbol $CC_n$. Notice that for this one needs to consider the Contou-Carr\`ere symbol not over a field, but over Artinian rings with non-trivial nilpotents. This relation between the Witt pairing and the Contou-Carr\`ere symbol was discovered in~\cite[\S\,4.3]{AP} for the case~$n=1$. The case $n=2$ was considered in~\cite[\S\,8]{OZ1}. Actually, a relation between the Contou-Carr\`ere symbol and the local class field theory was the main motivation for the construction of this symbol in~\cite{CC1} in the case $n=1$.

\subsection{Generalized Witt pairing}\label{ASWp}

First let us recall some basic facts on Witt vectors (see, e.g.,~\cite[Lect.\,26]{M}). Let $A$ be a ring and let $S$ be a set of positive integers such that $S$ is closed under taking divisors. The set of Witt vectors $W_S(A)$ is defined by the formula
$$
W_S(A):=\big\{w=(w_i)\,\mid\, i \in S,\;w_i\in A \big\}\,.
$$
Ghost (or auxiliary) coordinates $w(i)$ of a Witt vector $w \in W_S(A)$ are defined by the following polynomials in $w_i$ with integral coefficients:
$$
w(i):= \sum_{d  \mid i  }  dw_d^{i/d}\,,\qquad i \in S \,.
$$
Note that, conversely, the coordinates $w_i$ are expressed as polynomials in the ghost coordinates $w(i)$ with rational coefficients that have non-trivial denominators.

By definition, addition and multiplication in the ghost coordinates $w(i)$ of Witt vectors are coordinate-wise. Surprisingly, the corresponding addition and multiplication in the coordinates $w_i$ are given by polynomials in $w_i$ with {\it integral} coefficients (cf. Subsection~\ref{integr-expl-form}). Thus $W_S$ is a ring functor, that is, a functor from the category of rings to the category of rings. For short, put $W(A):=W_{\N}(A)$. For a prime $p$ and an integer $r\geqslant 1$, the ring $W_{p^r}(A):=W_{\{1,p,\,\ldots,\,p^{r-1}\}}(A)$ is called the ring of $r$-truncated $p$-Witt vectors.

\begin{rmk}\label{rmk:projWitt}
Let $S'\subset S$ be an embedding of subsets of $\N$ and $S$ and $S'$ are closed under taking divisors. Then the natural projection $W_{S}\to W_{S'}$ is a morphism of ring functors. In the ghost coordinates this homomorphism is also the natural projection.
\end{rmk}

\medskip

We will use only the additive structure on Witt vectors, thus we will consider $W_S$ as a group functor (see Section~\ref{sect:notation}). Note that if $S=\N$ and $A$ is a $\Q$-algebra, then for any Witt vector $w\in W(A)$, there is an equality in $A[[x]]$
\begin{equation}\label{ghost-series}
-\log\prod_{i\geqslant 1}(1-w_i x^i)=\sum_{i\geqslant 1}w(i)\frac{x^i}{i}\,.
\end{equation}
It follows that for any ring $A$ (not only for a $\Q$-algebra), there is a functorial isomorphism of groups
$$
\Upsilon\;:\;W(A)\stackrel{\sim}\lrto 1+xA[[x]]\,,\qquad w=(w_i)\longmapsto \prod_{i\geqslant 1}(1-w_ix^i)\,,
$$
where the group structure on $1+xA[[x]]$ is given by the product of power series. In particular, there is a functorial embedding of groups $W(A)\hookrightarrow A[[x]]^*$.

\medskip

Consider now the composition of maps
$$
K^M_n\big(\LL^n(A)\big)\times W\big(\LL^n(A)\big)\lrto K^M_n\big(\LL^n(A)[[x]]\big)\times \LL^n(A)[[x]]^*\stackrel{\sim}\lrto
$$
$$
\stackrel{\sim}\lrto \LL^n(A)[[x]]^*\times K^M_n\big(\LL^n(A)[[x]]\big)\lrto K^M_{n+1}\big(\LL^n(A)[[x]]\big)\lrto A[[x]]^*\,,
$$
where the first map is induced by the natural homomorphism of rings ${\LL^n(A)\to \LL^n(A)[[x]]}$ and the isomorphism of groups $\Upsilon\colon W\big(\LL^n(A)\big)\stackrel{\sim}\lrto 1+x\LL^n(A)[[x]]$, the second map is the transposition, the third map is given by the product between Milnor $K$-groups, and the last map is the completed Contou-Carr\`ere symbol $\widehat{CC}_n$ (see Definition~\ref{defin:compCC} and Example~\ref{examp:CCcompl}(i)). Lemma~\ref{lemma-trivkernel} implies that the image of this composition is contained in the subgroup $1+xA[[x]]\simeq W(A)$ in $A[[x]]^*$. Thus we obtain a bilinear pairing
$$
K^M_n\big(\LL^n(A)\big)\times W\big(\LL^n(A)\big)\lrto  W(A)\,,
$$
which is functorial with respect to the ring $A$. Let $\{f_1,\ldots,f_n\}$, $f_i\in \LL^n(A)^*$, be a symbol in $K^M_n\big(\LL^n(A)\big)$ and let $(g_1,g_2,\ldots)$, $g_i\in\LL^n(A)$, be a Witt vector in $W\big(\LL^n(A)\big)$. Following Witt~\cite{Witt}, we denote the pairing of $\{f_1,\ldots,f_n\}$ and $(g_1,g_2,\ldots)$ by $(f_1,\ldots,f_n\mid g_1,g_2,\ldots]$.

\medskip

\begin{lemma}\label{Prop-Witt-explicit}
If $A$ is a $\Q$-algebra, then for any symbol $\{f_1, \ldots, f_n \}$ in $K^M_n\big(\LL^n(A)\big)$, any Witt vector $(g_1,g_2,\ldots)$ in $W\big(\LL^n(A)\big)$, and any $i\in \N$, there is an equality in $A$ between the $i$-th ghost coordinate
\begin{equation}  \label{Witt-ghost}
(f_1, \ldots, f_n \mid g_1,g_2,\ldots](i)= \res\left(g(i)\,\frac{df_1}{f_1}\wedge\ldots\wedge\frac{df_n}{f_n}\right)\,.
\end{equation}
\end{lemma}
\begin{proof}
Using Proposition~\ref{examp:VS} and formula~\eqref{ghost-series}, we obtain the equalities
\begin{multline*}
\hspace{-0.5cm}
-\log\,\Upsilon(f_1,\ldots,f_n\mid g_1,g_2,\ldots]=
-\log\,\widehat{CC}_n \Big\{\Upsilon(g_1,g_2,\ldots),f_1,\ldots,f_n\Big\}=\\=
-\res\Big(\log\Big(\,\mbox{$\prod\limits_{i\geqslant 1}(1-g_i x^i)$}\Big)\,\frac{df_1}{f_1}\wedge\ldots\wedge\frac{df_n}{f_n}\Big)=
\res\Big(\Big(\,\mbox{$\sum\limits_{i\geqslant 1}g(i)$}\frac{x^i}{i}\Big)\,\frac{df_1}{f_1}\wedge\ldots\wedge\frac{df_n}{f_n}\Big)=\\=
\sum_{i\geqslant 1}\res\,\Big(g(i)\,\frac{df_1}{f_1}\wedge\ldots\wedge\frac{df_n}{f_n}\Big)\frac{x^i}{i}\,.
\end{multline*}
\end{proof}

\medskip

Now let $S$ be a subset in $\N$ closed under taking divisors. Define a group functor $U_S$ by the formula
$$
U_S:=\Ker\big(W\lrto W_S\big)\,.
$$

\begin{prop}\label{prop-Witt}
There is a bilinear morphism of functors
\begin{equation}\label{eq:genWitt}
L^nK^M_n\times L^nW_S\lrto W_S
\end{equation}
given in the ghost coordinates by formula~\eqref{Witt-ghost}.
\end{prop}
\begin{proof}
It is enough to prove that for any ring $A$, any symbol $\{f_1, \ldots, f_n \}$ in  $K^M_n\big(\LL^n(A)\big)$, and any Witt vector $(g_1,g_2,\ldots)$ in $U_S\big(\LL^n(A)\big)$, the pairing ${(f_1,\ldots,f_n\mid g_1,g_2\ldots]}$ belongs to the subgroup $U_S(A)$. If $A$ is a $\Q$-algebra, then this follows immediately from Remark~\ref{rmk:projWitt} and Lemma~\ref{Prop-Witt-explicit}.

In general, we need to show that the composition of morphisms of functors
$$
(L^n\gm)^{\times n}\times L^nU_S\lrto L^nK^M_n\times L^nW\lrto W\lrto W_S
$$
is equal to zero. Clearly, the functor $U_S$ is represented by the affine space $\Ab^{\N\smallsetminus S}$. Hence by Proposition~\ref{lemma-repraffine}$(i)$, $L^nU_S$ is represented by an ind-affine space, which is an ind-flat ind-affine scheme over $\z$. Thus we conclude by Theorem~\ref{cor:uniq} applied to the embedding of rings $\Z\subset\Q$.
\end{proof}

\begin{rmk}\label{rmk:pairingWitt}
Let $S'\subset S$ be an embedding of subsets in $\N$ such that $S$ and $S'$ are closed under taking divisors. Then pairings~\eqref{eq:genWitt} for $S$ and $S'$ commute with each other under the natural projections $W_{S}\to W_{S'}$ and $L^nW_{S}\to L^nW_{S'}$ (see Remark~\ref{rmk:projWitt}). This follows directly from the proof of Proposition~\ref{prop-Witt}.
\end{rmk}

When $S=\{1,\ldots,p^{r-1}\}$, $r\geqslant 1$, pairing~\eqref{eq:genWitt} evaluated at $A=\F_q$ is equal to the pairing from~\cite[\S\,3, Def.\,5]{P1} (this follows from functoriality of pairing~\eqref{eq:genWitt}).

\quash{
Also, combining Theorem~\ref{theor:integalexpl} with Proposition~\ref{prop-Witt}, we obtain that pairing~\eqref{eq:genWitt} is given by some universal polynomials with integral coefficients. The variables in these polynomials are coefficients of the iterated Laurent series $f_i$ and $g_j$ as above. This is a generalization of~\cite[\S\,3, Lem.\,1]{P1}.

Consider a ring $A$, a symbol $\{f_1, \ldots, f_n \}  \in  K^M_n\big(\LL^n(A)\big)$, and a Witt vector $(g_1,g_2,\ldots)\in W_S\big(\LL^n(A)\big)$. Let $f_i=\mbox{$\sum\limits_{l\in\z^n}a_{i,l}t^l$}$, $1\leqslant i\leqslant n$,
$g_j=\mbox{$\sum\limits_{l\in\z^n}b_{j,l}t^l$}$, $j\geqslant 1$, and let
$(f_1,\ldots,f_n\mid g_1,g_2,\ldots]=(w_1,w_2,\ldots)$. Then $w_i\in A$ are given by universal polynomials with integral coefficients in $a_{i,l}$, $a_{i,0}^{-1}$, and $b_{j,l}$, where $1\leqslant i\leqslant n$, $j\geqslant 1$, $l\in\z^n$.
}

\end{document}